%% file: blob_tilting.tex
\documentclass[a4paper]{amsart}

\usepackage{amsmath}
\usepackage{mathtools}
\usepackage{amssymb} 
\usepackage{enumitem}
\usepackage{hyperref}
\usepackage[all]{xy}
\UseComputerModernTips
\CompileMatrices

\usepackage{tikz}
\usetikzlibrary{matrix,arrows}
\usepackage{tikz-cd}

\usepackage{graphics}

\usepackage{todonotes}

\newcommand{\defnemph}[1]{\emph{#1}}
\newcommand{\field}{\Bbbk}
\newcommand{\N}{\mathbb{N}}
\newcommand{\Z}{\mathbb{Z}}

\newcommand{\iso}{\cong}
\newcommand{\W}{W}
\newcommand{\expr}[1]{\underline{#1}}
\newcommand{\len}{\ell}
\newcommand{\weyl}{\Delta}
\newcommand{\dweyl}{\nabla}
\newcommand{\wt}[1]{{\boldsymbol{#1}}}
\newcommand{\res}[1]{\boldsymbol{#1}}
\newcommand{\domres}[1]{\res{i}^{#1}}
\newcommand{\tableau}[1]{\mathfrak{#1}}
\newcommand{\domtab}[1]{\tableau{t}^{#1}}
\newcommand{\domleq}{\trianglelefteq}
\newcommand{\domless}{\vartriangleleft}
\newcommand{\domgreater}{\vartriangleright}
\newcommand{\modcat}[1]{{#1}\mathrm{-mod}}
\newcommand{\JW}{\mathrm{JW}}
\newcommand{\subqtab}[2]{\tableau{t}^{#1}_{#2}}

\DeclareMathOperator{\Std}{Std}

\DeclareMathOperator{\rad}{rad}
\DeclareMathOperator{\soc}{soc}
\DeclareMathOperator{\ind}{ind}
\DeclareMathOperator{\restr}{res}
\DeclareMathOperator{\pr}{pr}
\DeclareMathOperator{\Hom}{Hom}
\DeclareMathOperator{\End}{End}
\DeclareMathOperator{\Ext}{Ext}

\theoremstyle{plain}
\newtheorem{thm}{Theorem}[section]
\newtheorem*{thm*}{Theorem}
\newtheorem{prop}[thm]{Proposition}
\newtheorem{cor}[thm]{Corollary}

\newtheorem{lem}[thm]{Lemma}

\newtheorem*{conj*}{Conjecture}
\theoremstyle{definition}
\newtheorem{defn}[thm]{Definition}
\newtheorem{rem}[thm]{Remark}
\newtheorem{eg}[thm]{Example}
\newtheorem*{rem*}{Remark}

\begin{document}


\title{Indecomposable tilting modules for the blob algebra}
\author{A.~Hazi \and P.~P.~Martin \and A.~E.~Parker}
\address{Department of Mathematics \\ University of Leeds \\ Leeds, LS2 9JT \\ UK}
\email{amit.hazi@cantab.net}
\email{P.P.Martin@maths.leeds.ac.uk}
\email{A.E.Parker@leeds.ac.uk}

\begin{abstract}
The blob algebra is a finite-dimensional quotient of
the Hecke algebra of type $B$ which is almost always
quasi-hereditary. We construct the indecomposable tilting
modules for the blob algebra over a field of characteristic $0$ in the
doubly critical case. Every indecomposable tilting module of maximal
highest weight is either a projective module or an extension of a
simple module by a projective module. Moreover, every indecomposable
tilting module is a submodule of an indecomposable tilting module of
maximal highest weight. We conclude that the graded Weyl
multiplicities of the indecomposable tilting modules in this case are
given by inverse Kazhdan--Lusztig polynomials of type $\tilde{A}_1$. 

\smallskip
\noindent \textbf{\keywordsname :} blob algebra, tilting modules, KLR algebra, Soergel bimodules. 

\end{abstract}

\maketitle


\section*{Introduction}


\input{intro.tex}

\section*{Acknowledgements}

We thank EPSRC for  financial support  under grant EP/L001152/1.

\section{Preliminaries: the blob algebra \texorpdfstring{$B_n^\kappa$}{B_n^k}}
\label{sec:prelim}


%
%


Suppose $e>1$ is an integer and let $I=\Z/e\Z$. An \defnemph{adjacency-free bicharge} is an ordered pair $\kappa=(\kappa_1,\kappa_2) \in I^2$ such that $\kappa_1 \neq \kappa_2,\kappa_2 \pm 1$ (this implicitly requires $e \geq 4$). For $i \in I$ define
\begin{equation*}
\langle i|\kappa \rangle=\begin{cases}
1 & \text{if $i=\kappa_1$ or $i=\kappa_2$,} \\
0 & \text{otherwise.}
\end{cases}
\end{equation*}
For any $n \in \N$, the symmetric group $S_n$ acts on the set of tuples $I^n$ by permutation. We write $s_r$ for the simple transposition $(r \; r\!+\!1)$ in the symmetric group $S_n$.

\begin{defn} \label{de:b1}
  Let $\field$ be a field, $n,e \in \N$,
  and $\kappa$ be an adjacency-free bicharge.
  The (doubly critical) \defnemph{blob algebra} $B_n^{\kappa}$ over $\field$ is the $\Z$-graded
  $\field$-algebra generated by
\begin{align}
\psi_r & & & \text{for $1 \leq r \leq n-1$,} \\
y_r & & & \text{for $1 \leq r \leq n$,} \\
e(\res{i}) & &  & \text{for $\res{i} \in I^n$,}
\end{align}
subject to relations
\begin{align}
e(\res{i})e(\res{j}) & =\delta_{\res{i},\res{j}}e(\res{i}) & & \text{for all $\res{i},\res{j} \in I^n$} \\
\sum_{\res{i} \in I^n} e(\res{i}) & =1 & & \\
y_r e(\res{i}) & =e(\res{i}) y_r & & \\
\psi_r e(\res{i}) & =e(s_r \res{i}) \psi_r & & \\
y_r y_s & =y_s y_r & & \\
\psi_r y_s & =y_s \psi_r & & \text{when $s \neq r,r+1$} \\
\psi_r \psi_s & =\psi_s \psi_r & & \text{when $|r-s|>1$} \\
\label{eq:yslide1}
\psi_r y_{r+1} e(\res{i}) & =(y_r \psi_r-\delta_{i_r,i_{r+1}})e(\res{i}) & & \\
\label{eq:yslide2}
y_{r+1} \psi_r e(\res{i}) & =(\psi_r y_r-\delta_{i_r,i_{r+1}})e(\res{i}) & & \\
\label{eq:psisq}
\psi_r^2 e(\res{i}) & =\begin{cases}
e(\res{i}) & \text{if $i_{r+1} \neq i_r,i_r \pm 1$} \\
0 & \text{if $i_{r+1}=i_r$} \\
(y_{r+1}-y_r)e(\res{i}) & \text{if $i_{r+1}=i_r+1$} \\
(y_r-y_{r+1})e(\res{i}) & \text{if $i_{r+1}=i_r-1$}
\end{cases} & & \\
\psi_r \psi_{r+1} \psi_r e(\res{i}) & =\begin{cases}
(\psi_{r+1} \psi_r \psi_{r+1} - 1)e(\res{i}) & \text{if $i_{r+2}=i_r=i_{r+1}-1$} \\
(\psi_{r+1} \psi_r \psi_{r+1} + 1)e(\res{i}) & \text{if $i_{r+2}=i_r=i_{r+1}+1$} \\
\psi_{r+1} \psi_r \psi_{r+1} e(\res{i}) & \text{otherwise}
\end{cases} & & \\
\label{eq:kinpres}
y_1^{\langle i_1 | \kappa \rangle} e(\res{i}) & =0 & & \\
\label{eq:blobdefiningreln}
e(\res{i}) & =0 & & \text{when $i_2=i_1+1$}
\end{align}
and a grading defined by
\begin{align*}
\deg e(\res{i}) & =0 \text{,} & \deg y_r e(\res{i}) & =2 \text{,} & \deg \psi_r e(\res{i}) & =\begin{cases}
1 & \text{if $i_{r+1}=i_r \pm 1$,} \\
-2 & \text{if $i_{r+1}=i_r$,} \\
0 & \text{otherwise.}
\end{cases}
\end{align*}
\end{defn}


In the presentation in Definition~\ref{de:b1},
each $e(\res{i})$ is a (non-central) idempotent,
each $\psi_r$ is analogous to the simple transposition
$s_r$
in the symmetric group $S_n$, 
and each $y_r$ is akin to the nilpotent part of the corresponding Jucys--Murphy element in the symmetric group algebra $\field S_n$.

There is also a presentation of this algebra in terms of \defnemph{KLR diagrams} \cite[\S~3.2]{libedinskyplaza}. A KLR diagram with $n$ strings consists of $n$ paths of the form $p:[0,1] \rightarrow \mathbb{R} \times [0,1]$ satisfying the following properties:
\begin{itemize}
\item for each path $p$ we have $p(0)=(x,0)$ and $p_r(1)=(x',1)$ for some $x,x' \in \mathbb{R}$;

\item all intersections are transversal;

\item there are no triple intersections;

\item each path may be decorated with a finite number of dots at non-intersection points.
\end{itemize}
Each path $p$ is also labelled with a residue $i \in I$. 

We consider KLR diagrams up to isotopy; in other words, we are allowed to move these paths continuously as long as the properties above still hold and no intersections are added or removed. The \defnemph{bottom} (resp.~\defnemph{top}) of a KLR diagram is the sequence of residues labelling the paths, ordered by the relation $p \prec p'$ if $p(0)<p'(0)$ (resp.~$p \prec p'$ if $p(1)<p'(1)$). The product of two diagrams $D$ and $D'$ is defined to be their vertical concatenation (with $D$ on top of $D'$) whenever the bottom of $D$ equals the top of $D'$. Otherwise the product is defined to be $0$. The diagrammatic blob algebra $B_n^{\kappa}$ is then the set of all $\field$-linear combinations of KLR diagrams with $n$ strings, with a diagrammatic product defined by $\field$-linear extension, subject to the following relations:
\begin{align*}
\raisebox{-0.6cm}{%
\includegraphics[scale=0.75]{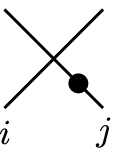}} & =
\raisebox{-0.6cm}{%
\includegraphics[scale=0.75]{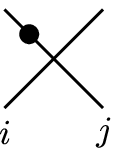}}-\delta_{ij}\;
\raisebox{-0.6cm}{%
\includegraphics[scale=0.75]{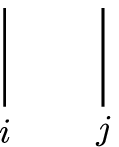}} \\
\raisebox{-0.6cm}{%
\includegraphics[scale=0.75]{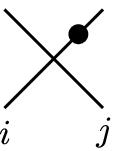}} & =
\raisebox{-0.6cm}{%
\includegraphics[scale=0.75]{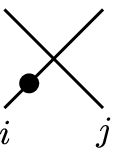}}-\delta_{ij}\;
\raisebox{-0.6cm}{%
\includegraphics[scale=0.75]{fig/twostraight.eps}} \\
\raisebox{-1.0cm}{%
\includegraphics[scale=0.75]{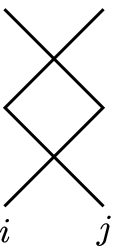}} & =
\begin{cases}
\raisebox{-1.0cm}{%
\includegraphics[scale=0.75]{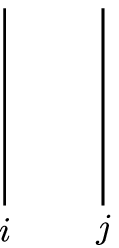}} & \text{if $|i-j|>1$,} \\
\raisebox{-1.0cm}{%
\includegraphics[scale=0.75]{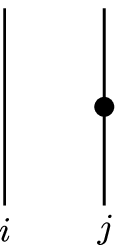}}\;-\;
\raisebox{-1.0cm}{%
\includegraphics[scale=0.75]{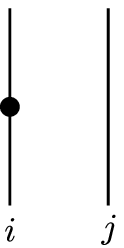}} & \text{if $j=i+1$,} \\
\raisebox{-1.0cm}{%
\includegraphics[scale=0.75]{fig/twostraightleftdot.eps}}\;-\;
\raisebox{-1.0cm}{%
\includegraphics[scale=0.75]{fig/twostraightrightdot.eps}} & \text{if $j=i-1$,} \\
0 & \text{if $i=j$.}
\end{cases} \\
\raisebox{-1.25cm}{%
\includegraphics[scale=0.75]{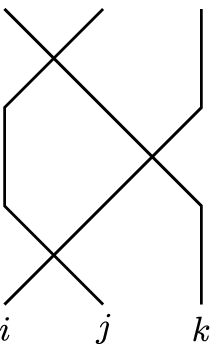}} & =
\raisebox{-1.25cm}{%
\includegraphics[scale=0.75]{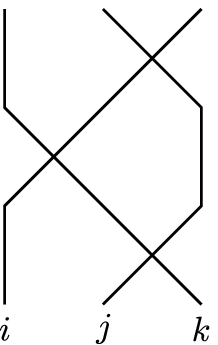}}+\alpha\;
\raisebox{-1.25cm}{%
\includegraphics[scale=0.75]{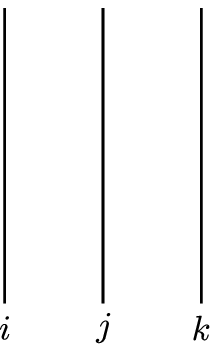}}
\end{align*}
in all regions of a KLR diagram, where $\alpha=1$ when $i=k=j-1$, $\alpha=-1$ when $i=k=j+1$, and $\alpha=0$ otherwise, as well as the relations
\begin{align*}
\raisebox{-1.0cm}{%
\includegraphics[scale=0.75]{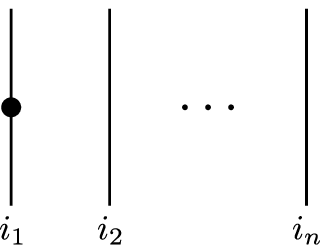}} & =0\text{,} & & \text{if $i_1=\kappa_j$ for some $j$,} \\
\raisebox{-1.0cm}{%
\includegraphics[scale=0.75]{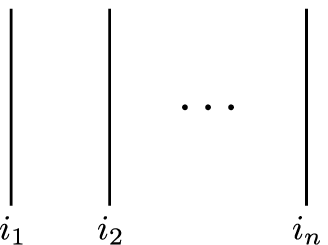}} & =0\text{,} & & \text{if $i_1 \neq \kappa_j$ for all $j$,} \\
\raisebox{-1.0cm}{%
\includegraphics[scale=0.75]{fig/cyclonodot.eps}} & =0\text{,} & & \text{if $i_2=i_1+1$.}
\end{align*}

If $\expr{w}=\expr{s_{r_1} s_{r_2} \dotsm s_{r_k}}$ is a reduced expression in $S_n$, we write $\psi_{\expr{w}}=\psi_{r_1} \psi_{r_2} \dotsm \psi_{r_k}$ for the product of the corresponding $\psi$-generators. Diagrammatically $\psi_{\expr{w}}$ (or more precisely, $\psi_{\expr{w}}e(\res{i})$ for some $\res{i} \in I^n$) looks like the wiring diagram for $\expr{w}$. We also write $(\overline{\phantom{\psi}})$ for the unique anti-involution which fixes each of the generators $\psi_r$, $y_r$, and $e(\res{i})$.

\subsection{Locality}
\label{sec:locality}

We call a relation in the generators of $B_n^{\kappa}$ \defnemph{local} if the relation still holds when the indices of the generators are shifted by some amount. All the relations in Definition \ref{de:b1} above are local except for \eqref{eq:kinpres} and \eqref{eq:blobdefiningreln}. The relation \eqref{eq:kinpres} is also the only one in which $\kappa$ appears. Incidentally it is immediately clear that all other relations do not depend on precise values of sequences $\res{i} \in I^n$ indexing the idempotents, but only on relative differences $i_{r+1}-i_r$ for some integer $1 \leq r \leq n$. In fact for any $i \in I$, if $\kappa'=(\kappa_1+i,\kappa_2+i)$ then we have $B_n^{\kappa} \iso B_n^{\kappa'}$, and this isomorphism maps $e(\res{i}) \mapsto e(\res{i}+(i,\dotsc,i))$. Thus $B_n^{\kappa}$ only depends on the difference $\kappa_1-\kappa_2 \in I$ up to isomorphism.

When simplifying KLR diagrams we adopt the convention of circling regions in some colour wherever we apply a local relation only involving $\psi$-generators. These circles are only a helpful annotation and should \emph{not} be considered an intrinsic part of the diagram. Similarly whenever we apply relations \eqref{eq:yslide1} or \eqref{eq:yslide2} in the distinct residue case, we will draw a coloured arrow parallel to the string to indicate how the $y$-generator `slides' along the string. The most important non-local relation which we will use takes the following form.

\begin{lem} \label{lem:yjump}
Let $\res{i} \in I^n$ and $1 \leq r \leq n-1$ be an integer such that $|i_r-i_{r+1}|=1$ but $e(s_r \res{i})=0$ in $B_n^{\kappa}$. Then $y_{r+1} e(\res{i})=y_r e(\res{i})$. 
\end{lem}

\begin{proof}
Apply \eqref{eq:psisq} to obtain
\begin{equation*}
y_{r+1} e(\res{i})=y_r e(\res{i}) \pm \psi_r^2 e(\res{i})=y_r e(\res{i}) \pm \psi_r e(s_r \res{i}) \psi_r=y_r e(\res{i}) \text{.}
\end{equation*}
\end{proof}

When applying Lemma \ref{lem:yjump} to a KLR diagram, we will draw a dashed coloured line transverse to the strings to indicate which idempotent $e(\res{i})$ we are using, and a coloured arrow to show where the $y$-generator `jumps' to a different string.

\subsection{The classical blob algebra}

Definition \ref{de:b1} presents the blob algebra as a quotient of a cyclotomic KLR algebra as in \cite{plaza-ryom-hansen}, with the same generators and all the same relations plus the extra relation \eqref{eq:blobdefiningreln}. This does not correspond to the original definition of the blob algebra in \cite{martin-saleur} as an extension of the Temperley--Lieb algebra. However, our definition is equivalent in many cases due to the Brundan--Kleshchev isomorphism \cite[Theorem 1.1]{brundan-kleshchev} between cyclotomic KLR algebras and cyclotomic Hecke algebras. 

\begin{thm}[{\cite[Corollary 3.6]{plaza-ryom-hansen}}] \label{thm:KLR-classicblob-equiv}
Suppose $e>1$ is an integer which is not a multiple of the characteristic of $\field$. Let $m$ be an integer with $1<m<e-1$. 
Set
  $\kappa=(0,m)$, an adjacency-free bicharge.
  Then $B_n^{\kappa}$ has a presentation as an ungraded algebra over $\field$, with generators $U_r$ for $0 \leq r \leq n-1$ subject to the following relations:
\begin{align*}
U_r^2& = -[2]U_r & & \text{if $1 \leq r \leq n-1$,} \\
U_r U_s U_r& =U_r & & \text{if $|r-s|=1$ and $1 \leq r,s \leq n-1$,} \\
U_r U_s& = U_s U_r & & \text{if $|r-s|>1$ and $0 \leq r,s \leq n-1$,} \\
U_1 U_0 U_1 & =[m+1] U_1 \text{,} & & \\
U_0^2& =-[m] U_0 \text{,}
\end{align*} 
where $[k]=[k]_q=q^{-k+1}+q^{-k+3}+\dotsb+q^{k-1}$, $q$ is an $e'$th primitive root of unity in $\field$, and 
\begin{equation*}
e'=\begin{cases}
2e & \text{if $e$ is even,} \\
e & \text{otherwise.}
\end{cases}
\end{equation*}
\end{thm}

\begin{rem} \hfill
\begin{enumerate}[label=(\arabic*)]
\item The statement of \cite[Corollary 3.6]{plaza-ryom-hansen} uses the bicharge $\kappa=(k,-k)$ (where $k \in I$ such that $2k \equiv m \pmod{n}$) and a `negative variant' form of \eqref{eq:blobdefiningreln}. To transform this into Theorem~\ref{thm:KLR-classicblob-equiv} it is necessary to shift the residues by $-k$ (as mentioned in \S\ref{sec:locality}) and apply the isomorphism
\begin{align*}
R_n^{(0,-m)} & \longrightarrow R_n^{(0,m)} \\
\psi_r & \longmapsto -\psi_r \\
y_r & \longmapsto -y_r \\
e(\res{i}) & \longmapsto e(-\res{i})
\end{align*}
of cyclotomic KLR algebras with bicharges $(0,-m)$ and $(0,m)$.

\item Theorem \ref{thm:KLR-classicblob-equiv} is the most general version of what is commonly stated in the literature, but it can probably be extended to other cases as well. For example, when $e$ equals the characteristic of $\field$, $B_n^{\kappa}$ behaves like the classical blob algebra over $\field$ with $q=1$. In addition, adjacency-freeness of $\kappa$ and the condition that $1<m<e-1$ can potentially be relaxed, at the cost of modifying relation \eqref{eq:blobdefiningreln} (this is similar to what happens for the Temperley--Lieb algebra \cite[Remark~3.7]{plaza-ryom-hansen}).
\end{enumerate}
\end{rem}

\subsection{Weights and multipartitions}
\label{sec:wtsbiparts}


In general the representation theory of KLR algebras
is governed by the combinatorics of multipartitions, 
while that of the blob algebra is naturally governed by the geometry of a
suitable weight lattice
\cite{martin-woodcock2}. 
To understand the blob algebra in KLR terms 
it is enough to focus on one-column bipartitions. 

A \defnemph{one-column bipartition} of $n$ is an ordered pair $\wt{\lambda}=(1^{\lambda_1},1^{\lambda_2})$ with $\lambda_1,\lambda_2 \in \Z_{\geq 0}$ and $\lambda_1+\lambda_2=n$. We write $\Lambda(n)$ for the set of all one-column bipartitions of $n$. The mapping
\begin{align*}
\Lambda(n) & \longrightarrow \{-n,-n+2,\dotsc,n-2,n\} \\
\wt{\lambda} & \longmapsto \lambda_1-\lambda_2
\end{align*}
is a bijection between one-column bipartitions and the classical
weight set for the blob algebra. For this reason we will usually call
one-column bipartitions \defnemph{weights} when working in a
representation-theoretic context.
For two weights $\wt{\lambda},\wt{\mu} \in \Lambda(n)$ we write $\wt{\lambda} \domless \wt{\mu}$ (and say $\wt{\mu}$ {dominates} $\wt{\lambda}$) if $|\lambda_1-\lambda_2|>|\mu_1-\mu_2|$ (following \cite{martin-woodcock}). 


The \defnemph{Young diagram} for $\wt{\lambda} \in \Lambda(n)$ is defined to be the set
\begin{equation*}
[\wt{\lambda}]= \{(r,1) : 1 \leq r \leq \lambda_1\} \cup \{(r,2) : 1 \leq r \leq \lambda_2\}
\end{equation*}
Elements of this set are usually called \defnemph{boxes}, because the traditional way to depict Young diagrams is as a collection of boxes, e.g.
\[
[( 1^4 , 1^5 )]  \; = \; 
\left( \;
\raisebox{-.1in}{\includegraphics[width=.242cm]{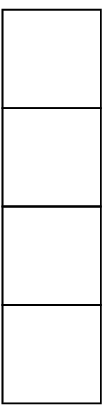}} \; , \;\;
\raisebox{-.183in}{\includegraphics[width=.242cm]{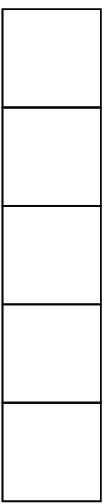}}
\; \right)
\]
A \defnemph{tableau} of shape $\wt{\lambda}$ is a bijection $[\wt{\lambda}] \rightarrow \{1,2,\dotsc,n\}$, which is usually depicted by writing each assignment inside the corresponding box, e.g.
\[
\left( \;
\raisebox{-.1in}{\includegraphics[width=.242cm]{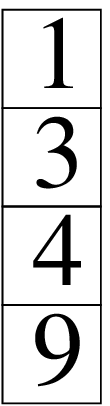}} \; , \;\;
\raisebox{-.183in}{\includegraphics[width=.242cm]{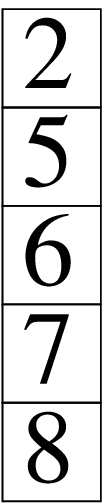}}
\; \right)
\]
A tableau is called \defnemph{standard} if the entries in the boxes increase going down each column. A standard tableau $\tableau{t}$ corresponds in a natural way to a sequence $\tableau{t}|_k \in \Lambda(k)$ of Young diagrams obtained by adding exactly one box at each stage. Such sequences are in bijection with paths of length $n$ on the global lattice of weights $\Z$, where a path is just a function $p:\{0,1,2,\dotsc,n\} \rightarrow \Z$ with $p(0)=0$ and $p(k+1)-p(k)=\pm 1$ for all integers $0 \leq k \leq n-1$. Adding a box in the first column corresponds to a rightward ($+1$) step and vice versa.

We write $\domtab{\wt{\lambda}}$ for the standard tableau of shape $\wt{\lambda}$ obtained by labelling the boxes of $[\wt{\lambda}]$ with increasing entries ordered from left to right and from top to bottom like a book, e.g.
\[
\domtab{(1^4 , 1^5 )} = 
\left( \;
\raisebox{-.1in}{\includegraphics[width=.242cm]{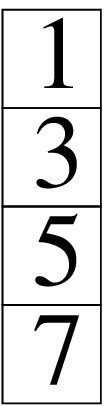}} \; , \;\;
\raisebox{-.183in}{\includegraphics[width=.242cm]{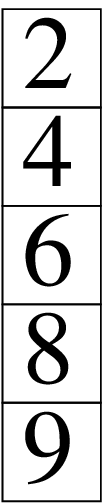}}
\; \right)
\]

The ($\kappa$-)\defnemph{residue} of a box with coordinates $(r,m)$ is defined to be $\kappa_m+1-r \in I$. The \defnemph{residue sequence} $\res{i}^{\tableau{t}}$ of a tableau $\tableau{t}$ is the sequence of residues of the boxes $(\tableau{t}^{-1}(1),\tableau{t}^{-1}(2),\dotsc,\tableau{t}^{-1}(n))$. We write $\domres{\wt{\lambda}}$ instead of $\res{i}^{\domtab{\wt{\lambda}}}$ for the residue sequence of the dominant tableau $\domtab{\wt{\lambda}}$.


\section{Cellularity of $B_n^{\kappa}$}
\label{sec:cellularity}

Suppose $\tableau{t}$ is a standard tableau of shape $\wt{\lambda}$. Let $d_{\tableau{t}} \in S_n$ be the permutation such that $d_{\tableau{t}} \domtab{\wt{\lambda}}=\tableau{t}$. 

\begin{thm}[{\cite[Theorem 6.8]{plaza-ryom-hansen}}] \label{thm:cellbasis}
Fix a reduced expression $\expr{d}_{\tableau{t}}$ for each $d_{\tableau{t}}$ over all $\wt{\lambda} \in \Lambda(n)$ and $\tableau{t} \in \Std(\wt{\lambda})$. The elements
\begin{equation*}
\psi_{\tableau{s} \tableau{t}}=\psi_{\expr{d}_{\tableau{s}}} e(\domres{\wt{\lambda}})\psi_{\expr{d}_{\tableau{t}}^{-1}} \in B_n^{\kappa}
\end{equation*}
over all $\wt{\lambda} \in \Lambda(n)$ and all $\tableau{s},\tableau{t} \in \Std(\wt{\lambda})$ form a graded cellular basis for $B_n^{\kappa}$ with respect to the partial order $\domleq$ on weights and the anti-involution $\psi \mapsto \overline{\psi}$.
In other words
\end{thm}

For the precise definition of a graded cellular basis see \cite[Definition 2.1]{hu-mathas}.
An important corollary, especially in conjunction with Lemma \ref{lem:yjump}, is the following.

\begin{cor}
Let $\res{i} \in I^n$. If there is no standard tableau $\tableau{t}$ with ($\kappa$-)residue $\res{i}$, then $e(\res{i})=0$ in $B_n^{\kappa}$.
\end{cor}

\begin{rem} \hfill
\begin{enumerate}[label=(\arabic*)]
\item The degree of $\psi_{\tableau{s} \tableau{t}}$ does not depend on the choices of $\expr{d}_{\tableau{s}}$ or $\expr{d}_{\tableau{t}}$, and has a combinatorial definition based on $\tableau{s}$ and $\tableau{t}$ (see Theorem~\ref{thm:degtab} below).


\item The graded cellular structure on $B_n^{\kappa}$ is in fact graded quasi-hereditary, which we will use frequently from now on. The idempotent-truncated algebras $e(\domres{\lambda})B_n^{\kappa}e(\domres{\wt{\lambda}})$, studied extensively in \cite{plaza-grdecompblob,libedinskyplaza} are also graded cellular but are not quasi-hereditary.
\end{enumerate}
\end{rem}

\subsection{Graded cellular and quasi-hereditary algebras} 


We fix some notation for graded modules. If $M=\bigoplus_j M^j$ is a graded vector space, we define the grade shift $M\langle k\rangle$ for $k \in \Z$ by $M\langle k\rangle^j=M^{j-k}$. For $M,N$ graded $B_n^{\kappa}$-modules, we call a degree-preserving homomorphism $M \rightarrow N$ \defnemph{homogeneous} of degree $0$. When we write $\Hom_{B_n^{\kappa}}(M,N)$ we always mean the space of \emph{ungraded} homomorphisms. By convention any homomorphism we write with a grade shifted object is homogeneous of degree $0$, but homomorphisms without grade shifts may be ungraded.

We recall some facts about graded cellular and quasi-hereditary algebras \cite{hu-mathas}. Let $\wt{\lambda} \in \Lambda(n)$, and write $B_n^{\kappa,\domgreater \wt{\lambda}}$ for the subspace spanned by all basis elements indexed by standard tableaux for weights $\wt{\mu} \domgreater \wt{\lambda}$. Cellularity essentially means that for any standard tableaux $\tableau{s},\tableau{t} \in \Std(\wt{\lambda})$, we can write the action of $B_n^{\kappa}$ on the basis vector $\psi_{\tableau{s},\tableau{t}}$ modulo the subspace $B_n^{\kappa,\domgreater \wt{\lambda}}$ as
\begin{equation*}
a \psi_{\tableau{s} \tableau{t}}=\sum_{\tableau{v} \in \Std(\wt{\lambda})} r_{\tableau{s}\tableau{v}}(a) \psi_{\tableau{v} \tableau{t}} \pmod{B_n^{\kappa, \domgreater \wt{\lambda}}}
\end{equation*}
where the scalars $r_{\tableau{s}\tableau{v}}(a)$ don't depend on $\tableau{t}$. We can use these scalars to define a module $\weyl(\wt{\lambda})$ with basis $\psi_{\tableau{s}}$ indexed by $\Std(\wt{\lambda})$, namely
\begin{equation*}
a \psi_{\tableau{s}} =\sum_{\tableau{v} \in \Std(\wt{\lambda})} r_{\tableau{s}\tableau{v}}(a) \psi_{\tableau{v}}
\end{equation*}
We call such modules \defnemph{cell modules} or \defnemph{Weyl modules}. Graded cellularity means that there is a degree function on tableaux (see Theorem~\ref{thm:degtab}) which makes the basis $\{\psi_{\tableau{s}}\}$ a homogeneous basis.

For any fixed standard tableaux $\tableau{a},\tableau{b} \in \Std(\wt{\lambda})$, we can define a contravariant bilinear form on $\weyl(\wt{\lambda})$ by 
\begin{equation*}
\psi_{\tableau{a} \tableau{s}} \psi_{\tableau{t} \tableau{b}}=\langle \psi_{\tableau{s}},\psi_{\tableau{t}} \rangle \psi_{\tableau{a} \tableau{b}} \pmod{B_n^{\kappa,\domgreater \wt{\lambda}}}
\end{equation*}
In fact this bilinear form does not depend on $\tableau{a}$ or $\tableau{b}$. For a general cellular algebra the quotient $\weyl(\wt{\lambda})/\rad \langle -,-\rangle$ is either a simple module, which we call $L(\wt{\lambda}),$ or $0$. The non-zero quotients give a complete list of non-isomorphic simple modules up to grade shift. In our case, none of the quotients are zero because $B_n^{\kappa}$ is quasi-hereditary. We write $P(\wt{\lambda})$ for the graded projective cover of $L(\wt{\lambda})$. For $M$ a graded $B_n^{\kappa}$, we define the graded composition factor multiplicities
\begin{equation*}
[M : L(\wt{\lambda})]_v=\sum_k [M : L(\wt{\lambda})\langle k \rangle] v^k \in \Z_{\geq 0}[v^{\pm 1}] \text{,}
\end{equation*}
where $[M : L(\wt{\lambda})\langle k \rangle]$ denotes the number of composition factors in a graded composition series isomorphic to $L(\wt{\lambda})\langle k \rangle$. Similarly if $M$ has a graded Weyl filtration, we define 
\begin{equation*}
(M : \weyl(\wt{\lambda}))_v=\sum_k (M : \weyl(\wt{\lambda})\langle k \rangle) v^k \in \Z_{\geq 0}[v^{\pm 1}] \text{,}
\end{equation*}
where $(M : \weyl(\wt{\lambda})\langle k\rangle)$ denotes the number of subquotients in a graded Weyl filtration isomorphic to $\weyl(\wt{\lambda})\langle k\rangle$. For the ungraded counterparts of these multiplicities we use the same notation but without the subscript $v$.


As $B_n^{\kappa}$ is quasi-hereditary, we also have the notion of a \defnemph{tilting module}. A {tilting module} for $B_n^{\kappa}$ is a module with a filtration by Weyl modules as well as a filtration by dual Weyl modules. For each weight $\wt{\lambda}$, there is an indecomposable tilting module $T(\wt{\lambda})$ of highest weight $\wt{\lambda}$, and all indecomposable tilting modules are of this form \cite{ringel}. In the graded setting this classification only gives a grading on $T(\wt{\lambda})$ up to grade shift. We will fix the grading so that $(T(\wt{\lambda}) : \weyl(\wt{\lambda}))_v=1$.


The anti-involution gives rise to a duality functor on $B_n^{\kappa}$-modules which reverses grade shift. The unshifted simple module $L(\wt{\lambda})$ is self-dual, so the dual Weyl module $\dweyl(\wt{\lambda})$ has socle isomorphic to $L(\wt{\lambda})$. Similarly the unshifted injective envelope $I(\wt{\lambda})$ is isomorphic to the dual of $P(\wt{\lambda})$. By highest weight considerations $T(\wt{\lambda})$ is self-dual. For $h \in \Z_{\geq 0}[v^{\pm 1}]$, we write $\overline{h}=h(v^{-1})$. 


\subsection{Tower of recollement}

For fixed $m,e$ and varying $n$, the family of classical blob algebras (with presentation as in Theorem~\ref{thm:KLR-classicblob-equiv}) has the structure of a \defnemph{tower of recollement} \cite[Example 1.2(ii)]{coxmartinparkerxi}. A tower of recollement consists of a collection of algebras and idempotents in these algebras which satisfy certain axioms, giving rise to several functors between module categories which pass representation-theoretic information between the algebras. Constructing the functors and verifying the axioms are both more easily accomplished in the classical presentation of the blob algebra. For this reason we will assume for the moment that Theorem~\ref{thm:KLR-classicblob-equiv} holds so that the tower of recollement structure transfers to $\{B_n^{\kappa}\}_{n \in \N}$.
For the basic definitions and some examples see \cite[Section~1]{coxmartinparkerxi}, and \cite[Section~3]{martin-ryom-hansen} for applications.

For each $n \in \N$ we have a pair of adjoint functors
\begin{align*}
\ind : \modcat{B_n^{\kappa}} & \longrightarrow \modcat{B_{n+1}^{\kappa}} \text{,} & 
\restr : \modcat{B_{n+1}^{\kappa}} & \longrightarrow \modcat{B_n^{\kappa}}
\end{align*}
called \defnemph{induction} and \defnemph{restriction} respectively. As a right adjoint functor, restriction is left exact, and similarly induction is right exact. However, restriction also happens to be right exact as well. For $\wt{\lambda} \in \Lambda(n+1)$ write $\wt{\lambda}=(1^{\lambda_1},1^{\lambda_2})$. If $\lambda_1 \geq \lambda_2>0$ we have a short exact sequence
\begin{equation*}
\xymatrix{
0 \ar[r] & \weyl(1^{\lambda_1-1},1^{\lambda_2}) \ar[r] & \restr \weyl(1^{\lambda_1},1^{\lambda_2}) \ar[r] & \weyl(1^{\lambda_1},1^{\lambda_2-1}) \ar[r] & 0}
\end{equation*}
while $\restr \weyl(1^{n+1},\emptyset)=\weyl(1^n,\emptyset)$. When $0<\lambda_1 \leq \lambda_2$ there are similar exact sequences with the two outer terms switched. Induction on Weyl modules also produces exact sequences in this way, but without a boundary exception.

We also have another pair of adjoint functors
\begin{align*}
G : \modcat{B_n^{\kappa}} & \longrightarrow \modcat{B_{n+2}^{\kappa}} \text{,} & 
F : \modcat{B_{n+2}^{\kappa}} & \longrightarrow \modcat{B_n^{\kappa}}
\end{align*}
called \defnemph{globalisation} and \defnemph{localisation} respectively. Again localisation is right exact as well as being left exact. For $\wt{\lambda}=(1^{\lambda_1},1^{\lambda_2}) \in \Lambda(n+2)$ we have
\begin{equation*}
F\weyl(1^{\lambda_1},1^{\lambda_2})=\begin{cases}
\weyl(1^{\lambda_1-1},1^{\lambda_2-1}) & \text{if $\lambda_1,\lambda_2 \geq 1$,} \\
0 & \text{otherwise.}
\end{cases}
\end{equation*}
and similarly for $\dweyl(\wt{\lambda})$ and $L(\wt{\lambda})$ \cite[Proposition~3]{martin-ryom-hansen}. Moreover, as long as $\lambda_1,\lambda_2 \geq 1$ we also have $P(\wt{\lambda})$, $\dweyl(\wt{\lambda})$, $I(\wt{\lambda})$, and $T(\wt{\lambda})$ by \cite[A1(4)]{donkbk}, \cite[Proposition~A3.11]{donkbk}, and \cite[Lemma~A4.5]{donkbk}. This implies the \defnemph{stability} of decomposition numbers and tilting multiplicities across all $n$. In other words, for all $n \in \N$ and $\wt{\lambda},\wt{\mu} \in \Lambda(n)$ with $\wt{\lambda}=(1^{\lambda_1},1^{\lambda_2})$ and $\wt{\mu}=(1^{\mu_1},1^{\mu_2})$, the decomposition number $[\weyl(\wt{\mu}):L(\wt{\lambda})]$ only depends on $\lambda_1-\lambda_2$ and $\mu_1-\mu_2$ but not on $n$.

For $\wt{\lambda}=(1^{\lambda_1},1^{\lambda_2}) \in \Lambda(n)$ globalisation behaves similarly for Weyl modules and projective modules, with 
\begin{align*}
G\weyl(1^{\lambda_1},1^{\lambda_2})& =\weyl(1^{\lambda_1+1},1^{\lambda_2+1}) \text{,} & GP(1^{\lambda_1},1^{\lambda_2})& =P(1^{\lambda_1+1},1^{\lambda_2+1})
\end{align*}
but \emph{not} for simple modules, dual Weyl modules, injective modules, or tilting modules. Globalisation is exact on the full subcategory of $\weyl$-filtered modules \cite[Proposition~4]{martin-ryom-hansen}. It also acts as a right inverse for localisation, i.e.~$F \circ G$ is naturally isomorphic to the identity.

Finally we have the key relationship between induction/restriction and localisation/globalisation, which is the natural isomorphism
\begin{equation*}
\ind \iso {\restr} \circ G \text{.}
\end{equation*}
In the case of $B_n^{\kappa}$, the tower of recollement structure behaves well with the anti-involution so the dual statement
\begin{equation*}
\restr \iso F \circ {\ind}
\end{equation*}
also holds.

\subsection{Linkage principle}

There is a linkage principle for the blob algebra, in terms of the following alcove geometry. Let $\W$ be the infinite dihedral group acting on $\Z$ generated by reflections $s_{k}$ about the integers $(\kappa_1-\kappa_2)+ke$ for any $k \in \Z$. Each alcove consists of the integers $(\kappa_1-\kappa_2)+ke<j<(\kappa_1-\kappa_2)+(k+1)e$ lying between two adjacent reflection points. Weights lying inside an alcove are called regular, while those on a reflection point are singular. The fundamental alcove is the unique alcove containing the integer $0$. Two integers are called linked if they are in the same $\W$-orbit. The group $\W$ also acts partially on paths in $\Z$. For a path $p$, if $p(k)$ is the reflection point $(\kappa_1-\kappa_2)+je$, then we write 
\begin{equation*}
s_j^k p(r)=\begin{cases}
p(r) & \text{if $r \leq k$,} \\
s_j p(r) & \text{if $r>k$.}
\end{cases}
\end{equation*}
In other words, $s_j^k p$ is the path obtained by reflecting $p$ after the $k$th point. We say that two paths are linked if one can be obtained by a sequence of reflections of the other.

Write $\Std_{\wt{\lambda}}(\wt{\mu})$ for the set of standard tableaux of shape $\wt{\mu}$ with residue sequence $\domres{\wt{\lambda}}$. It turns out that this set can be described entirely in terms of the alcove geometry above, using the fact that weights and tableaux correspond to points in $\Z$ and paths in $\Z$ respectively.
\begin{prop}[{\cite[Lemma 4.7]{plaza-grdecompblob}}] \label{prop:tablink}
Let $\wt{\lambda},\wt{\mu} \in \Lambda(n)$. Under the tableau-path bijection, the set $\Std_{\wt{\lambda}}(\wt{\mu})$ corresponds to paths which end at $\wt{\mu}$ in the same linkage class as $\domtab{\wt{\lambda}}$.
\end{prop}

\begin{eg}
Suppose $e=4$, $\kappa=(0,2)$, and $n=9$. Let $\wt{\lambda}=(1^8,1)$. The tableau $\domtab{\wt{\lambda}}$ corresponds to the path in red. This path crosses $2$ alcove walls, so there are $2^2=4$ different paths in the linkage class of $\domtab{\wt{\lambda}}$. The other $3$ paths in this linkage class are illustrated in black from the point where they diverge from $\domtab{\wt{\lambda}}$.
\[
\raisebox{-.183in}{\includegraphics[width=9.2cm]{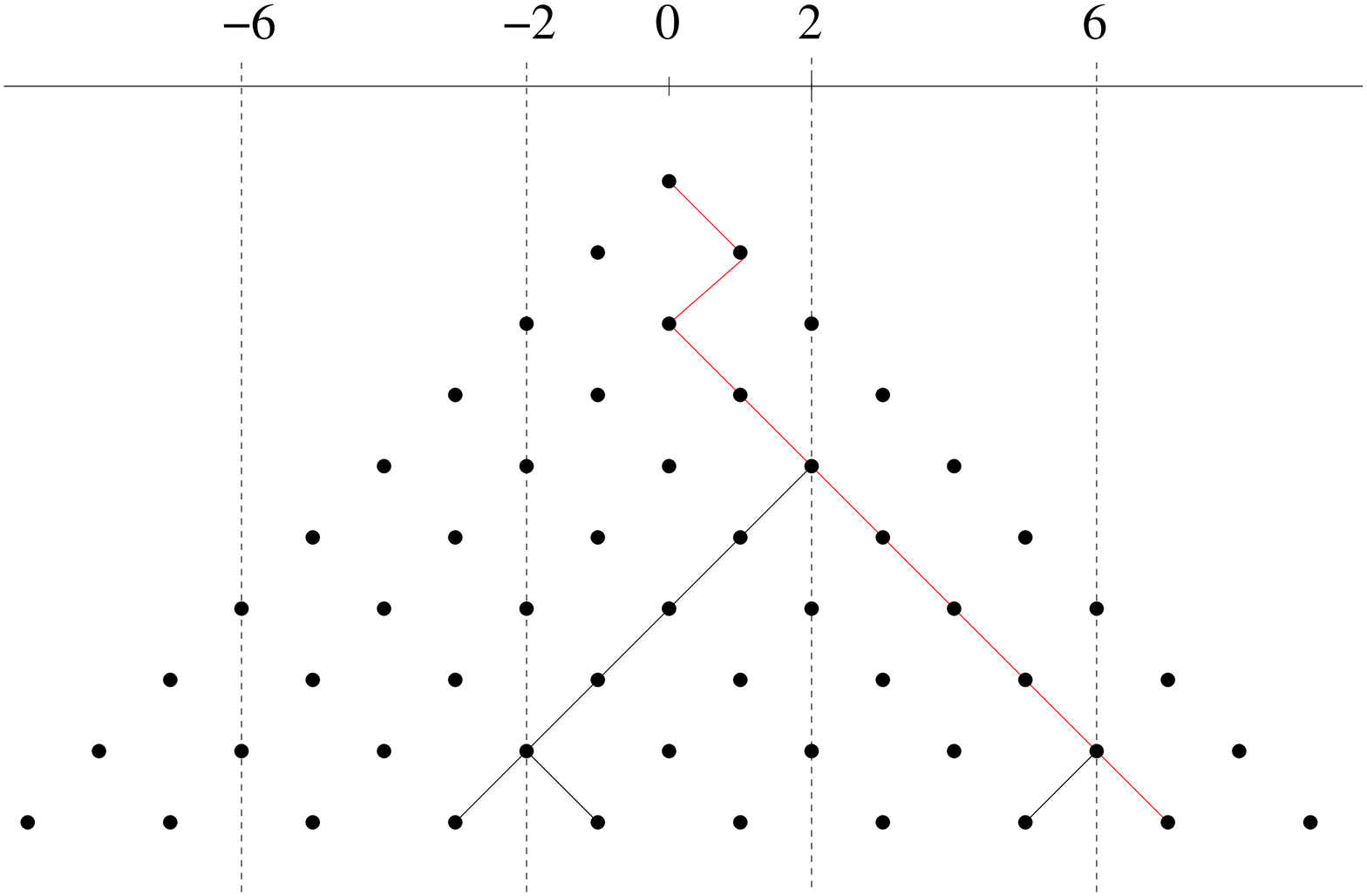}}
\]
These paths correspond to the tableaux
$$
\includegraphics{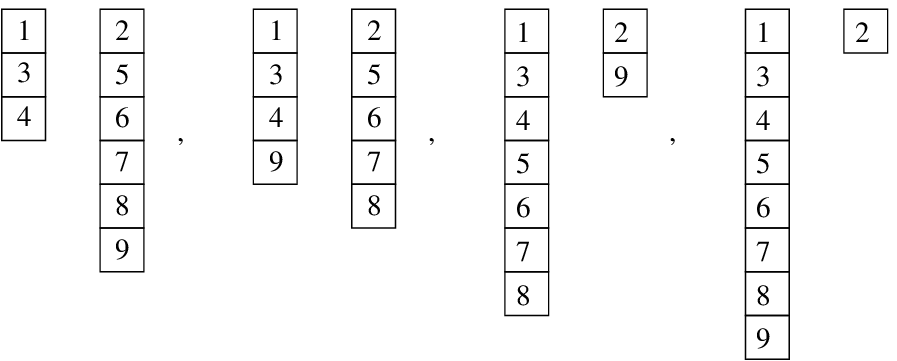}
$$
\end{eg}

\begin{cor}
If $[\weyl(\wt{\mu}):L(\wt{\lambda})] \neq 0$ then $\wt{\mu}$ and $\wt{\lambda}$ are in the same linkage class.
\end{cor}

A consequence of the above result is that if $\wt{\lambda},\wt{\mu} \in \Lambda(n)$ are in different linkage classes, then they are also in different blocks. We write $\pr_{\wt{\lambda}}$ for the functor which projects modules and homomorphisms onto the block(s) of simple modules parametrised by weights in the linkage class of $\wt{\mu}$.

The degrees of tableaux in $\Std_{\wt{\lambda}}(\wt{\mu})$ can also be calculated from their corresponding path. We call a subsequence of $e$ consecutive steps in a path a \defnemph{wall-to-wall step} if the steps start from a wall (i.e.~a reflection point) and continue in a single direction until they reach another wall. For $\tableau{t} \in \Std_{\wt{\lambda}}(\wt{\mu})$ a standard tableau write $w(\tableau{t})$ for the number of wall-to-wall steps across the fundamental alcove.

\begin{thm}[{\cite[Theorem 4.9]{plaza-grdecompblob}}] \label{thm:degtab}
Let $\tableau{t} \in \Std_{\wt{\lambda}}(\wt{\mu})$. Let $\delta(\tableau{t})$ be $1$ if the first step after all wall-to-wall steps points toward the origin, and $0$ otherwise. Then $\deg \tableau{t}=w(\tableau{t})+\delta(\tableau{t})$. 
\end{thm}

Finally we describe the decomposition numbers in characteristic $0$ in terms of the alcove geometry. For any regular weight $\wt{\lambda}$, there exists a unique weight $\wt{\lambda}_{\rm fund}$ in the fundamental alcove and $w_{\wt{\lambda}} \in \W$ such that $w_{\wt{\lambda}}(\wt{\lambda}_{\rm fund})=\wt{\lambda}$. For $x,y \in \W$, define $h_{y,x}$ by
\begin{equation*}
h_{y,x}(v)=\begin{cases}
v^{\len(x)-\len(y)} & \text{if $y \leq x$,} \\
0 & \text{otherwise.}
\end{cases}
\end{equation*}
This is the Kazhdan--Lusztig polynomial associated to $\W$ (in the notation of \cite{soergel-KL}).

\begin{thm}[{\cite[Theorem~5.11]{plaza-grdecompblob}}] \label{thm:grdecompblob}
Suppose $\field$ is a field of characteristic $0$. Let $\wt{\lambda},\wt{\mu}$ be two regular weights lying in the same linkage class. Then we have
\begin{equation*}
[\weyl(\wt{\mu}) : L(\wt{\lambda})]_v=h_{w_{\wt{\mu}},w_{\wt{\lambda}}}(v) \text{.}
\end{equation*}
\end{thm}

There is also a singular version of this result. If $\wt{\lambda}$ is a singular weight, we label the weights in the linkage class of $\wt{\lambda}$ following \cite[Example 5.5]{plaza-grdecompblob}. First set $\wt{\lambda}_0=\wt{\lambda}$. Suppose that $\wt{\lambda}$ corresponds to a positive classical weight (i.e.~a weight on the right side of the origin in our pictures). Working inductively, for $k$ even (resp.~odd) we define $\wt{\lambda}_{k+1}$ to be the rightmost (resp.~leftmost) weight in the linkage class distinct from $\wt{\lambda}_0,\wt{\lambda}_1,\dotsc,\wt{\lambda}_k$. Similarly, when $\wt{\lambda}$ corresponds to a negative classical weight, for $k$ even (resp.~odd) we define $\wt{\lambda}_{k+1}$ to be the leftmost (resp.~rightmost) weight in the linkage class distinct from $\wt{\lambda}_0,\wt{\lambda}_1,\dotsc, \wt{\lambda}_k$.

\begin{thm}[{\cite[Theorem~5.14]{plaza-grdecompblob}}] \label{thm:grdecompblob-singular}
Suppose $\field$ is a field of characteristic $0$. Let $\wt{\lambda}$ be a singular weight. Then if $\wt{\lambda}_k$ is defined we have
\begin{equation*}
[\weyl(\wt{\lambda}_k):L(\wt{\lambda})]_v=v^k \text{.}
\end{equation*}
\end{thm}

\begin{rem}
In general, it is easier to use tableaux when working with permutations of the form $d_{\tableau{t}}$ for some tableau $\tableau{t}$ of shape $\wt{\lambda}$, as one can read off $d_{\tableau{t}}$ directly from the two tableaux $\tableau{t}$ and $\domtab{\wt{\lambda}}$. By contrast, it is easier to use paths in order to apply Proposition~\ref{prop:tablink}. We will mostly use tableaux in the arguments below, but the careful reader may use the tableau-path bijection in order to translate our arguments into the language of paths if necessary.
\end{rem}

\section{Bases for projective indecomposable modules}
\label{sec:bases}



For the rest of this paper, we will assume that $\field$ is a field of characteristic $0$. Most of the previous results are known to hold in some form for the classical blob algebra. To proceed further we must make use of the KLR-style presentation of $B_n^{\kappa}$, and in particular the grading.

\subsection{A Temperley--Lieb subalgebra}

As $B_n^{\kappa}$ is graded, it has a subalgebra of degree $0$ elements. This subalgebra was classified in \cite[\S~5.4--5.5]{libedinskyplaza}. We summarise their results below.

\begin{defn} \label{de:diamond} 
Let $\wt{\lambda}=(1^{\lambda_1},1^{\lambda_2}) \in \Lambda(n)$. Suppose the weight $\wt{\lambda}$ does not lie in the interior of the fundamental alcove. We define $f_{\wt{\lambda}}$ to be the minimal positive integer such that the $f_{\wt{\lambda}}$th point of the path corresponding to $\domtab{\wt{\lambda}}$ lies on a wall of the fundamental alcove. In other words,
\begin{equation}
f_{\wt{\lambda}}=\begin{cases} 
\min(\{2\lambda_2+(\kappa_1-\kappa_2)+je : j \in \Z\} \cap \N) & \text{if $\lambda_1 \geq \lambda_2$,} \\
\min(\{2\lambda_1-(\kappa_1-\kappa_2)+je : j \in \Z\} \cap \N) & \text{if $\lambda_1<\lambda_2$.}
\end{cases} \label{eq:flambda}
\end{equation}

For $j \in \N$ write $f(j)=f_{\wt{\lambda}}+je$. For all $j \in \N$ such that $f(j) \leq n-e$ we define the \defnemph{diamond} of $\wt{\lambda}$ at position $f(j)$ to be
\begin{multline}
U^{\wt{\lambda}}_j=(\psi_{f(j)})(\psi_{f(j)-1}\psi_{f(j)+1})(\psi_{f(j)-2}\psi_{f(j)}\psi_{f(j)+2}) \dotsm \\
\dotsm (\psi_{f(j)-e+1}\psi_{f(j)-e+3} \dotsm \psi_{f(j)+e-3}\psi_{f(j)+e-1}) \dotsm \\
\dotsm (\psi_{f(j)-2}\psi_{f(j)}\psi_{f(j)+2})(\psi_{f(j)-1}\psi_{f(j)+1})(\psi_{f(j)})e(\domres{\wt{\lambda}}) \text{.}
\end{multline}
\end{defn}

The name `diamond' comes from the corresponding KLR diagram for this element, e.g.~
\begin{equation*}
\includegraphics[scale=0.4]{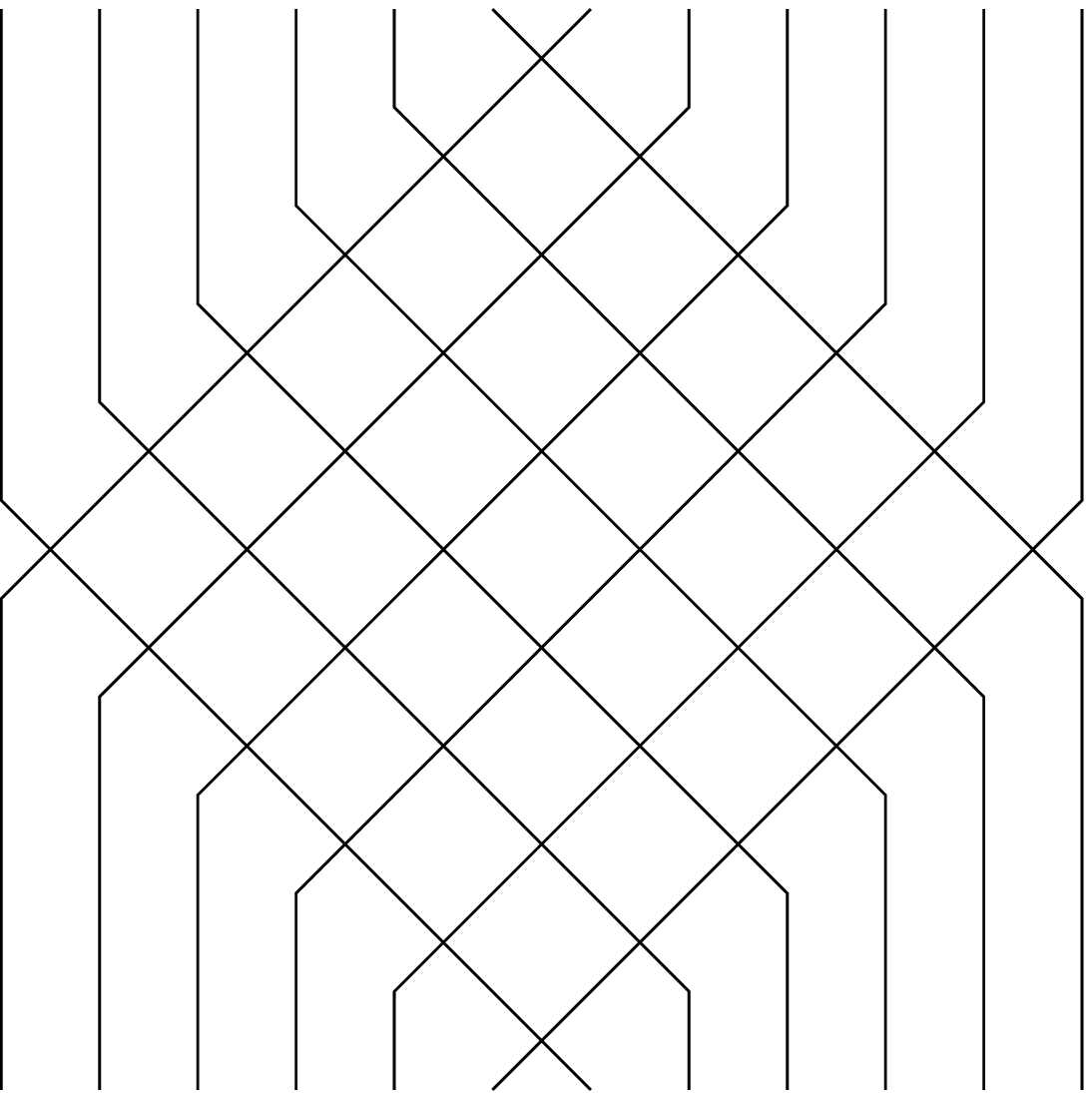}
\end{equation*}
for $e=6$.
The cyclotomic KLR algebra versions of these elements previously appeared in \cite[(4.2)]{kleshchev-mathas-ram}, while the effect of similar permutations on paths was seen even earlier, e.g.~\cite[Figure 4]{martin90}.

\begin{thm}[{\cite[Theorem 5.24]{libedinskyplaza}}]
Let $\wt{\lambda} \in \Lambda(n)$. The diamonds of weight $\wt{\lambda}$ generate the degree $0$ subalgebra of $e(\domres{\wt{\lambda}}) B_n^{\kappa} e(\domres{\wt{\lambda}})$. This subalgebra is isomorphic to a Temperley--Lieb algebra with loop parameter $2(-1)^{e-1}$, with the diamond at position $f_{\wt{\lambda}}+je$ corresponding to the standard Temperley--Lieb diagrammatic generator at index $j$. In other words, the diamonds of weight $\wt{\lambda}$ satisfy the relations
\begin{align*}
U^{\wt{\lambda}}_i U^{\wt{\lambda}}_j & =U^{\wt{\lambda}}_j U^{\wt{\lambda}}_i & & \text{when $|i-j|>1$,} \\
U^{\wt{\lambda}}_i U^{\wt{\lambda}}_j U^{\wt{\lambda}}_i & =U^{\wt{\lambda}}_i & & \text{when $|i-j|=1$,} \\
(U^{\wt{\lambda}}_i)^2& =2(-1)^{e-1} U^{\wt{\lambda}}_i & & \text{for all $i$,}
\end{align*}
and this gives a complete presentation of the subalgebra generated by them.
\end{thm}

Recall that in quantum characteristic $0$ the Temperley--Lieb algebra is semisimple, with a unique $1$-dimensional irreducible module.
The central idempotent corresponding to this irreducible module is
sometimes
called the \defnemph{Jones--Wenzl projector}.
We write $\JW^{\wt{\lambda}}$ for the corresponding idempotent in $e(\domres{\wt{\lambda}}) B_n^{\kappa} e(\domres{\wt{\lambda}})$. In our notation, one of the defining properties of $\JW^{\wt{\lambda}}$ is that $U^{\wt{\lambda}}_j \JW^{\wt{\lambda}}=0$ for all $j$.

\begin{lem}
Let $\wt{\lambda} \in \Lambda(n)$. Then $P(\wt{\lambda}) \iso B_n^{\kappa} \JW^{\wt{\lambda}}$.
\end{lem}

\begin{proof}
Clearly $B_n^{\kappa} \JW^{\wt{\lambda}}$ is an indecomposable projective module, as $\JW^{\wt{\lambda}}$ is a primitive idempotent for the Temperley--Lieb subalgebra (and thus also for $e(\domres{\wt{\lambda}})B_n^{\kappa}e(\domres{\wt{\lambda}})$ and $B_n^{\kappa}$).
Suppose $B_n^{\kappa} \JW^{\wt{\lambda}}=P(\wt{\mu})$.
Then $B_n^{\kappa}$ maps onto $\weyl(\wt{\mu})$, which induces a surjective homomorphism
\begin{equation*}
e(\domres{\wt{\lambda}})B_n^{\kappa} \JW^{\wt{\lambda}} \longrightarrow e(\domres{\wt{\lambda}})\weyl(\wt{\mu})
\end{equation*}
of $e(\domres{\wt{\lambda}})B_n^{\kappa}e(\domres{\wt{\lambda}})$-modules.
By the defining property of $\JW^{\wt{\lambda}}$, the degree $0$ part of the domain is $1$-dimensional.
On the other hand, the degree $0$ part of the codomain is the cellular module of weight $\wt{\mu}$ for the Temperley--Lieb subalgebra.
(Here we use the fact that the cellular structure of the Temperley--Lieb subalgebra is compatible with that of $e(\domres{\wt{\lambda}})B_n^{\kappa}e(\domres{\wt{\lambda}})$, because the latter is positively graded by Theorem~\ref{thm:degtab}.)
This has dimension strictly larger than $1$ unless $\wt{\mu}=\wt{\lambda}$.
%
\end{proof}

\subsection{Maximal degree tableaux}

The following key combinatorial lemma constructs maximal degree tableaux, which are of fundamental importance in the characteristic $0$ representation theory of $B_n^{\kappa}$.

\begin{lem} \label{lem:subqtab}
Let $\wt{\lambda} \in \Lambda(n)$ be a weight. For each $\wt{\mu} \in \W \wt{\lambda}$ with $\wt{\lambda} \domleq \wt{\mu}$, there is a unique tableau $\subqtab{\wt{\mu}}{\wt{\lambda}} \in \Std_{\wt{\lambda}}(\wt{\mu})$ of maximal degree 
\begin{equation*}
\deg \subqtab{\wt{\mu}}{\wt{\lambda}}=
\begin{cases}
\len(w_{\wt{\lambda}})-\len(w_{\wt{\mu}}) & \text{if $\wt{\lambda}$ is regular,} \\
k & \text{if $\wt{\lambda}$ is singular and $\wt{\mu}=\wt{\lambda}_k$.}
\end{cases}
\end{equation*}
\end{lem}

\begin{proof}
Let $\tableau{t} \in \Std_{\wt{\lambda}}(\wt{\mu})$, and write $d$ for $\len(w_{\wt{\lambda}})-\len(w_{\wt{\mu}})$. From Theorem \ref{thm:degtab} recall that $\deg \tableau{t}$ is either $w(\tableau{t})$ or $w(\tableau{t})+1$, where $w(\tableau{t})$ is the number of wall-to-wall steps inside the fundamental alcove for the path corresponding to $\tableau{t}$. By Proposition \ref{prop:tablink} $\tableau{t}$ lies in the linkage class of $\domtab{\wt{\lambda}}$. The path corresponding to $\domtab{\wt{\lambda}}$ contains $\len(w_{\wt{\lambda}})-1$ wall-to-wall steps, whereas any path with endpoint $\wt{\mu}$ must have at least $\len(w_{\wt{\mu}})-1$ wall-to-wall steps outside the fundamental alcove to get there. Thus $w(\tableau{t})$ is bounded above by $d$. 

There are four cases, according to the parity of $d$ and whether
$\wt{\lambda}$ and $\wt{\mu}$ lie on the same side of the origin or
not. We will focus on one of these cases; the other three are
similar. Suppose $d$ is even and that $\wt{\lambda}$ and $\wt{\mu}$
both lie on the same side of the origin. First we note that since
paths to $\wt{\lambda}$ and $\wt{\mu}$ must eventually pass through
the same wall of the fundamental alcove, $w(\tableau{t})$ is even for
all $\tableau{t} \in \Std_{\wt{\lambda}}(\wt{\mu})$.
There exists a tableau
$\subqtab{\wt{\mu}}{\wt{\lambda}} \in \Std_{\wt{\lambda}}(\wt{\mu})$
with $w(\subqtab{\wt{\mu}}{\wt{\lambda}})=d$ maximal, e.g.
\[
\raisebox{-.183in}{\includegraphics[width=9.2cm]{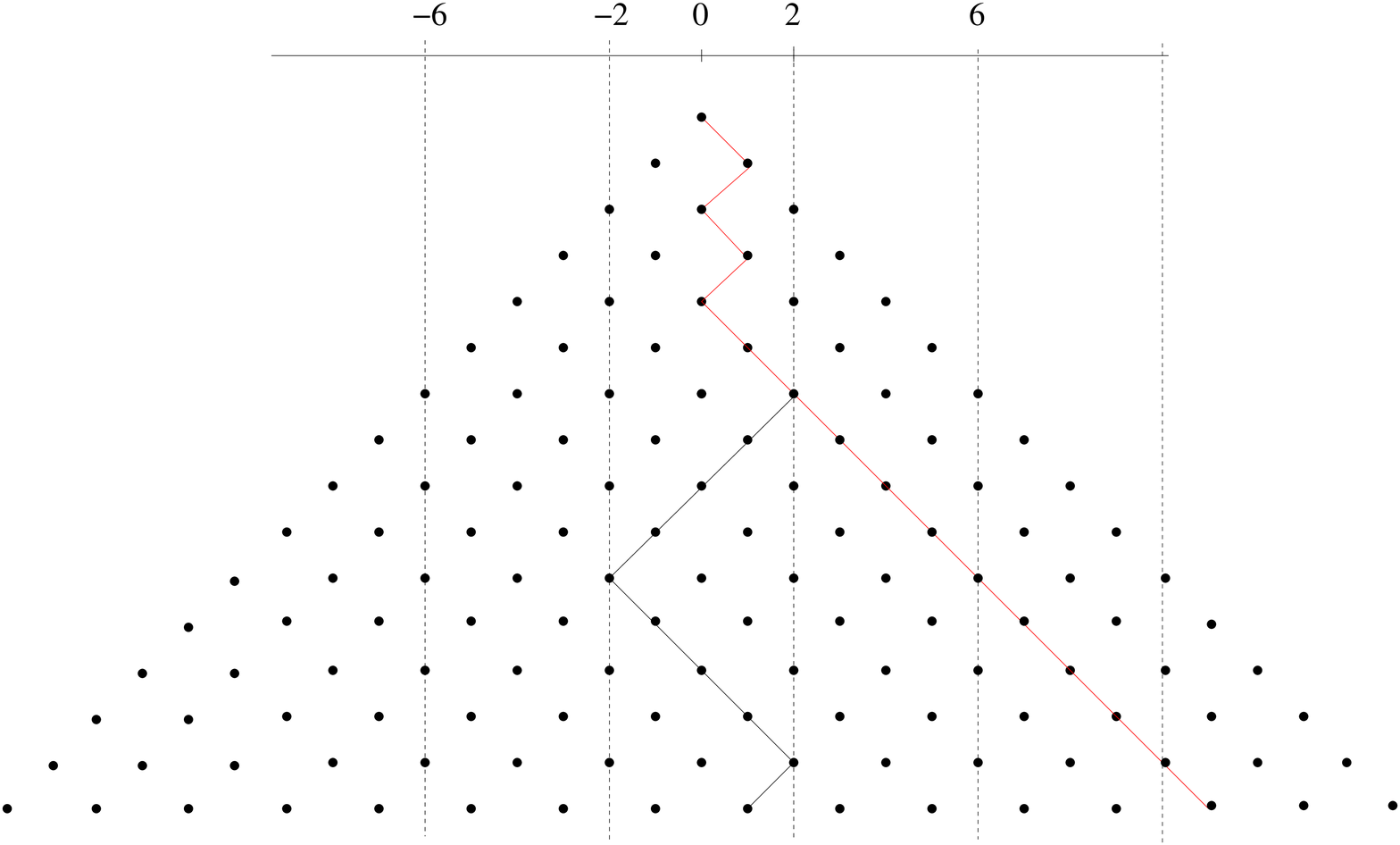}}
\]
Moreover, this tableau is unique: for any such path, the wall-to-wall steps inside the fundamental alcove must occur as early as possible. If not, the path would have to leave and then return to the fundamental alcove, wasting wall-to-wall steps in the process. Finally, $\subqtab{\wt{\mu}}{\wt{\lambda}}$ has maximal degree too. From the picture above $\deg \subqtab{\wt{\mu}}{\wt{\lambda}}=w(\subqtab{\wt{\mu}}{\wt{\lambda}})$, and for all other tableaux $\tableau{t}$ we have 
\begin{equation*}
\deg \tableau{t} \leq w(\tableau{t})+1 \leq (w(\subqtab{\wt{\mu}}{\wt{\lambda}})-2)+1<\deg \subqtab{\wt{\mu}}{\wt{\lambda}}
\end{equation*}
\end{proof}

\begin{rem} 
An alternative proof of this result uses \cite[Theorem 4.9]{libedinskyplaza} to reduce the problem of determining graded dimensions of Weyl modules to a calculation in the Iwahori--Hecke algebra corresponding to $\W$. The result follows from the observation that the `Bott--Samelson' elements (i.e.~products of simple Kazhdan--Lusztig generators) in this algebra are just sums of Kazhdan--Lusztig basis elements.
%

\end{rem}

An interesting application of these maximal degree tableaux is the following lemma, which identifies composition factors of Weyl modules in terms of the cellular basis.

\begin{lem} \label{lem:weylsubquot}
Let $\wt{\lambda} \in \Lambda(n)$, and suppose $\wt{\mu} \in \W \wt{\lambda}$ with $\wt{\lambda} \domleq \wt{\mu}$. Consider the submodule $B_n^{\kappa}\psi_{\subqtab{\wt{\mu}}{\wt{\lambda}}}$ of the Weyl module $\weyl(\wt{\mu})$.
Then there is a homomorphism
\begin{align*}
B_n^{\kappa}\psi_{\subqtab{\wt{\mu}}{\wt{\lambda}}} & \longrightarrow L(\wt{\lambda})\langle \deg \subqtab{\wt{\mu}}{\wt{\lambda}} \rangle \\
\psi_{\subqtab{\wt{\mu}}{\wt{\lambda}}} & \longmapsto \psi_{\domtab{\wt{\lambda}}} \text{.}
\end{align*}
\end{lem}

\begin{proof}
Let $d=\deg \subqtab{\wt{\mu}}{\wt{\lambda}}=\len(w_{\wt{\lambda}})-\len(w_{\wt{\mu}})$. From Theorem \ref{thm:grdecompblob} we have
\begin{equation*}
[\weyl(\wt{\mu}):L(\wt{\lambda})]_v=h_{w_{\wt{\mu}},w_{\wt{\lambda}}}(v)=v^d \text{,}
\end{equation*}
so the Weyl module $\weyl(\wt{\mu})$ contains exactly one subquotient isomorphic to $L(\wt{\lambda})\langle d\rangle$. Recall that $L(\wt{\lambda})$ is generated by a vector of residue $\domres{\wt{\lambda}}$ in degree $0$. This means that the unique subquotient of $\weyl(\wt{\mu})$ isomorphic to $L(\wt{\lambda})\langle d\rangle$ is generated by some vector of residue $\domres{\wt{\lambda}}$ in degree $d$. But from Proposition \ref{prop:tablink} and Lemma \ref{lem:subqtab} we have
\begin{equation*}
\dim_{v} e(\domres{\wt{\lambda}}) \weyl(\wt{\mu})=\sum_{\tableau{t} \in \Std_{\wt{\lambda}}(\wt{\mu})} v^{\deg \tableau{t}}=v^d+\text{l.o.t.}
\end{equation*}
In other words, the subspace of vectors with the correct residue and degree is one-dimensional, spanned by $\psi_{\subqtab{\wt{\mu}}{\wt{\lambda}}}$, so the result follows.
\end{proof}

Applying Brauer--Humphreys reciprocity, we can also identify the Weyl
subquotient isomorphic to $\weyl(\wt{\mu})$ inside $P(\wt{\lambda})$. 

\begin{cor} \label{cor:projsubquot}
Let $\wt{\lambda} \in \Lambda(n)$, and suppose $\wt{\mu} \in \W \wt{\lambda}$ with $\wt{\lambda} \domleq \wt{\mu}$. There is a surjective homomorphism 
\begin{align*}
B_n^{\kappa}\psi_{\domtab{\wt{\mu}}\subqtab{\wt{\mu}}{\wt{\lambda}}} \JW^{\wt{\lambda}} & \longrightarrow \weyl(\wt{\mu})\langle \deg \subqtab{\wt{\mu}}{\wt{\lambda}}\rangle \\
\psi_{\domtab{\wt{\mu}} \subqtab{\wt{\mu}}{\wt{\lambda}}} \JW^{\wt{\lambda}} & \longmapsto \psi_{\domtab{\wt{\mu}}} \text{,}
\end{align*}
where the domain is a submodule of $B_n^{\kappa} \JW^{\wt{\lambda}} \iso P(\wt{\lambda})$. 
\end{cor}

\section{Singular projective modules}
\label{sec:singprojmod}

The aim of this section is to determine the socles of the
indecomposable projective modules associated to singular weights
--- Theorem~\ref{thm:endwallprojsimplesocle} and Corollary~\ref{cor:singprojsimplesocle}.
This turns out to be enough to completely determine the structure of these
modules.
The result will then be used in \S\ref{ss:regprojmod} to address the
corresponding (harder) non-singular cases.

Our general strategy is to identify possible generators for the socle 
in Lemma~\ref{lem:possible-socle} and then to rule out all but one of them
via direct computation. The computation involves the Jones--Wenzl
projector, which is difficult to work with directly because in the standard
basis it is a sum with many terms. Luckily nearly all of these terms combine or
vanish in the computation when multiplied by certain cellular basis elements.

In this section we will assume that $n \equiv \kappa_1-\kappa_2 \pmod{e}$, or in other words that there is a wall at $n$. Fix $\wt{\eta}=(1^n,\emptyset) \in \Lambda(n)$ and let $m \in \N$ such that $n=f_{\wt{\eta}}+me$ (see \eqref{eq:flambda} for a definition of $f_{\wt{\eta}}$). Recall how the linkage class of $\wt{\eta}$ consists of the weights $\wt{\eta}_j$ for some non-negative integers $j$. The maximal weight in this linkage class is $\wt{\eta}_m$, which is on a wall of the fundamental alcove. Note that $f_{\wt{\eta}_j}=f_{\wt{\eta}}+je$, because the distance from $\wt{\eta}_j$ to the nearest fundamental alcove wall is $(m-j)e$ steps.

%

\subsection{Cellular basis factorization}

We begin with a factorization of some of the distinguished cellular basis elements from the previous section.

\begin{prop} \label{prop:psicrosses}
For all integers $0 \leq j \leq k \leq m$ we have
\begin{equation*}
\psi_{\domtab{\wt{\eta}_k} \subqtab{\wt{\eta}_k}{\wt{\eta}_j}}=x_j x_{j+1} \dotsm x_{k-1} \psi_{f_{\wt{\eta}}+je} \psi_{f_{\wt{\eta}}+(j+1)e} \dotsm \psi_{f_{\wt{\eta}}+(k-1)e} e(\domres{\wt{\eta}_j})
\end{equation*}
for some elements $x_r \in B_n^{\kappa}$ (with $j \leq r<k$) which satisfy the following properties:
\begin{enumerate}[label={\rm (\roman*)}]
\item \label{item:xrpropsfirst} for fixed $r$ the element $x_r$ does not depend on $j$ or $k$;

\item for $r \neq s$, $x_r x_s = x_s x_r$ and $x_r \psi_{f_{\wt{\eta}}+se}=\psi_{f_{\wt{\eta}}+se} x_r$;

\item \label{item:xrpropslast} for each $r$ we have
\begin{align*}
x_r\overline{x_r}& =e(\domres{\wt{\eta}_k}) \text{,} \\
\overline{x_r} x_r& =e(s_{f_{\wt{\eta}}+re} \domres{\wt{\eta}_j}) \text{.}
\end{align*}
\end{enumerate}
\end{prop}

\begin{proof}
Let $d=d_{\subqtab{\wt{\eta}_k}{\wt{\eta}_j}}$. Recall that $d$ is the permutation which maps $\domtab{\wt{\eta}_k}$ to $\subqtab{\wt{\eta}_k}{\wt{\eta}_j}$. 

For $0 \leq l \leq m$, write $\wt{\eta}_l=(1^{\eta_{l,1}},1^{\eta_{l,2}})$ and set $r_l=2\min(\eta_{l,1},\eta_{l,2})$. From 
\eqref{eq:flambda}
it is clear that
\begin{equation*}
f_{\wt{\eta}_l}=f_{\wt{\eta}}+le=\begin{cases}
r_l+f_{\wt{\eta}}& \text{if $l$ is even,} \\
r_l+(e-f_{\wt{\eta}}) & \text{if $l$ is odd.}
\end{cases}
\end{equation*}
This means that
\begin{equation*}
r_l=\begin{cases}
le & \text{if $l$ is even,} \\
(l-1)e+2f_{\wt{\eta}} & \text{if $l$ is odd.}
\end{cases}
\end{equation*}
Thus the integers $1 \leq r \leq r_j$ lie in the same boxes in the tableaux $\domtab{\wt{\eta}_j}$, $\subqtab{\wt{\eta}_k}{\wt{\eta}_j}$, and $\domtab{\wt{\eta}_k}$ so we have $d(r)=r$. Similarly when $r_k<r \leq n$, $r$ is in the same box in both $\domtab{\wt{\eta}_k}$ and $\subqtab{\wt{\eta}_k}{\wt{\eta}_j}$ so $d(r)=r$ here as well.

For $j \leq l<k$, the boxes in $\subqtab{\wt{\eta}_k}{\wt{\eta}_j}$ with labels $r_l<r \leq r_{l+1}$ form the skew tableau
\begingroup
\allowdisplaybreaks
\begin{align*}
\begin{gathered}
\includegraphics[scale=0.5]{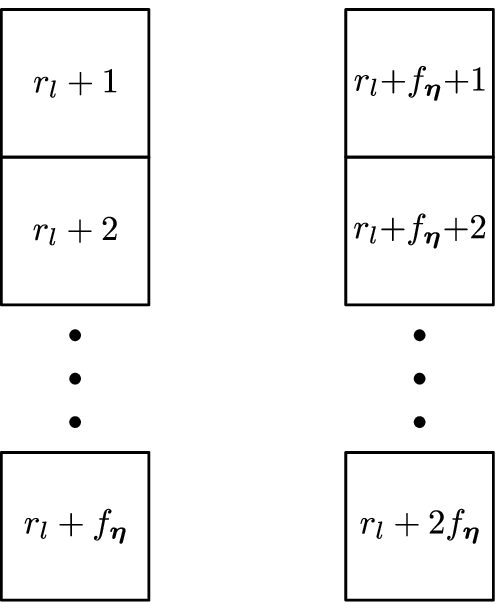}
\end{gathered} & & & \text{if $l$ is even,} \\
\begin{gathered}
\includegraphics[scale=0.5]{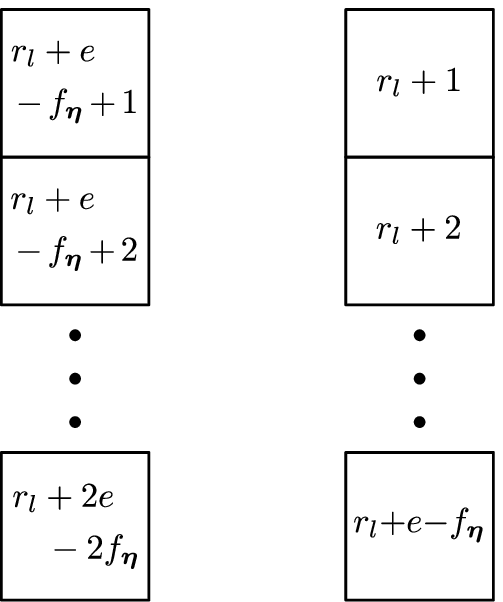}
\end{gathered} & & & \text{if $l$ is odd,}
\end{align*}
\endgroup
while the same boxes in $\domtab{\wt{\eta}_k}$ form the skew tableau
\begingroup
\allowdisplaybreaks
\begin{align*}
\begin{gathered}
\includegraphics[scale=0.5]{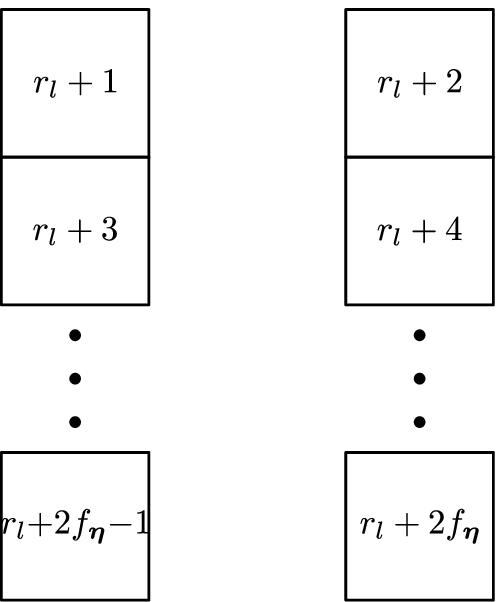}
\end{gathered} & & & \text{if $l$ is even,} \\
\begin{gathered}
\includegraphics[scale=0.5]{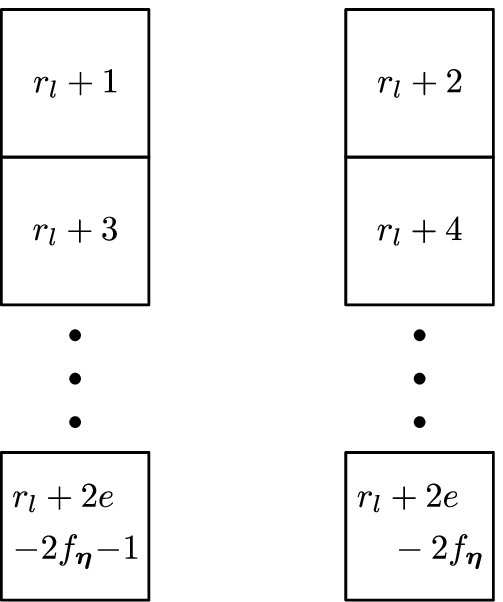}
\end{gathered} & & & \text{if $l$ is odd.}
\end{align*}
\endgroup
This of course means that $d$ restricted to $r_l<r \leq r_{l+1}$ is still a permutation $d_l$. In fact $d_l$ corresponds to a triangular portion of the lower half of a `diamond permutation':
\begin{equation*}
\includegraphics[scale=0.35]{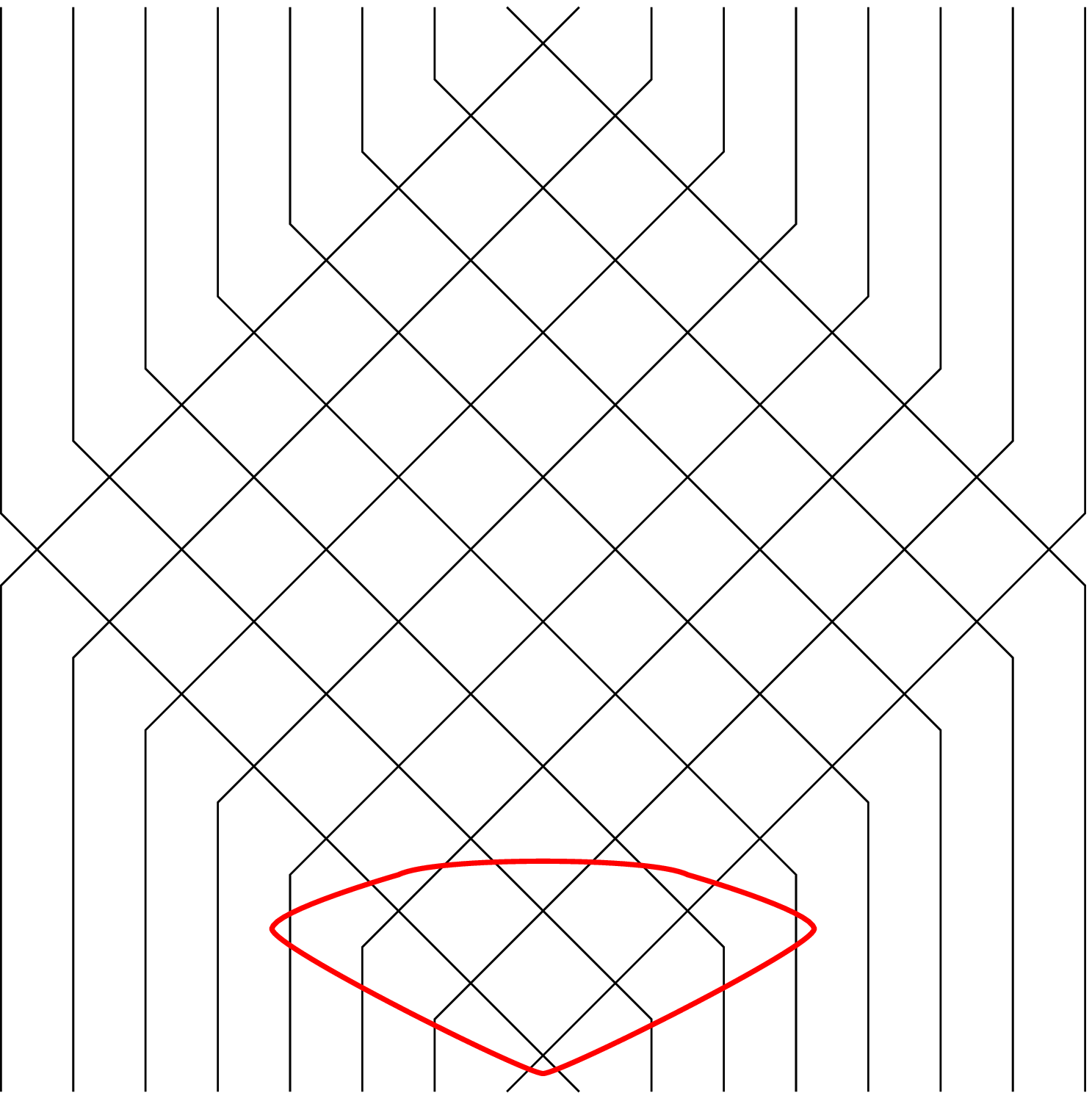}
\end{equation*}
The easiest way to see this is to apply the `layers' (each a product of several commuting transpositions) in turn to the skew tableaux above. For example, the first $(f_{\wt{\eta}}-1)$ layers permute the skew tableau with $f_{\wt{\eta}}$ rows as follows:
\begin{multline*}
\begin{gathered}
\includegraphics[scale=0.5]{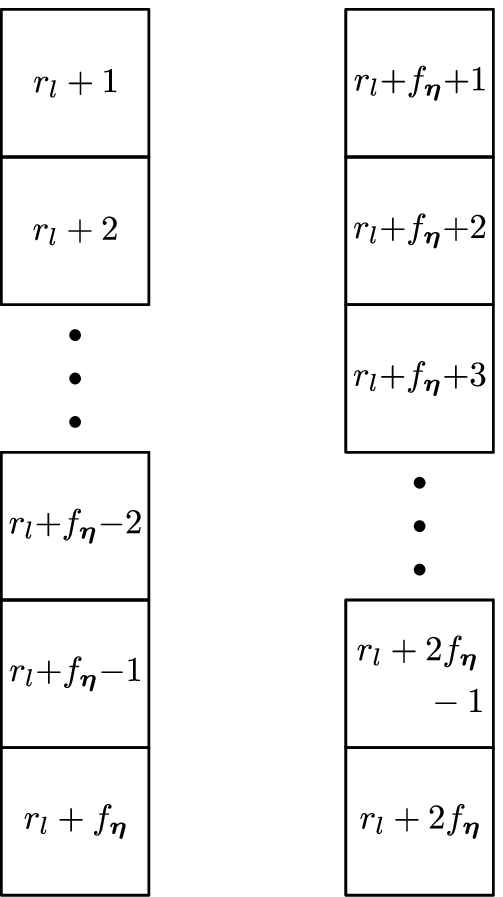}
\end{gathered} \quad \longmapsto \quad
\begin{gathered}
\includegraphics[scale=0.5]{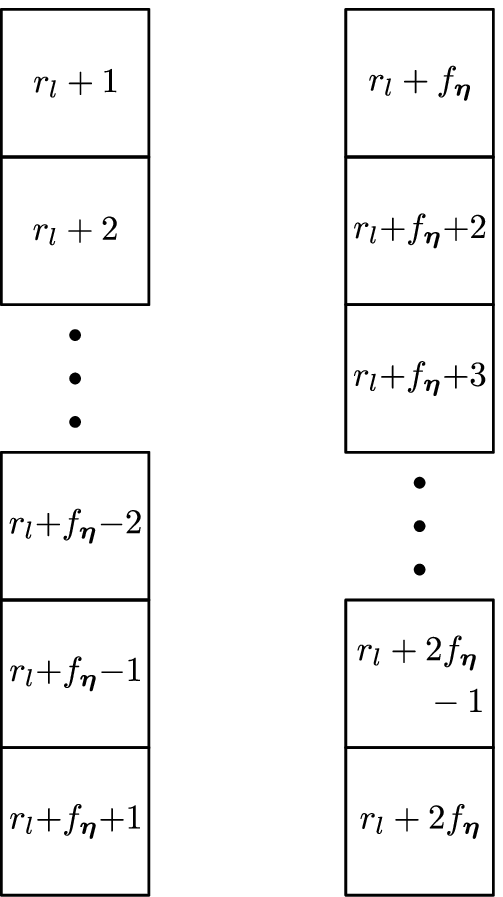}
\end{gathered} \quad \longmapsto \quad
\begin{gathered}
\includegraphics[scale=0.5]{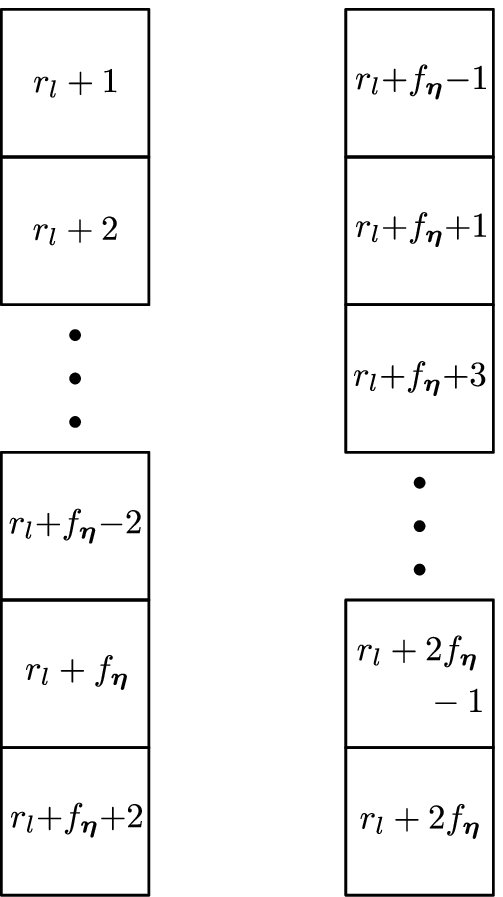}
\end{gathered} \\ \quad \longmapsto \quad
\begin{gathered}
\includegraphics[scale=0.5]{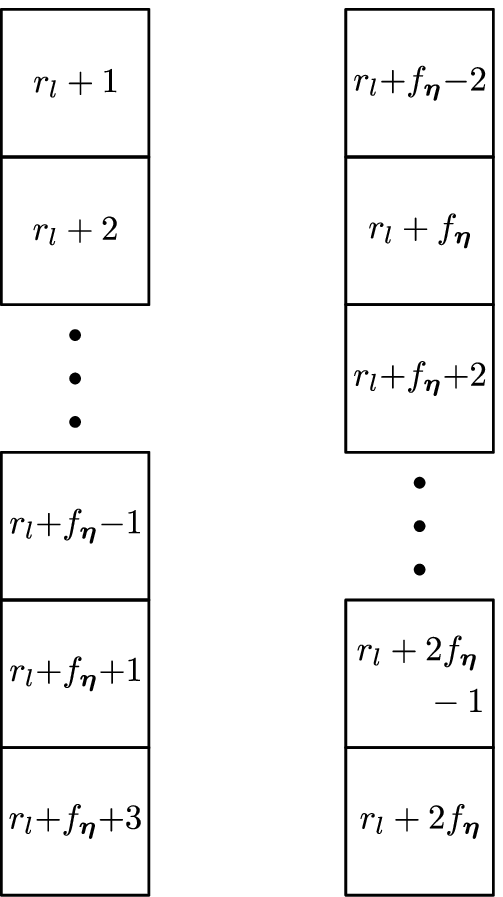}
\end{gathered} \quad \longmapsto \dotsb \longmapsto \quad
\begin{gathered}
\includegraphics[scale=0.5]{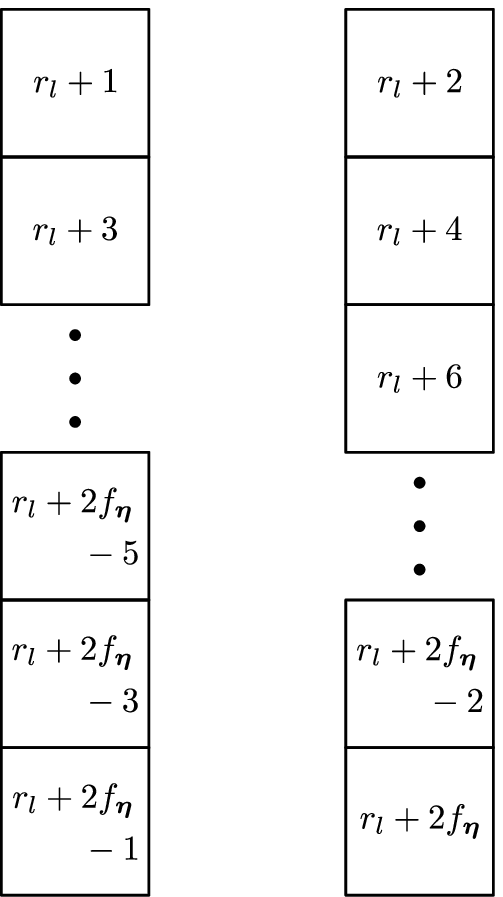}
\end{gathered}
\end{multline*}
The number of layers in the triangle is either $f_{\wt{\eta}}-1$ or $e-f_{\wt{\eta}}$ depending on parity. But $2 \leq f_{\wt{\eta}} \leq e-2$, so in both cases the corresponding diagram in the blob algebra factors as $x_l \psi_{f_{\wt{\eta}}+le}$ with $x_l$ generated by transpositions of degree $0$. Properties \ref{item:xrpropsfirst}--\ref{item:xrpropslast} follow immediately.
\end{proof}

\begin{eg}
Let $n=21$, $e=6$ and $\kappa=(0,3)$. The weight $\wt{\eta}=(1^{21},\emptyset)$ is singular because $21 \equiv 3-0 \pmod{6}$.
Observe that $\wt{\eta}_1=(1^3,1^{15})$ and that $\wt{\eta}_3=(1^9,1^{12})$. Then
\begin{equation*}
\psi_{\domtab{\wt{\eta}_3} \subqtab{\wt{\eta}_3}{\wt{\eta}_1}}=
\begin{gathered}
\includegraphics[scale=0.6]{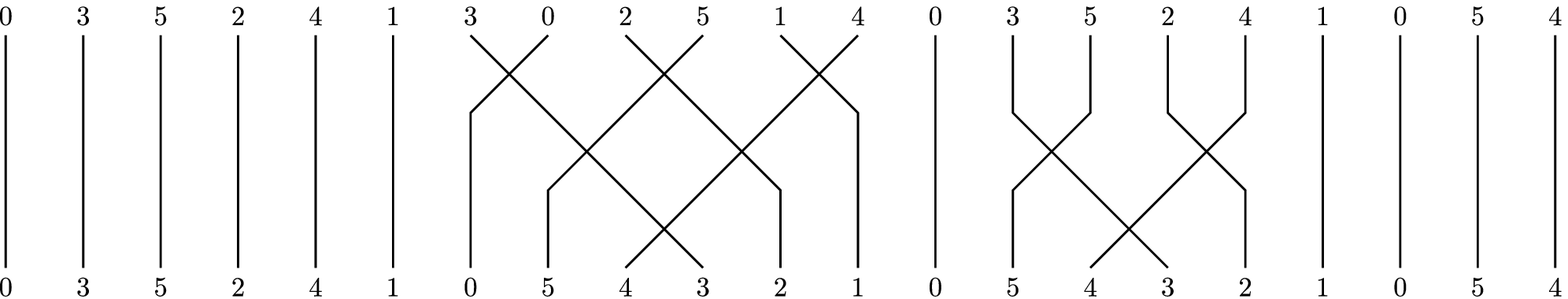} 
\end{gathered}
\text{.}
\end{equation*}
We also have
\begin{equation*}
x_1 x_2 = 
\begin{gathered}
\includegraphics[scale=0.6]{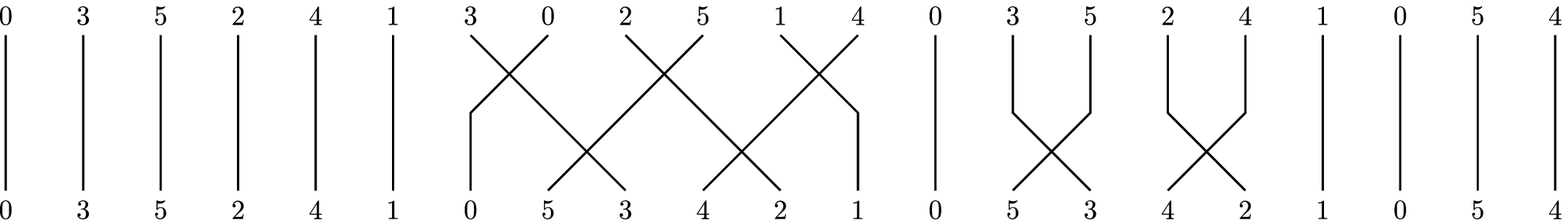}
\end{gathered}
\text{.}
\end{equation*}
\end{eg}

Some immediate consequences of Proposition~\ref{prop:psicrosses} include the following corollaries.

\begin{cor} \label{cor:transitivityofpsi}
For all integers $0 \leq j \leq k \leq l \leq m$ we have $\psi_{\domtab{\wt{\eta}_l} \subqtab{\wt{\eta}_l}{\wt{\eta}_k}} \psi_{\domtab{\wt{\eta}_k} \subqtab{\wt{\eta}_k}{\wt{\eta}_j}}=\psi_{\domtab{\wt{\eta}_l} \subqtab{\wt{\eta}_l}{\wt{\eta}_j}}$.
\end{cor}

\begin{cor} \label{cor:psidoublecrosses}
For all integers $0 \leq j \leq k \leq m$ we have
\begin{equation*}
\psi_{\subqtab{\wt{\eta}_k}{\wt{\eta}_j} \subqtab{\wt{\eta}_k}{\wt{\eta}_j}} = \psi_{f_{\wt{\eta}}+je}^2 \psi_{f_{\wt{\eta}}+(j+1)e}^2 \dotsm \psi_{f_{\wt{\eta}}+(k-1)e}^2 e(\domres{\wt{\eta}_j}) \text{.}
\end{equation*}
\end{cor}

\subsection{Vanishing terms}

It will be important to know later that certain products vanish in $B_n^{\kappa}$. Somewhat surprisingly this can happen even when the total degree is small.

\begin{lem} \label{lem:deg1vanishing}
We have
\begin{equation*}
\psi_{\domtab{\wt{\eta}_1} \subqtab{\wt{\eta_1}}{\wt{\eta}}} \psi_{\subqtab{\wt{\eta_1}}{\wt{\eta}} \domtab{\wt{\eta}_1}}=\psi_{f_{\wt{\eta}}}^2 e(s_{f_{\wt{\eta}}}\domres{\wt{\eta}})=0 \text{.}
\end{equation*}
\end{lem}

\begin{proof}
From Proposition \ref{prop:psicrosses} it is clear that the first product above vanishes if and only if the second product vanishes. 
We expand the first product by pulling apart the double transposition of degree $2$ and rewriting as a difference of dotted strings. In the first term, the left string with its dot can be pulled all the way to the left, because the residues of all the strings that it passes through are distinct. In the second term, the dot on the right string can jump almost all the way to the left, slide down a string, and then make one final jump to the leftmost string. Dots on the left vanish in $B_n^{\kappa}$, so we are done. The diagrams below depict this process when $f_{\wt{\eta}}=4$.
\newlength{\llu}
\setlength{\llu}{.81in}
\begingroup
\allowdisplaybreaks
\begin{align*}
\begin{gathered}
  \includegraphics[width=1\llu]{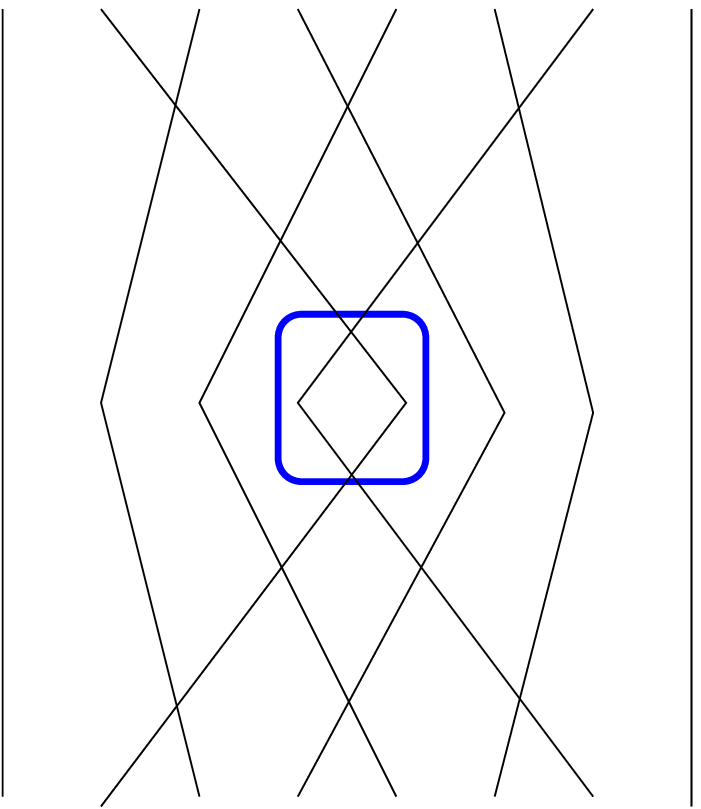}
\end{gathered}
\;\;\hspace{.1in} 
& = \;\;\hspace{.1in} 
\begin{gathered}
  \includegraphics[width=1\llu]{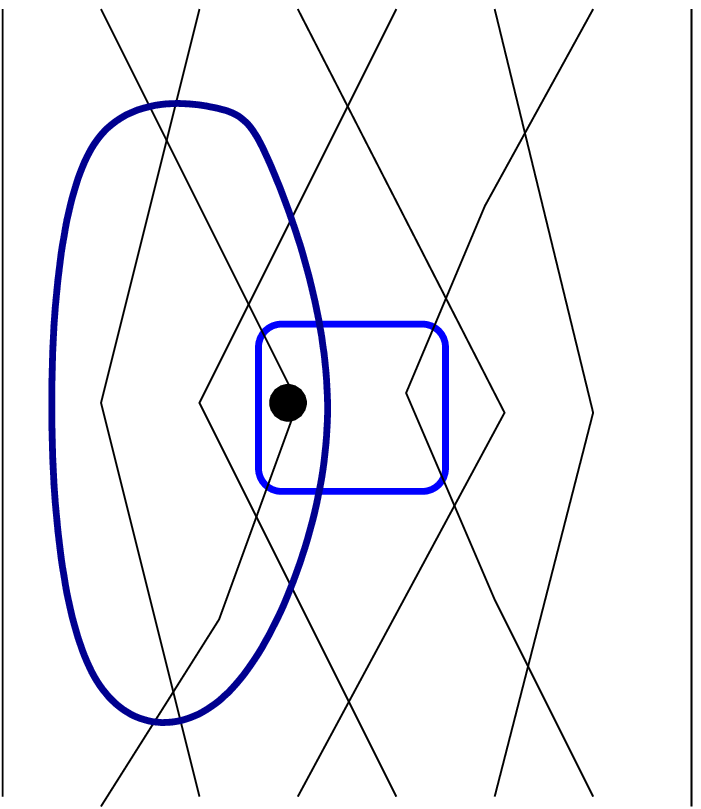}
\end{gathered}
\;\; - 
\begin{gathered}
  \includegraphics[width=1.1\llu]{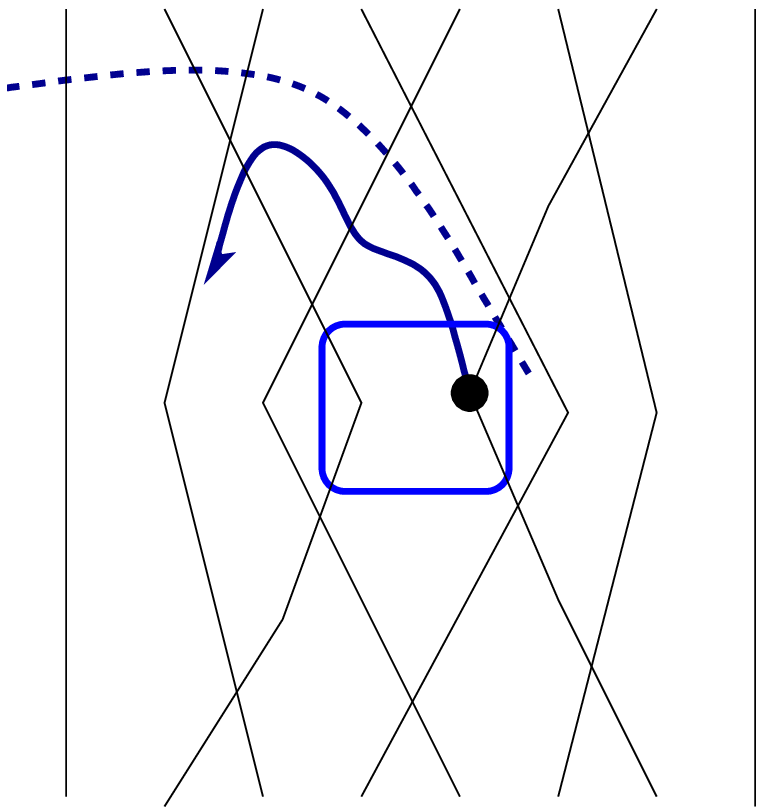}
\end{gathered}
\;\;\hspace{.1in} \\ 
\;\;\;\;\;\;\;\;\;\;\;\;\;\;\hspace{.1in} & =  \;\;\hspace{.1in} 
\begin{gathered}
  \includegraphics[width=1\llu]{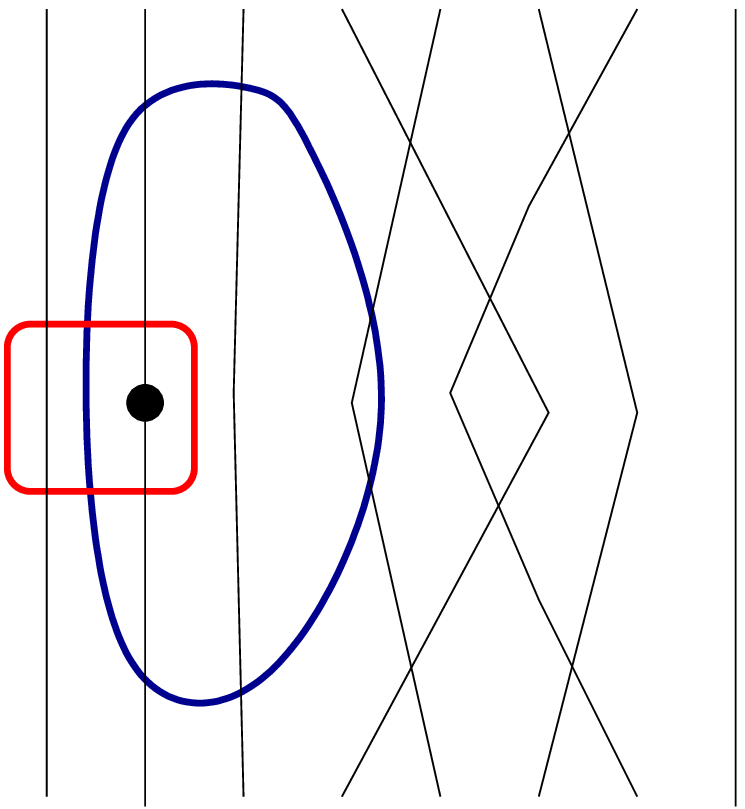}
\end{gathered}
\;\; - \;\; 
\begin{gathered}
  \includegraphics[width=1.0751\llu]{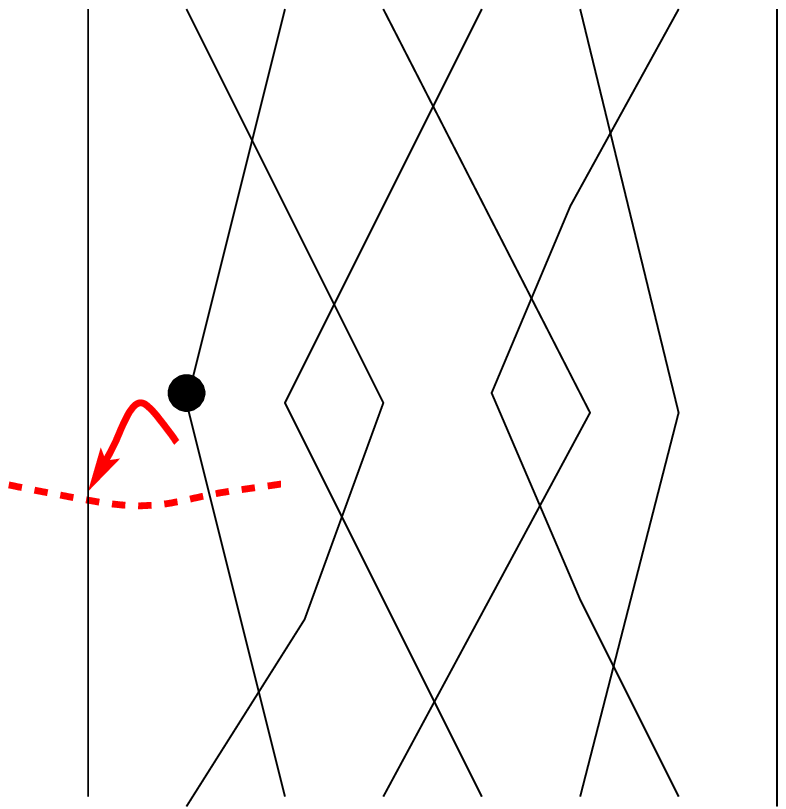}
\end{gathered}
\;\;\hspace{.1in}  \\ 
\;\;\;\;\;\;\;\;\;\;\;\;\;\; & = \;\;\hspace{.1in} 
\begin{gathered}
  \includegraphics[width=1.05\llu]{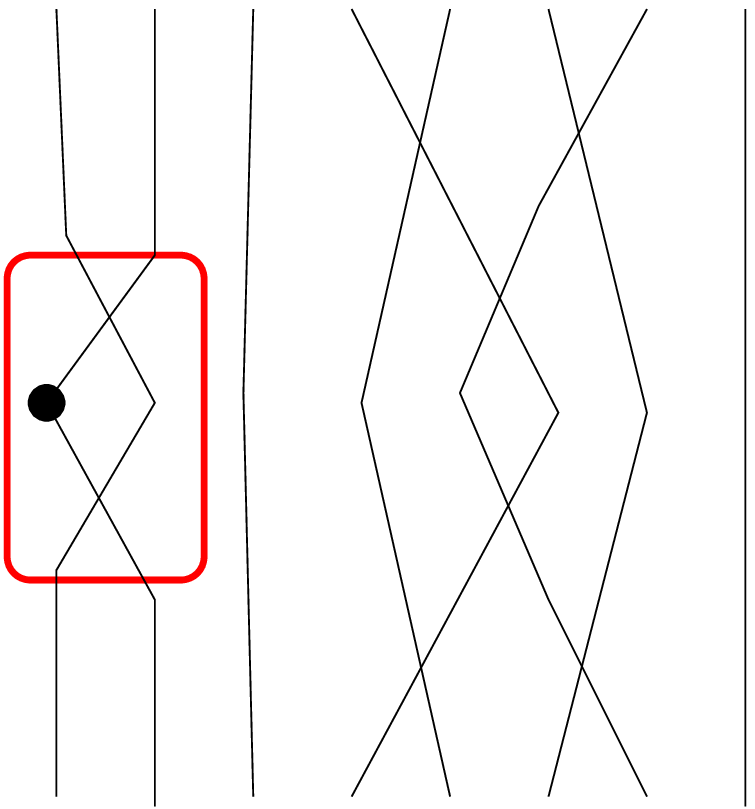}
\end{gathered}
\;\; - \;\; 
\begin{gathered}
  \includegraphics[width=1.0\llu]{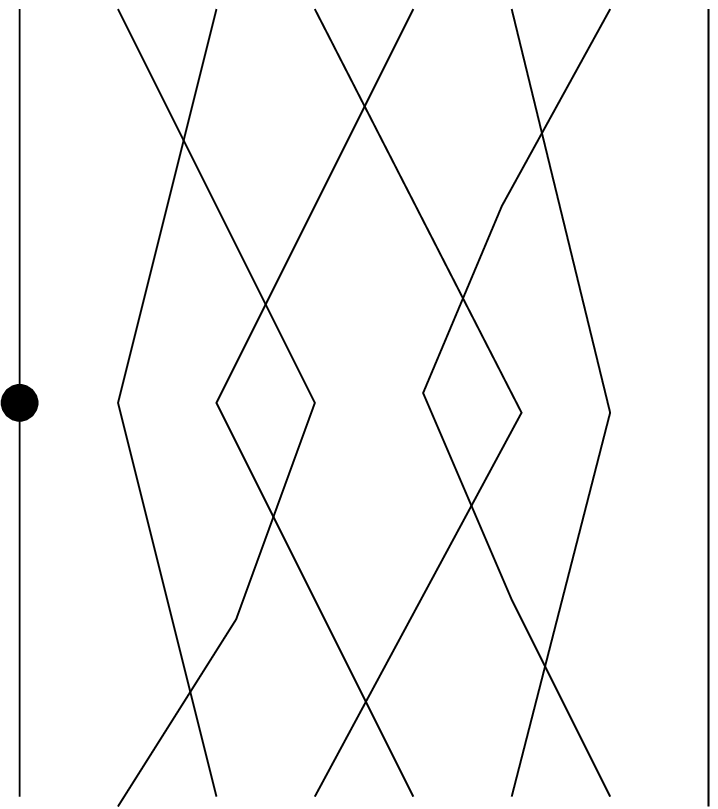}
\end{gathered} \\
\;\;\;\;\;\; & = \;\;\;\;\; 0-0 \text{.}
\end{align*}
\endgroup
\end{proof}


Another useful fact is that many cellular basis elements have a diamond as a factor, and thus vanish when multiplied by $\JW^{\wt{\eta}}$.

\begin{prop} \label{prop:diamondfactor}
Suppose $k<l \leq 2k$. Let $\tableau{t} \in \Std_{\wt{\eta}}(\wt{\eta}_l)$ with $\deg \tableau{t}=2k-l$. Then $\psi_{\tableau{t} \subqtab{\wt{\eta}_l}{\wt{\eta}}} \in U^{\wt{\eta}}_j B_n^{\kappa}$ for some $j \in \N$.
\end{prop}

\begin{proof}
Since $\deg \tableau{t}=2k-l<k<l$ it is clear that $\tableau{t} \neq \subqtab{\wt{\eta}_l}{\wt{\eta}}$. This means that the path corresponding to $\tableau{t}$ must diverge from that of $\subqtab{\wt{\eta}_l}{\wt{\eta}}$, by leaving the fundamental alcove early, before turning back at some wall after $(f_{\wt{\eta}}+je)$ steps for some $j<l$. 
The skew tableau corresponding to the $e$ steps before and after this turn-back point looks either like
\begin{equation*}
\includegraphics[scale=0.5]{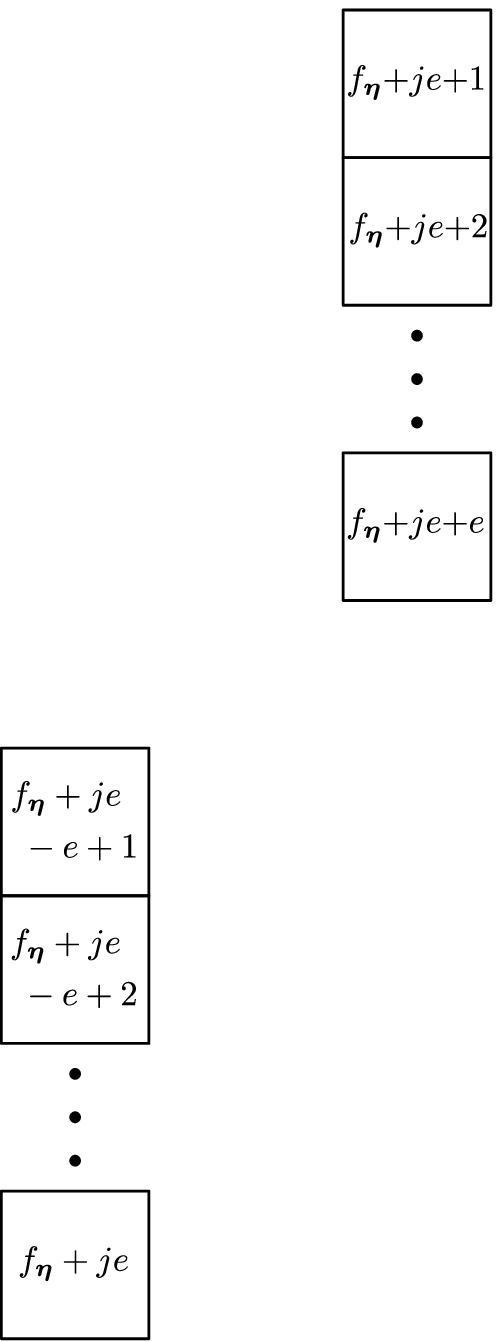}
\end{equation*}
or like
\begin{equation*}
\includegraphics[scale=0.5]{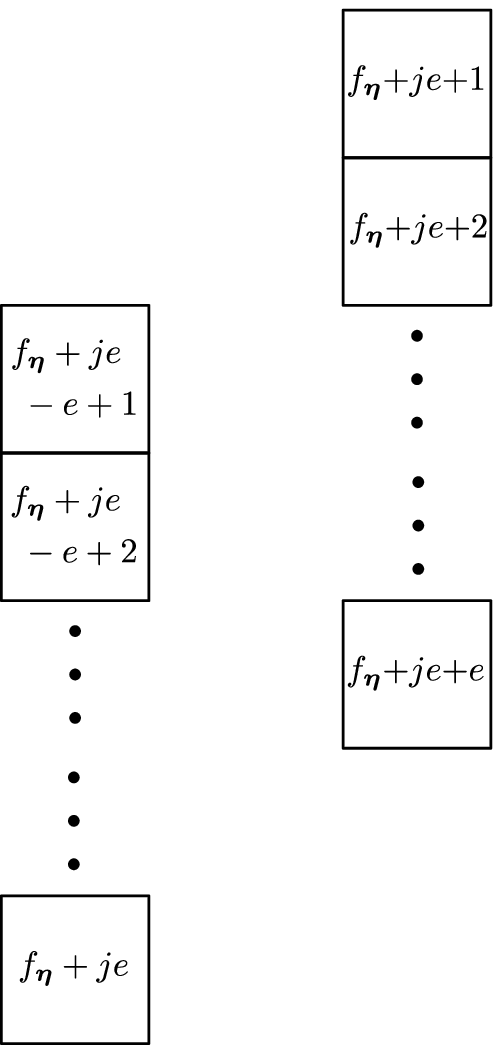}
\end{equation*}
or their mirror images.

Suppose we are in the first case, e.g.~with $\tableau{t}$ corresponding to the black path below:
\begin{center}
\includegraphics[scale=0.8]{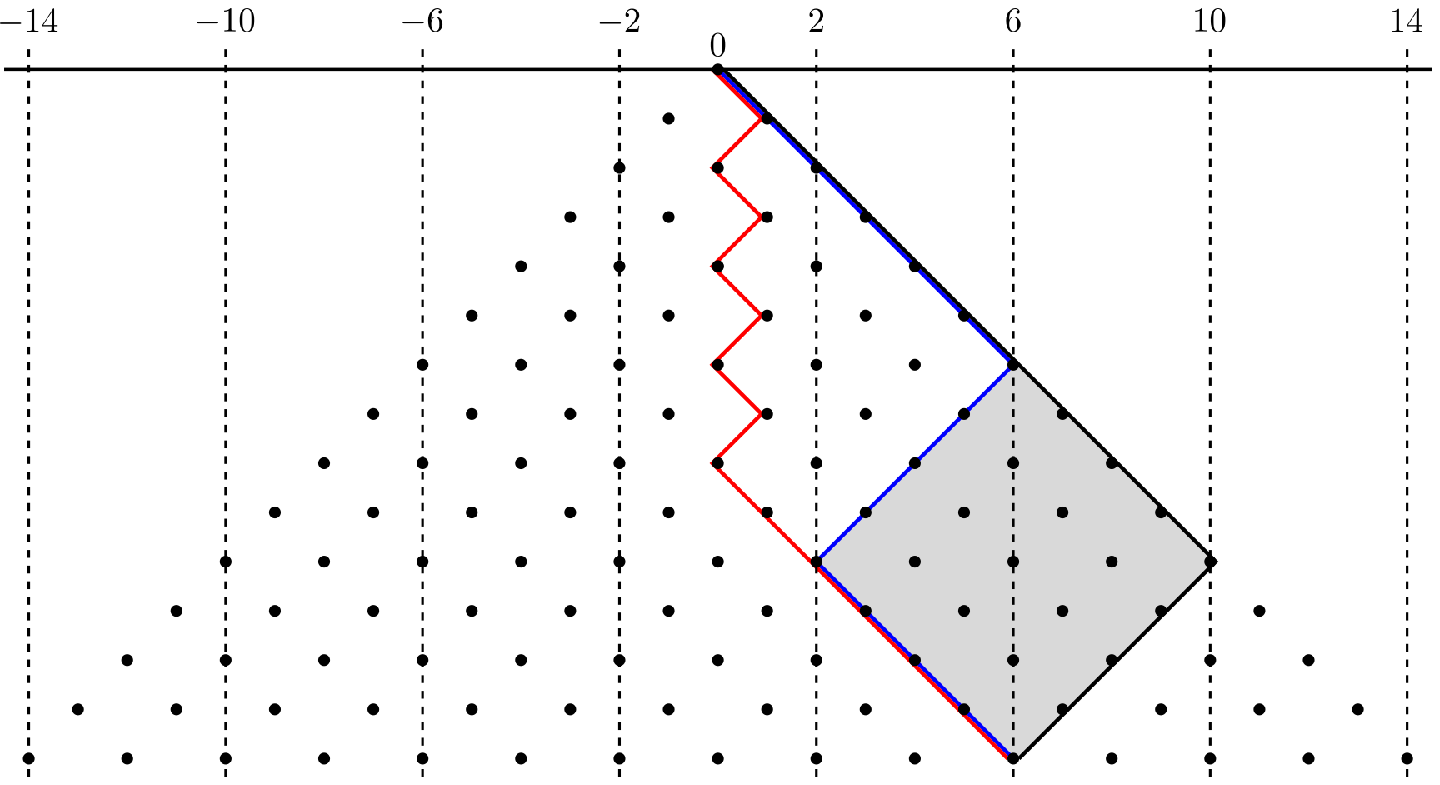}
\end{center}
Let $u \in S_n$ be the permutation which swaps the one-column partitions in the skew tableau above and fixes everything else.
As these columns contain two adjacent subsets of $e$ entries each, $u$ corresponds to a diamond permutation.
Write $\underline{u}$ for standard the reduced expression for $u$ coming from Definition~\ref{de:diamond}.
Applying $u$ to $\tableau{t}$ yields a new tableau $\tableau{t}'$ with the same residue (corresponding to the blue path above). 
In the path picture, it is clear that the region in grey bounded by $\tableau{t}'$ and $\tableau{t}$ is entirely contained within the region bounded by $\domtab{\wt{\eta}_l}$ (the red path) and $\tableau{t}$.
This means that $\underline{d}_{\tableau{t}}=\underline{u}\underline{d}_{\tableau{t}'}$ as reduced expressions (see e.g.~\cite[Algorithm~5.2]{libedinskyplaza}).
Thus 
\begin{equation*}
\psi_{\tableau{t} \subqtab{\wt{\eta}_l}{\wt{\eta}}}=\psi_{\underline{u}} \psi_{\tableau{t}' \subqtab{\wt{\eta}_l}{\wt{\eta}}}=U^{\wt{\eta}}_j \psi_{\tableau{t}' \subqtab{\wt{\eta}_l}{\wt{\eta}}} \in U^{\wt{\eta}}_j B_n^{\kappa}
\end{equation*}
which proves the result.


Now suppose we are in the second case, e.g.~with $\tableau{t}$ corresponding to the black path below:
\begin{center}
\includegraphics[scale=0.8]{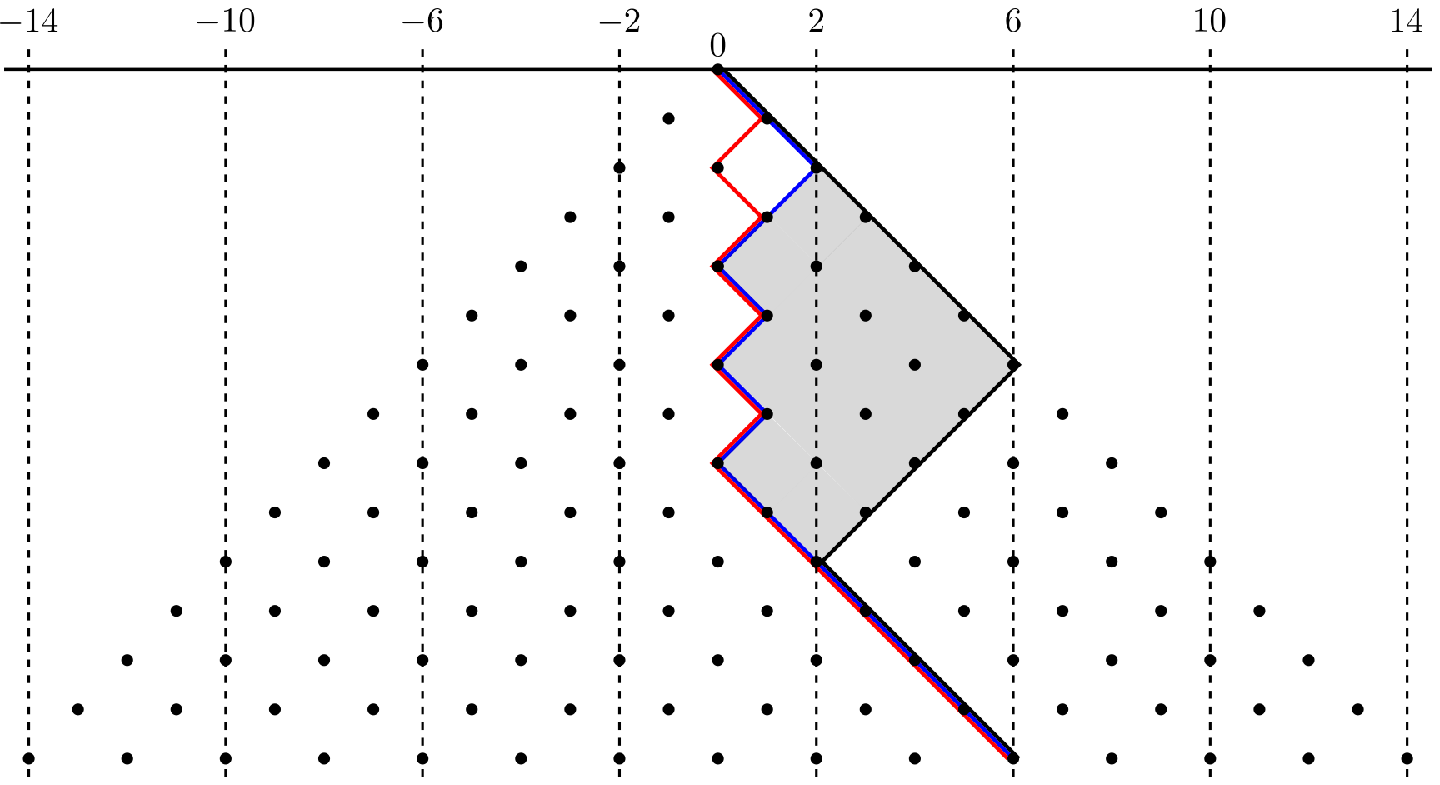}
\end{center}
Let $u'$ be the following permutation in blue
$$
\includegraphics[scale=0.5]{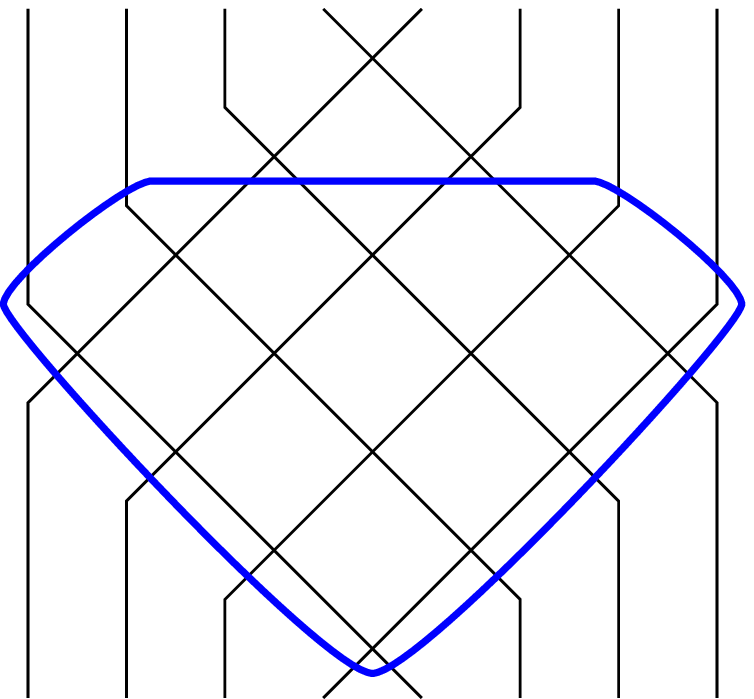}
$$
which we call a `cut diamond' permutation, corresponding to the first
\begin{equation*}
\begin{cases}
2e-f_{\wt{\eta}} & \text{if $j$ is even,} \\
e+f_{\wt{\eta}}-1 & \text{if $j$ is odd}
\end{cases}
\end{equation*}
layers of the diamond permutation centred at $f_{\wt{\eta}}+je$, and fix $\underline{u'}$ to be the corresponding reduced expression for $u'$.
Applying $u'$ to $\tableau{t}$ yields a new tableau $\tableau{t}'$ (corresponding to the blue path above), whose entries within the skew tableau are given by
\begin{equation*}
\begin{gathered}
\includegraphics[scale=0.5]{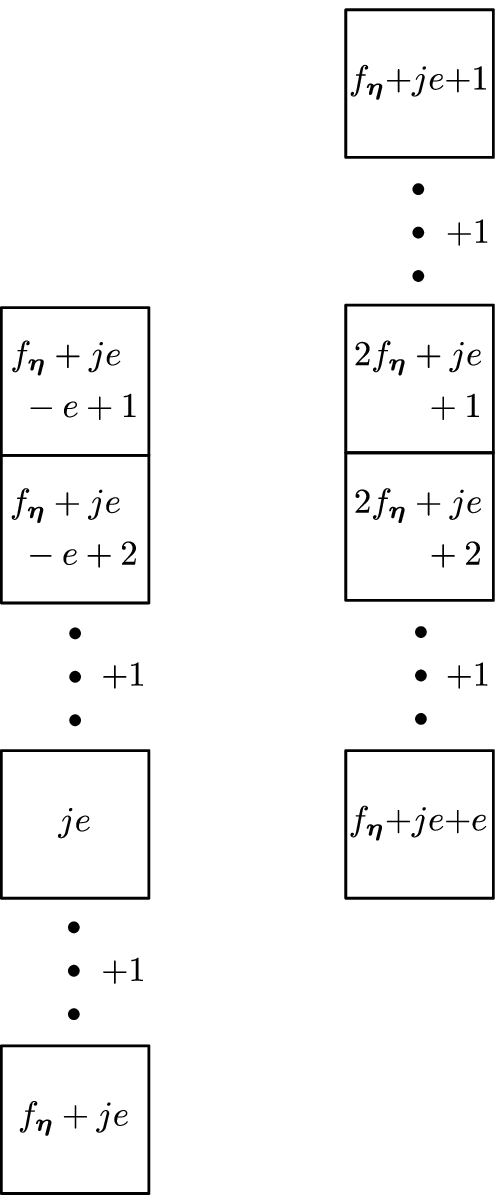}
\end{gathered} \quad \xmapsto{\substack{\text{first $e$}\\ \text{layers}}} \quad 
\begin{gathered}
\includegraphics[scale=0.5]{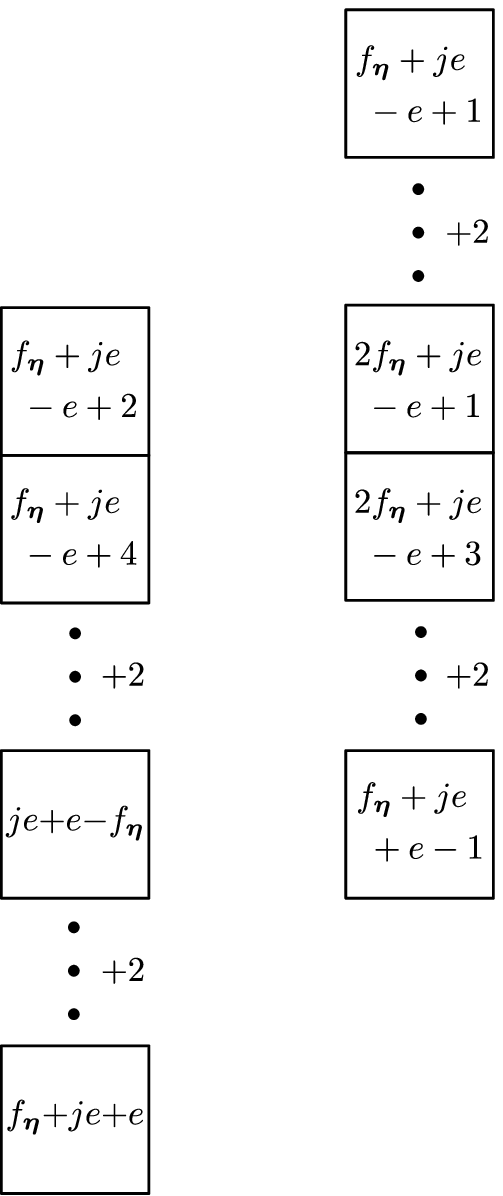}
\end{gathered} \quad \xmapsto{\substack{\text{remaining}\\ \text{layers}}} \quad
\begin{gathered}
\includegraphics[scale=0.5]{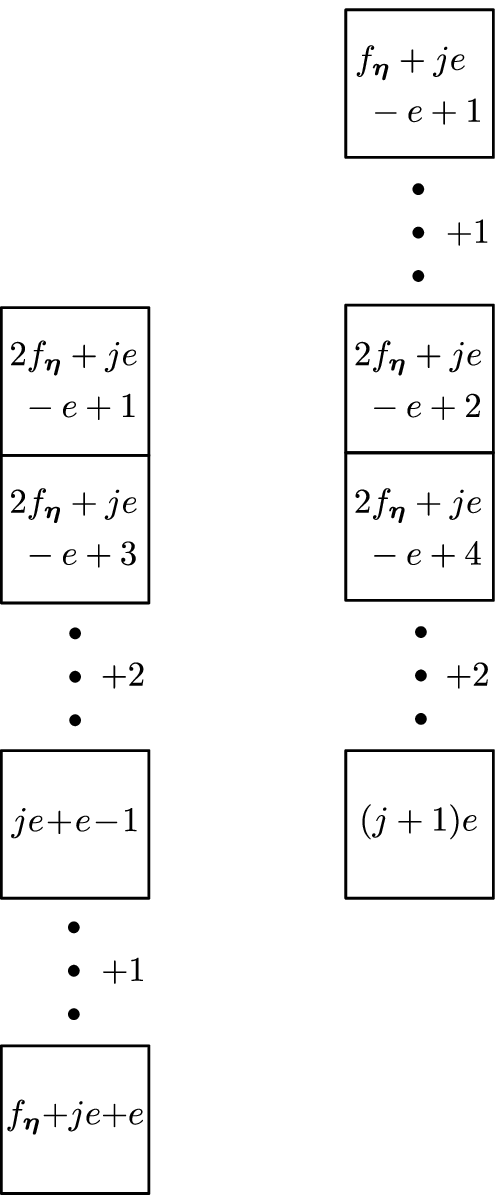}
\end{gathered}
\end{equation*}
where we have written the increments in the omitted boxes with $+1$ or $+2$.
Working in the path picture, we again note that the region in grey bounded by $\tableau{t}'$ and $\tableau{t}$ is entirely contained with the region bounded by $\domtab{\wt{\eta}_l}$ (the red path) and $\tableau{t}$.
As before this means that $\underline{d}_{\tableau{t}}=(\underline{u'})^{-1}\underline{d}_{\tableau{t}'}$ as reduced expressions.
From Proposition \ref{prop:psicrosses} we also know that $x_j \psi_{f_{\wt{\eta}}+je}e(\domres{\wt{\eta}})$ is a factor of $\psi_{\domtab{\wt{\eta}_l} \subqtab{\wt{\eta}_l}{\wt{\eta}}}$. Moreover, from the proof of this proposition, there is a reduced expression $\underline{u''}$ (which comes from the complement of the `cut diamond' in the KLR diagram above) for which $x_j \psi_{f_{\wt{\eta}+je}}e(\domres{\wt{\eta}})=\psi_{\underline{u''}}e(\domres{\wt{\eta}})$ and $\underline{u}=(\underline{u'})^{-1}\underline{u''}$. 
In fact, the construction of $\tableau{t}'$ ensures that $d_{\tableau{t}'}$ commutes with $u''$, because the support of $u''$ (i.e.~the elements not fixed by $u''$) is fixed by $d_{\tableau{t}'}$.
Combining everything together, we get
\begin{equation*}
\psi_{\tableau{t}\subqtab{\wt{\eta}_l}{\wt{\eta}}}=\psi_{(\underline{u'})^{-1}} \psi_{\tableau{t}' \domtab{\wt{\eta}_l}} \psi_{\domtab{\wt{\eta}_l} \subqtab{\wt{\eta}_l}{\wt{\eta}}} \in \psi_{(\underline{u'})^{-1}} \psi_{\underline{d}_{\tableau{t}'}} \psi_{\underline{u''}} e(\domres{\wt{\eta}})B_n^{\kappa}=\psi_{(\underline{u'})^{-1}} \psi_{\underline{u''}} e(\domres{\wt{\eta}})\psi_{\underline{d}_{\tableau{t}'}} B_n^{\kappa} \subseteq U^{\wt{\eta}}_j B_n^{\kappa} \text{.}
\end{equation*}
\end{proof}

The next result identifies possible candidates for generators of the socle of $P(\wt{\eta})$.

\begin{lem} \label{lem:possible-socle}
If $\soc P(\wt{\eta})$ contains a copy of $L(\wt{\eta})\langle 2k\rangle$ for some $k \geq 0$, then it must be the subspace
\begin{equation*}
\field \JW^{\wt{\eta}} \psi_{\subqtab{\wt{\eta}_k}{\wt{\eta}} \subqtab{\wt{\eta}_k}{\wt{\eta}}} \JW^{\wt{\eta}}
\end{equation*}
\end{lem}

\begin{proof}
$L(\wt{\eta})$ is $1$-dimensional, so it restricts to the unique $1$-dimensional irreducible representation of the Temperley--Lieb subalgebra. This means $\JW^{\wt{\eta}}$ acts on it as the identity, and thus any submodule isomorphic to $L(\wt{\eta})\langle 2k\rangle$ lies in the degree $2k$ part of $\JW^{\wt{\eta}}B_n^{\kappa} \JW^{\wt{\eta}}$. 
From Corollary~\ref{cor:projsubquot}, all cellular basis elements $\psi \neq \psi_{\subqtab{\wt{\eta}_k}{\wt{\eta}} \subqtab{\wt{\eta}_k}{\wt{\eta}}}$ with top and bottom residues $\domres{\wt{\eta}}$ and degree $2k$ which do not vanish in $P(\wt{\eta})=B_n^{\kappa}e(\domres{\wt{\eta}})\JW^{\wt{\eta}}$ have the form $\psi=\psi_{\tableau{t} \subqtab{\wt{\eta}_l}{\wt{\eta}}}$ for some integer $k<l\leq 2k$ with $\deg \tableau{t}=2k-l$.
By Proposition~\ref{prop:diamondfactor} such basis elements factor as $\psi=U^{\wt{\eta}}_j x$ for some $j \in \N$ and $x \in B_n^{\kappa}$. Since $\JW^{\wt{\eta}} U^{\wt{\eta}}_j=0$ the result follows.
\end{proof}

\subsection{Diamond simplification}

By Lemma~\ref{lem:possible-socle}, determining the socle of $P(\wt{\eta})$ will necessitate calculations involving $\JW^{\wt{\eta}}$. The next few lemmas give some methods for reducing the workload by eliminating diamonds.

\begin{lem} \label{lem:crossU1cross}
For all $k$ we have
\begin{equation*}
\psi_{f_{\wt{\eta}}+(k-1)e} U^{\wt{\eta}}_k \psi_{f_{\wt{\eta}}+(k-1)e}=\pm \psi_{f_{\wt{\eta}}+ke}^2 e(s_{f_{\wt{\eta}}+(k-1)e}\domres{\wt{\eta}}) \text{.}
\end{equation*}
\end{lem}

\begin{proof}
Apply \cite[Lemma 5.16]{libedinskyplaza} several times across the diamond. The remaining transpositions are all of degree $0$ except for the degree $1$ transpositions at the top and bottom. The degree $0$ transpositions cancel out and the result follows. The diagrams below depict what happens when $e=6$.
\begin{align*}
\begin{gathered}
\includegraphics[scale=0.25]{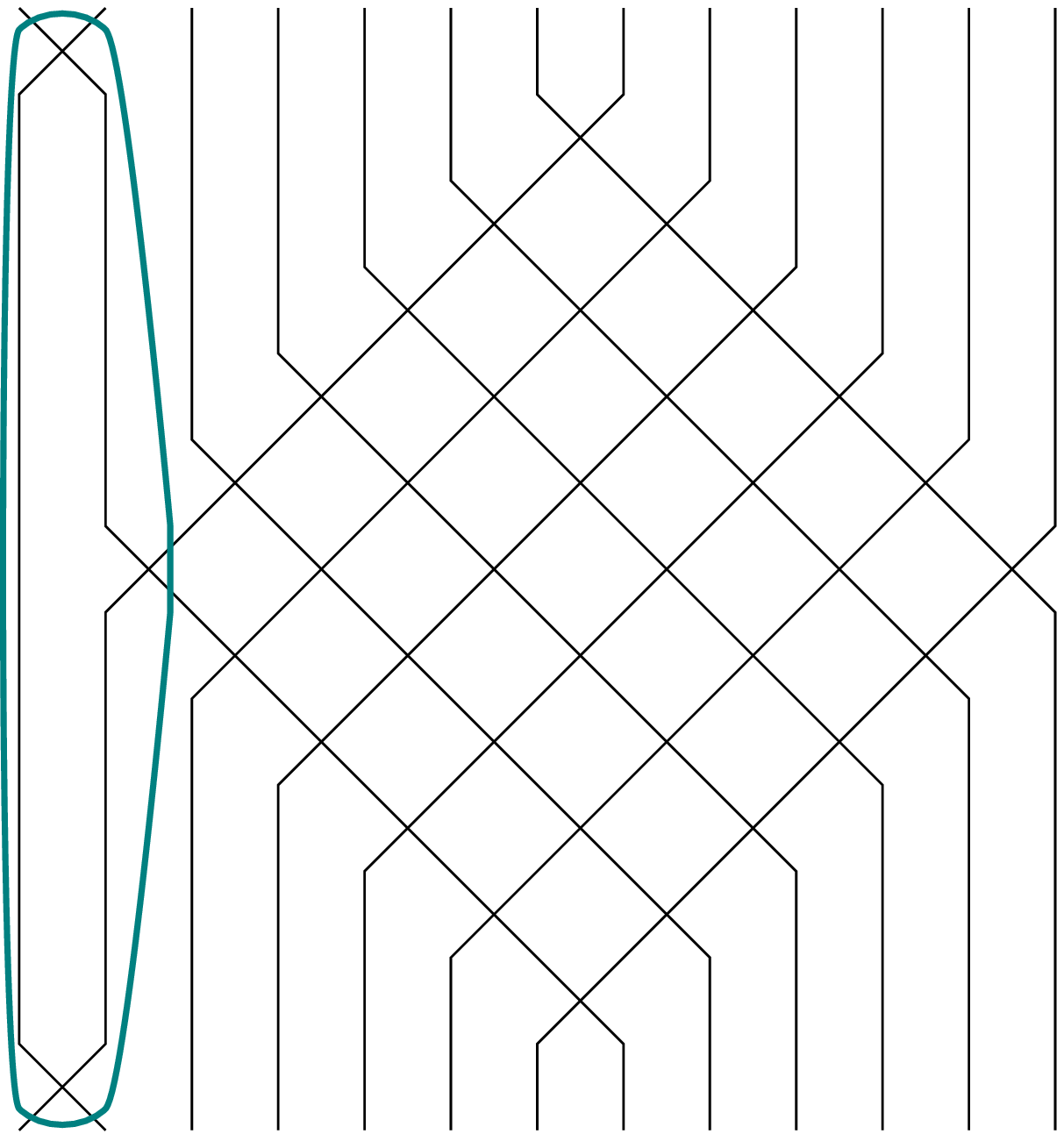}
\end{gathered}& =-\begin{gathered}
\includegraphics[scale=0.25]{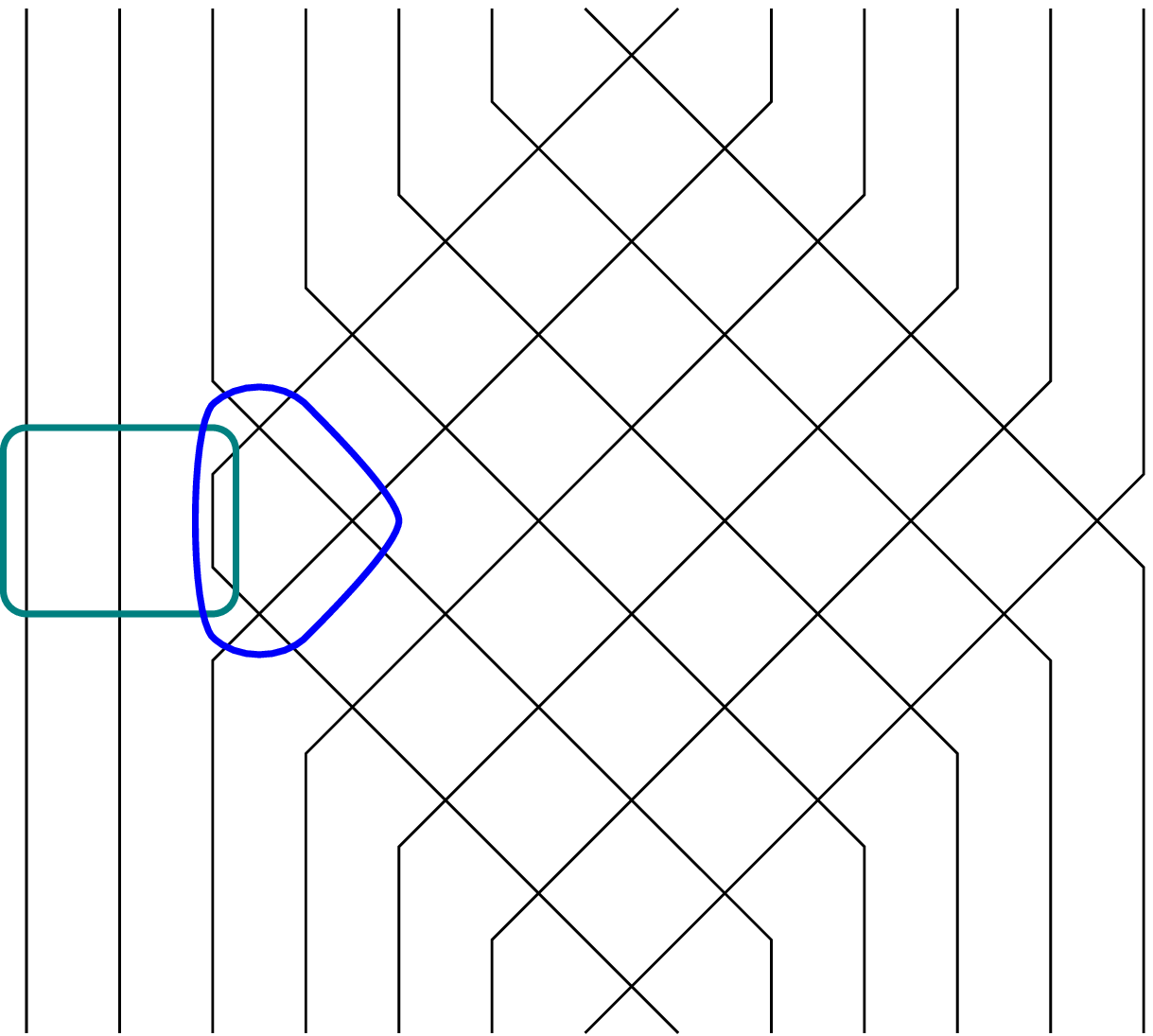}
\end{gathered}=\begin{gathered}
\includegraphics[scale=0.25]{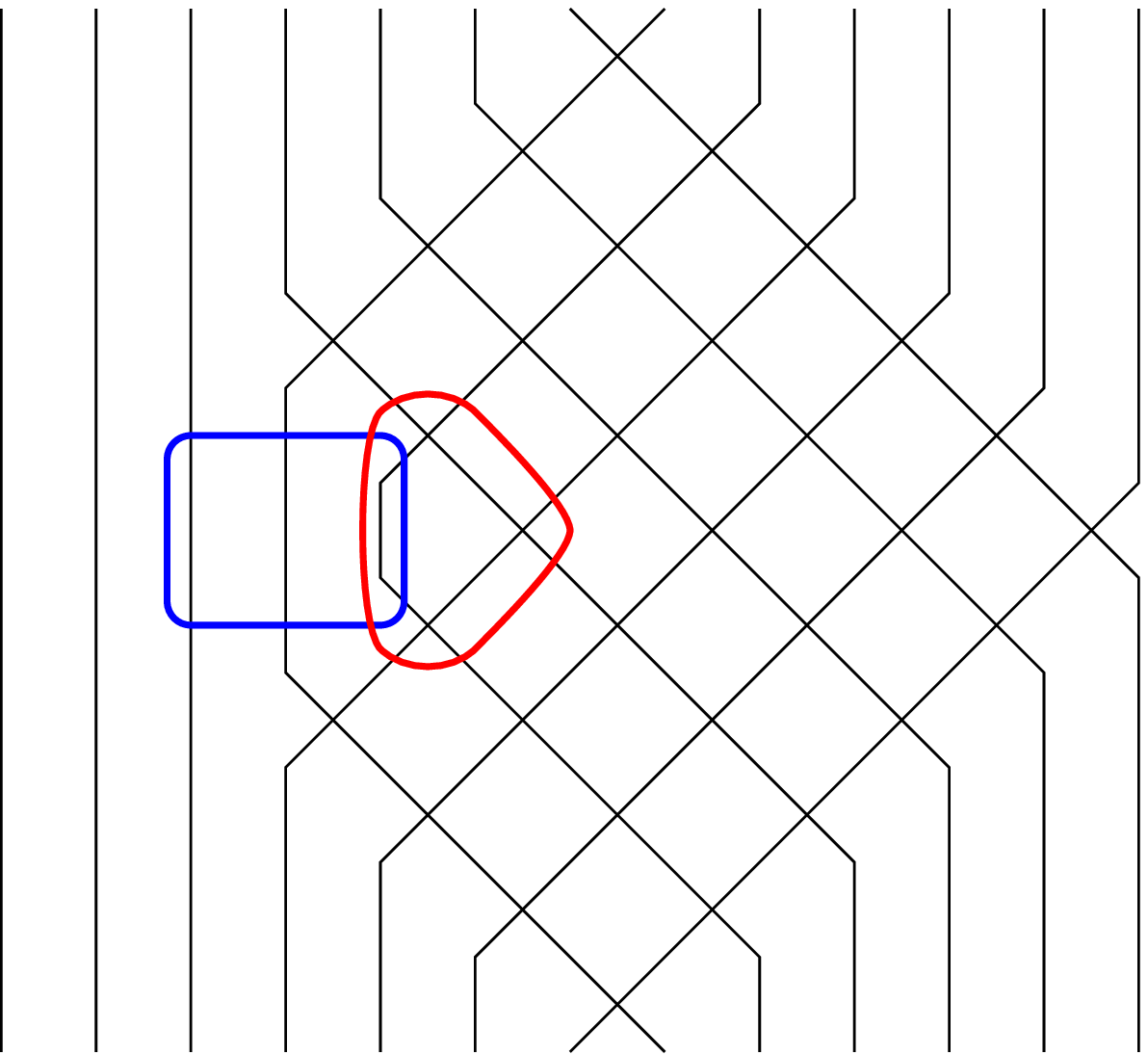}
\end{gathered} \\
& =-\begin{gathered}
\includegraphics[scale=0.25]{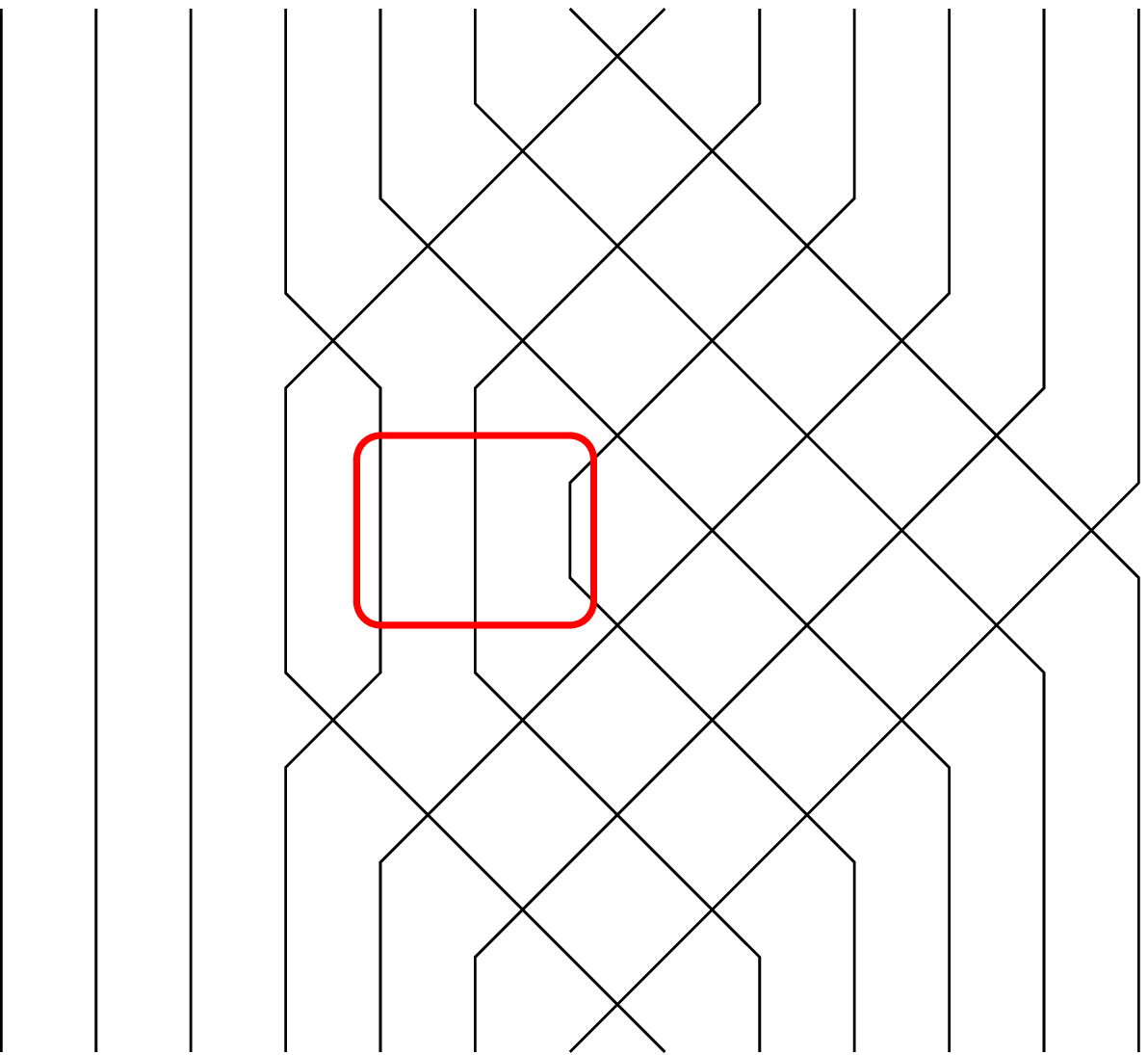}
\end{gathered}=\dotsb \\
\dotsb & =(-1)^e \begin{gathered}
\includegraphics[scale=0.25]{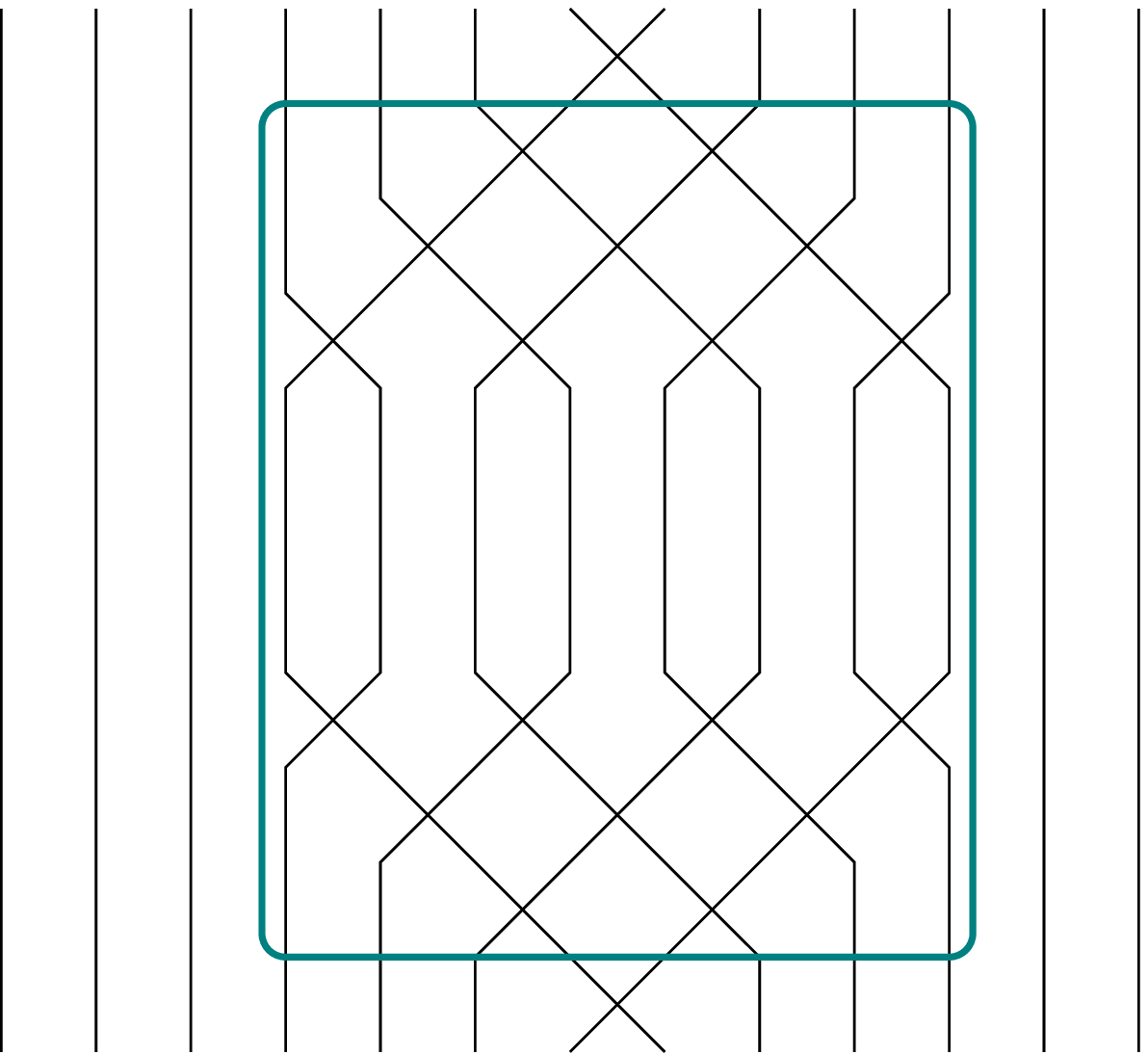}
\end{gathered}=(-1)^e\begin{gathered}
\includegraphics[scale=0.25]{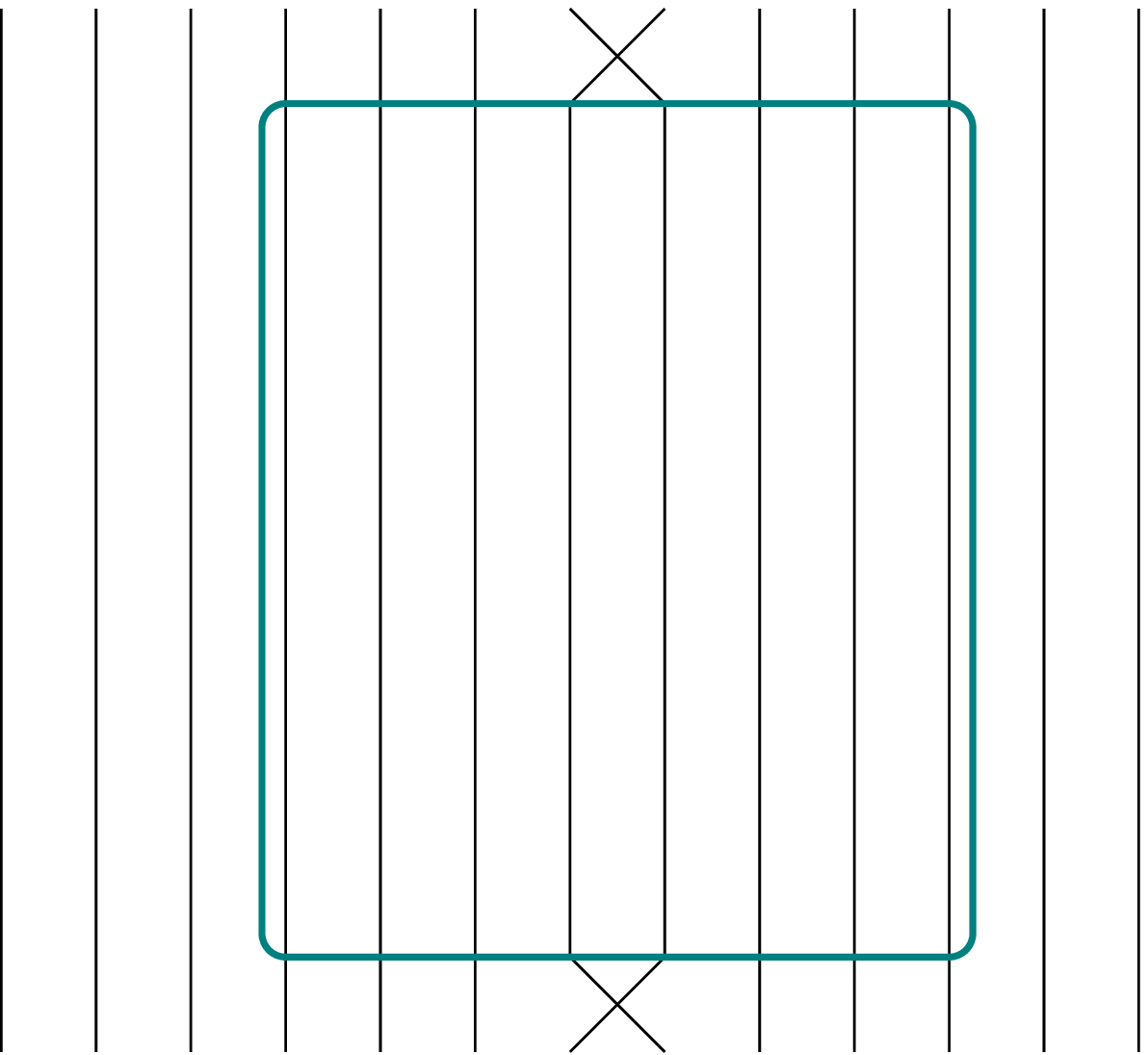}
\end{gathered}
\end{align*}
\end{proof}

\begin{lem} \label{lem:Uoffcenterdoublecross}
For all $k$ we have
\begin{equation*}
U^{\wt{\eta}}_k U^{\wt{\eta}}_{k-1} \psi_{f_{\wt{\eta}+ke}}^2 = \pm U^{\wt{\eta}}_k \psi_{f_{\wt{\eta}+(k-1)e}}^2 \text{.}
\end{equation*}
\end{lem}

\begin{proof}
This follows immediately from a variant of Lemma \ref{lem:crossU1cross}, which is proved in the same way.
$$
\begin{gathered}
\includegraphics[scale=0.5]{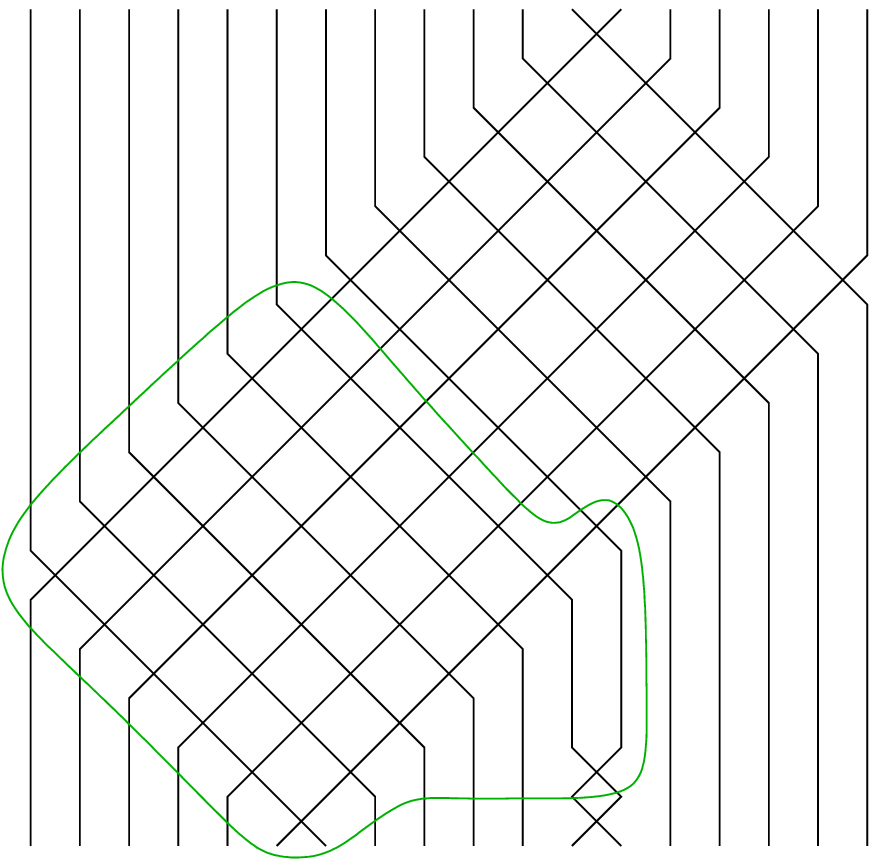}
\end{gathered}
= (-1)^e
\begin{gathered}
\includegraphics[scale=0.5]{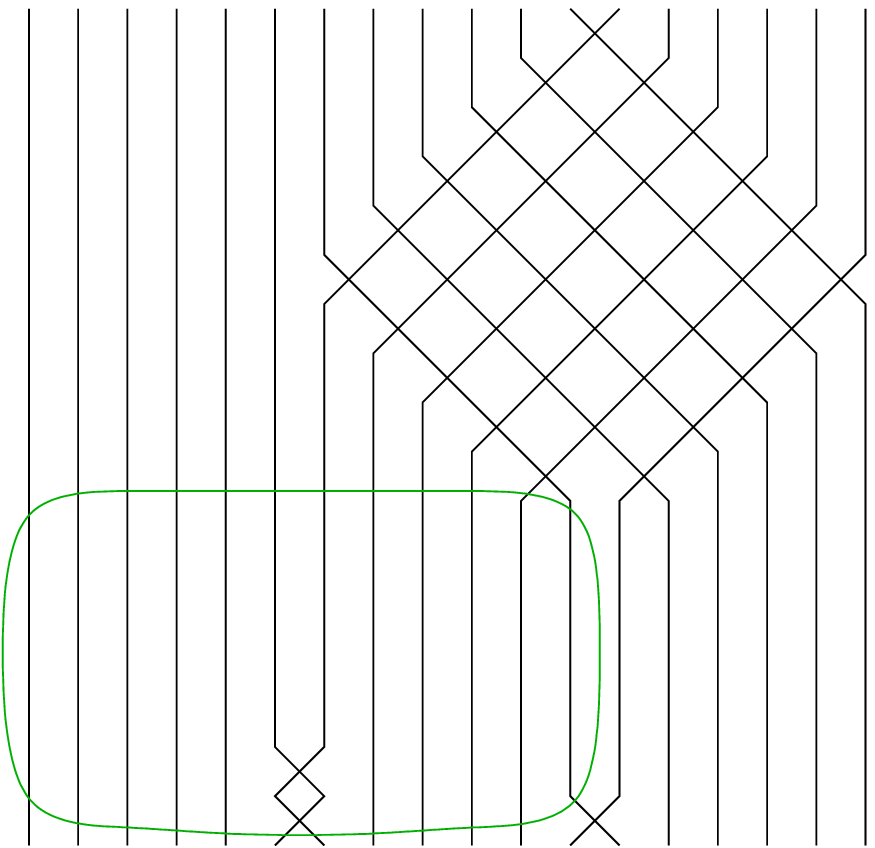}
\end{gathered}
$$
\end{proof}

\begin{lem} \label{lem:Ucenterdoublecross}
For all $1<k<m$ we have 
\begin{equation*}
U^{\wt{\eta}}_{k-1} \psi_{\subqtab{\wt{\eta}_k}{\wt{\eta}}\subqtab{\wt{\eta}_k}{\wt{\eta}}}=0 \text{.}
\end{equation*}
\end{lem}

\begin{proof}
Use Proposition \ref{prop:psicrosses} to rewrite $\psi_{\subqtab{\wt{\eta}_k}{\wt{\eta}}\subqtab{\wt{\eta}_k}{\wt{\eta}}}$ as a product of double transpositions. Expand the rightmost double transposition as a difference of dotted strings. First we show that these dots can `migrate' leftwards until they lie on top of the next pair of transpositions. In the first term, the dot on the left string can jump until it is on the right string above this double transposition. In the second term, the dot on the right string can slide along the southwest border of the diamond, jump left one string and slide until it is in place on the left string above the double transposition.
\begin{align*}
\begin{gathered}
\includegraphics[scale=0.5]{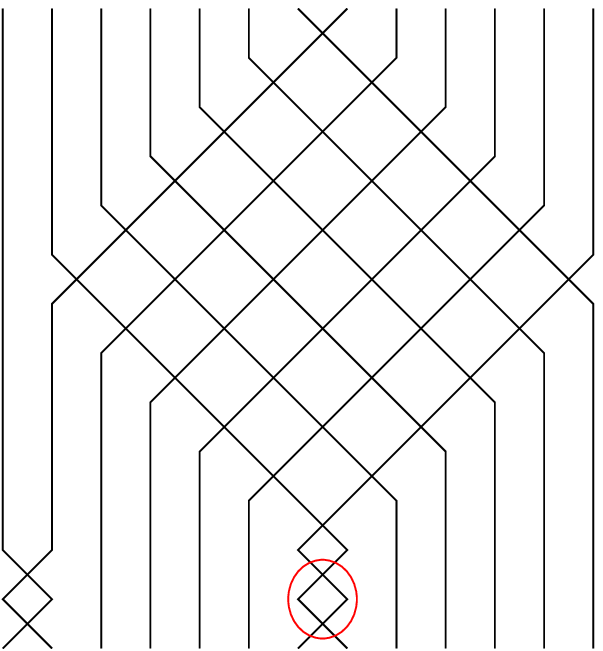}
\end{gathered}
&=
\begin{gathered}
\includegraphics[scale=0.5]{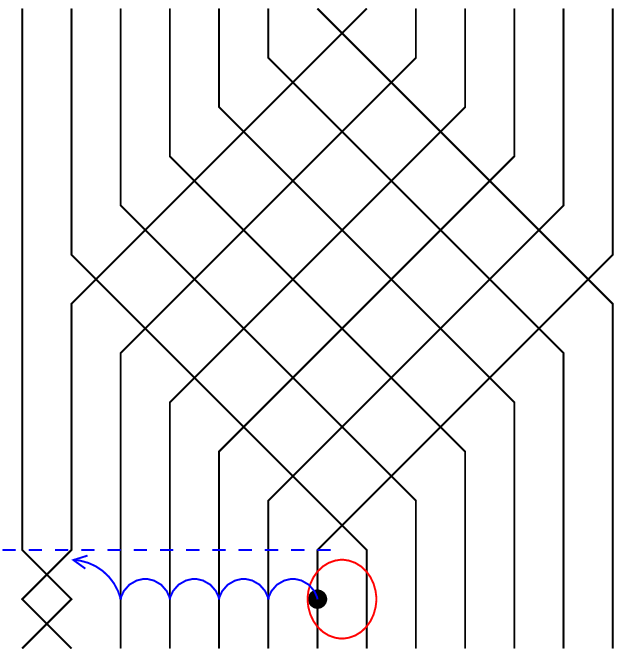}
\end{gathered}
-
\begin{gathered}
\includegraphics[scale=0.5]{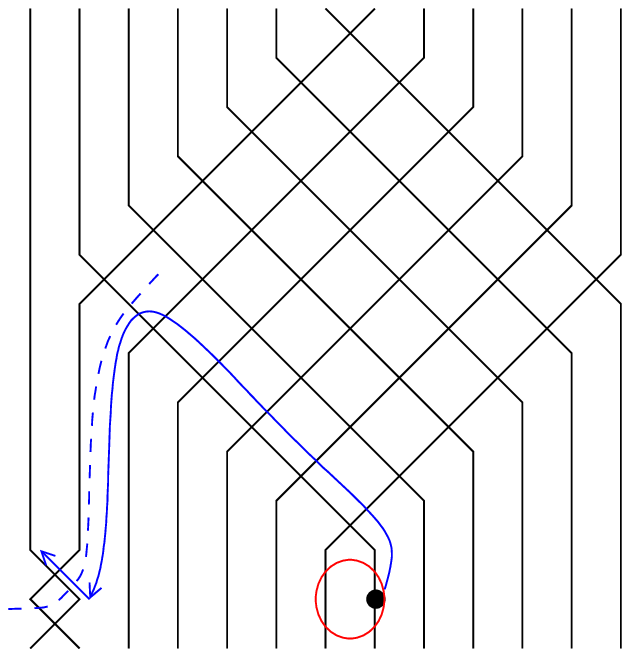}
\end{gathered}\\
&=
\begin{gathered}
\includegraphics[scale=0.5]{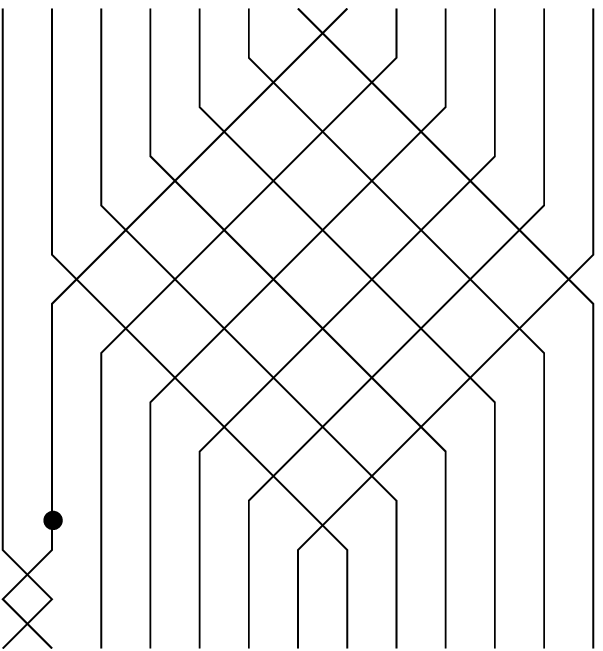}
\end{gathered}
-
\begin{gathered}
\includegraphics[scale=0.5]{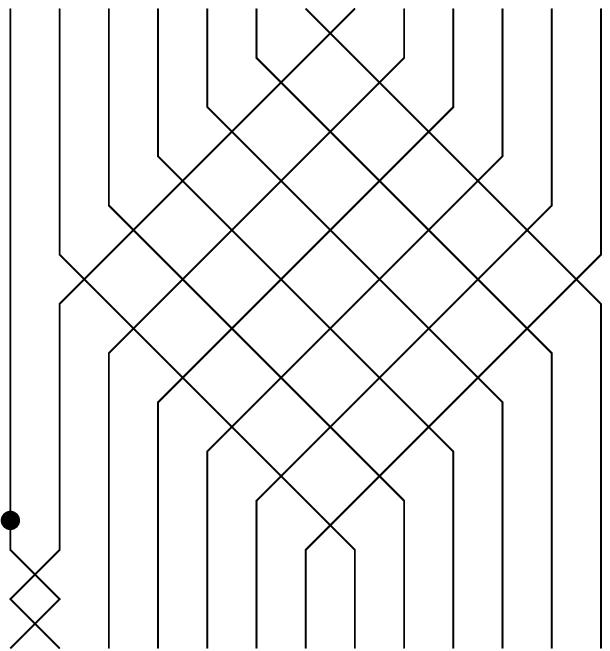}
\end{gathered}
\end{align*}

Next, we show we can continue this migration process leftwards without the diamond. As before, the dot on the left string above the double transposition can jump several strings leftwards until it is on the right string above the next double transposition. For the dot on the right string, we replace the both pairs of transpositions with pairs of maximally sized triangles, as seen in the proof of Proposition \ref{prop:psicrosses}. This dot then slides southwest along its string, jumps one string, and slides northwest until it is in the correct position.
\begin{align*}
\begin{gathered}
\includegraphics[scale=0.45]{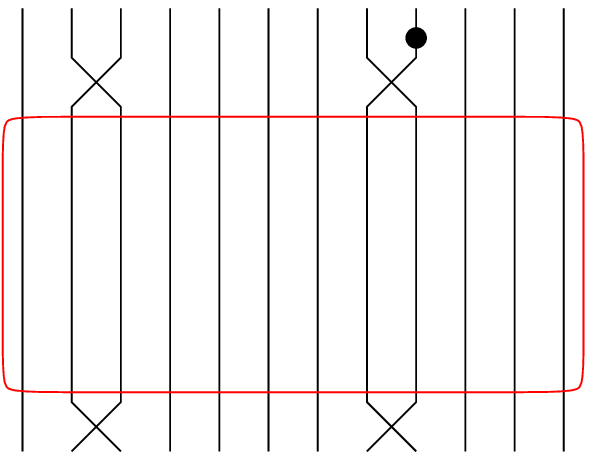}
\end{gathered}
-
\begin{gathered}
\includegraphics[scale=0.45]{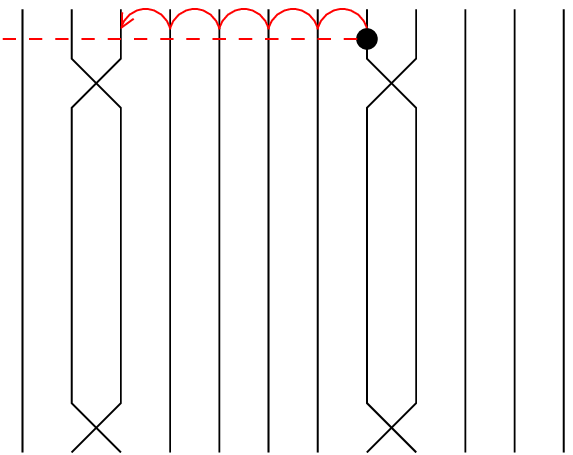}
\end{gathered} 
&=
\begin{gathered}
\includegraphics[scale=0.45]{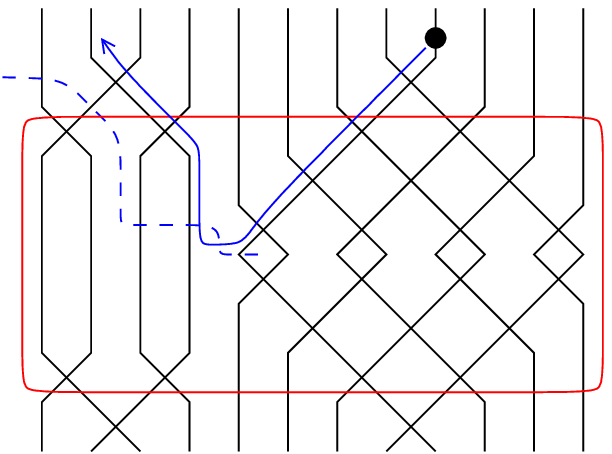}
\end{gathered}
- 
\begin{gathered}
\includegraphics[scale=0.45]{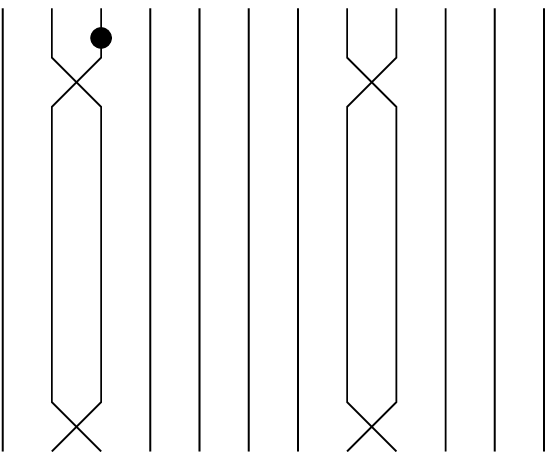}
\end{gathered}\\
&=
\begin{gathered}
\includegraphics[scale=0.45]{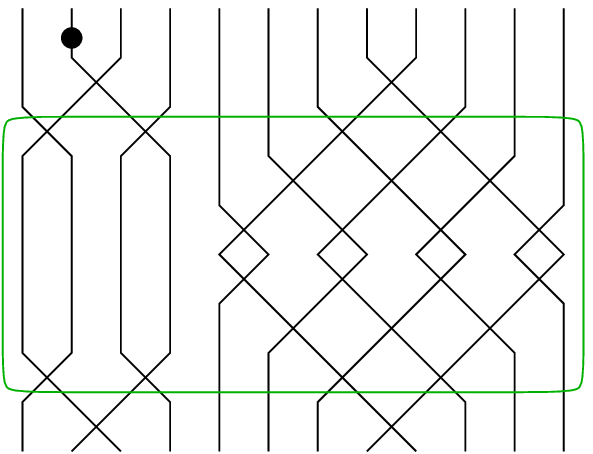}
\end{gathered}
- 
\begin{gathered}
\includegraphics[scale=0.45]{fig/Uoffdb2D.eps}
\end{gathered}\\
&=
\begin{gathered}
\includegraphics[scale=0.45]{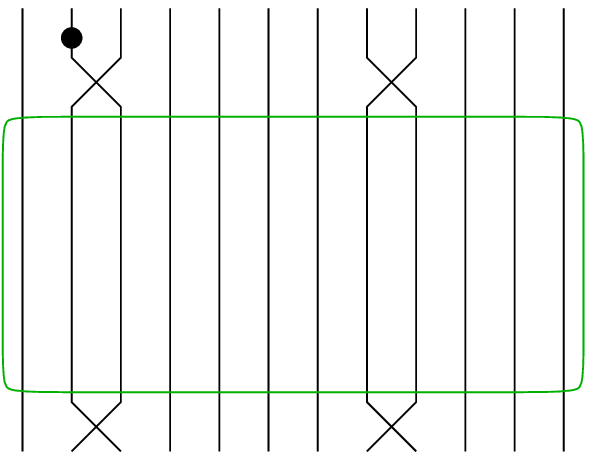}
\end{gathered}
- 
\begin{gathered}
\includegraphics[scale=0.45]{fig/Uoffdb2D.eps}
\end{gathered}
\end{align*}
Note that in both of the figures above we are only drawing a portion of the complete diagram.

Finally we end up with a difference of dotted strings for the leftmost double transposition. But we can replace this difference with another double transposition. Applying Lemma \ref{lem:deg1vanishing} gives the result.
\end{proof}

\subsection{Socle calculation}

We pool together our previous results into one grand calculation to identify the socle of $P(\wt{\eta})$. The heart of the argument is to show that certain products of $\JW^{\wt{\eta}}$ with cellular basis elements do not vanish in $B_n^{\kappa}$. This is potentially extremely difficult, as the number of summands when $\JW^{\wt{\eta}}$ is written in the standard monomial basis grows very quickly. Thankfully many of these monomials end up vanishing in the product. For $r \leq s$ write $U^{\wt{\eta}}_{r,s}=U^{\wt{\eta}}_r U^{\wt{\eta}}_{r+1} \dotsm U^{\wt{\eta}}_s$. First, we identify a non-vanishing monomial in the product.

\begin{thm} \label{thm:U_1k-nonvanishing}
Let $r \leq s$. If
\begin{equation*}
\psi_{\domtab{\wt{\eta}_1} \subqtab{\wt{\eta}_1}{\wt{\eta}}} U^{\wt{\eta}}_{r,s} \psi_{\subqtab{\wt{\eta}_k}{\wt{\eta}} \subqtab{\wt{\eta}_k}{\wt{\eta}}} \JW^{\wt{\eta}} \neq 0
\end{equation*}
then $(r,s)=(1,k)$. In this case, we have
\begin{equation*}
\psi_{\domtab{\wt{\eta}_1} \subqtab{\wt{\eta}_1}{\wt{\eta}}} U^{\wt{\eta}}_{1,k} \psi_{\subqtab{\wt{\eta}_k}{\wt{\eta}} \subqtab{\wt{\eta}_k}{\wt{\eta}}} \JW^{\wt{\eta}}= \pm \psi_{\subqtab{\wt{\eta}_{k+1}}{\wt{\eta}_1} \subqtab{\wt{\eta}_{k+1}}{\wt{\eta}}}  \JW^{\wt{\eta}} \text{.}
\end{equation*}
\end{thm}

\begin{proof}
When $r>1$, we have
\begin{equation}
\psi_{\domtab{\wt{\eta}_1} \subqtab{\wt{\eta}_1}{\wt{\eta}}} U^{\wt{\eta}}_r=U^{\wt{\eta}_1}_{r-1} \psi_{\domtab{\wt{\eta}_1} \subqtab{\wt{\eta}_1}{\wt{\eta}}} \label{eq:Udeg1commute}
\end{equation}
by Proposition \ref{prop:psicrosses}. Similarly when $r>k$, we have
\begin{equation}
U^{\wt{\eta}}_r  \psi_{\subqtab{\wt{\eta}_k}{\wt{\eta}} \subqtab{\wt{\eta}_k}{\wt{\eta}}}=\psi_{\subqtab{\wt{\eta}_k}{\wt{\eta}} \subqtab{\wt{\eta}_k}{\wt{\eta}}} U^{\wt{\eta}}_r \text{.} \label{eq:Upsicommute}
\end{equation}

When $1<r \leq s$ we have
\begin{align*}
\psi_{\domtab{\wt{\eta}_1} \subqtab{\wt{\eta}_1}{\wt{\eta}}} U^{\wt{\eta}}_{r,s} \psi_{\subqtab{\wt{\eta}_k}{\wt{\eta}} \subqtab{\wt{\eta}_k}{\wt{\eta}}} & =U^{\wt{\eta}_1}_{r-1,s-1} \psi_{\domtab{\wt{\eta}_1} \subqtab{\wt{\eta}_1}{\wt{\eta}}} \psi_{\subqtab{\wt{\eta}_k}{\wt{\eta}} \subqtab{\wt{\eta}_k}{\wt{\eta}}} \\
& =U^{\wt{\eta}_1}_{r-1,s-1} \psi_{\domtab{\wt{\eta}_1} \subqtab{\wt{\eta}_1}{\wt{\eta}}} \psi_{f_{\wt{\eta}}+e}^2 \psi_{f_{\wt{\eta}}+2e}^2 \dotsm \psi_{f_{\wt{\eta}}+(k-1)e}^2 \\
& =U^{\wt{\eta}_1}_{r-1,s-1} \psi_{\domtab{\wt{\eta}_1} \subqtab{\wt{\eta}_1}{\wt{\eta}}} \psi_{\subqtab{\wt{\eta}_1}{\wt{\eta}} \subqtab{\wt{\eta}_1}{\wt{\eta}}} \psi_{f_{\wt{\eta}}+2e}^2 \dotsm \psi_{f_{\wt{\eta}}+(k-1)e}^2 \\
& =0
\end{align*}
using \eqref{eq:Udeg1commute}, Corollary \ref{cor:psidoublecrosses}, and Lemma \ref{lem:deg1vanishing}. Similarly when $r \leq s \leq k-1$ this expression vanishes by Corollary \ref{cor:psidoublecrosses} and Lemma \ref{lem:Ucenterdoublecross}. Finally
\begin{equation*}
U^{\wt{\eta}}_{r,s} \psi_{\subqtab{\wt{\eta}_k}{\wt{\eta}} \subqtab{\wt{\eta}_k}{\wt{\eta}}} \JW^{\wt{\eta}}=U^{\wt{\eta}}_{r,s-1} \psi_{\subqtab{\wt{\eta}_k}{\wt{\eta}} \subqtab{\wt{\eta}_k}{\wt{\eta}}} U^{\wt{\eta}}_{s} \JW^{\wt{\eta}}=0
\end{equation*}
if $s>k$ by \eqref{eq:Upsicommute} and the defining property of $\JW^{\wt{\eta}}$. Putting this all together, if
\begin{equation*}
\psi_{\domtab{\wt{\eta}_1} \subqtab{\wt{\eta}_1}{\wt{\eta}}} U^{\wt{\eta}}_{r,s} \psi_{\subqtab{\wt{\eta}_k}{\wt{\eta}} \subqtab{\wt{\eta}_k}{\wt{\eta}}} \JW^{\wt{\eta}} \neq 0
\end{equation*}
then $r=1$ and $s=k$.

Using Corollary \ref{cor:psidoublecrosses} and Lemma \ref{lem:Uoffcenterdoublecross}, we observe that
\begin{align*}
\psi_{\domtab{\wt{\eta}_1} \subqtab{\wt{\eta}_1}{\wt{\eta}}} U^{\wt{\eta}}_{1,k} \psi_{\subqtab{\wt{\eta}_k}{\wt{\eta}} \subqtab{\wt{\eta}_k}{\wt{\eta}}} \JW^{\wt{\eta}} & =\psi_{\domtab{\wt{\eta}_1} \subqtab{\wt{\eta}_1}{\wt{\eta}}} U^{\wt{\eta}}_{1,k-2} U^{\wt{\eta}}_{k-1} U^{\wt{\eta}}_k \psi_{f_{\wt{\eta}}+(k-1)e}^2 \psi_{f_{\wt{\eta}}+(k-2)e}^2 \dotsm \psi_{f_{\wt{\eta}}+e}^2 \JW^{\wt{\eta}} \\
& = \pm \psi_{\domtab{\wt{\eta}_1} \subqtab{\wt{\eta}_1}{\wt{\eta}}} U^{\wt{\eta}}_{1,k-2} U^{\wt{\eta}}_{k-1} U^{\wt{\eta}}_k U^{\wt{\eta}}_{k-1} \psi_{f_{\wt{\eta}}+ke}^2 \psi_{f_{\wt{\eta}}+(k-2)e}^2 \psi_{f_{\wt{\eta}}+(k-3)e}^2 \dotsm \psi_{f_{\wt{\eta}}+e}^2 \JW^{\wt{\eta}} \\
& = \pm \psi_{\domtab{\wt{\eta}_1} \subqtab{\wt{\eta}_1}{\wt{\eta}}} U^{\wt{\eta}}_{1,k-1} \psi_{f_{\wt{\eta}}+(k-2)e}^2 \psi_{f_{\wt{\eta}}+(k-3)e}^2 \dotsm \psi_{f_{\wt{\eta}}+e}^2 \psi_{f_{\wt{\eta}}+ke}^2 \JW^{\wt{\eta}} \text{.}
\end{align*}
Apply this several times to obtain
\begin{equation*}
\psi_{\domtab{\wt{\eta}_1} \subqtab{\wt{\eta}_1}{\wt{\eta}}} U^{\wt{\eta}}_{1,k} \psi_{\subqtab{\wt{\eta}_k}{\wt{\eta}} \subqtab{\wt{\eta}_k}{\wt{\eta}}} \JW^{\wt{\eta}}= \pm \psi_{\domtab{\wt{\eta}_1} \subqtab{\wt{\eta}_1}{\wt{\eta}}} U^{\wt{\eta}}_1 \psi_{f_{\wt{\eta}}}^2 \psi_{f_{\wt{\eta}}+2e}^2 \psi_{f_{\wt{\eta}}+3e}^2 \dotsm \psi_{f_{\wt{\eta}}+ke}^2 \JW^{\wt{\eta}} \text{.}
\end{equation*}
Then by Lemma \ref{lem:crossU1cross} and Corollaries \ref{cor:transitivityofpsi} and \ref{cor:psidoublecrosses} this is equal to 
\begin{align*}
\pm \psi_{\domtab{\wt{\eta}_1} \subqtab{\wt{\eta}_1}{\wt{\eta}}} \psi_{f_{\wt{\eta}}+e}^2 \psi_{f_{\wt{\eta}}+2e}^2 \psi_{f_{\wt{\eta}}+3e}^2 \dotsm \psi_{f_{\wt{\eta}}+ke}^2 \JW^{\wt{\eta}}
& = \pm \psi_{f_{\wt{\eta}}+e}^2 \psi_{f_{\wt{\eta}}+2e}^2 \psi_{f_{\wt{\eta}}+3e}^2 \dotsm \psi_{f_{\wt{\eta}}+ke}^2 \psi_{\domtab{\wt{\eta}_1} \subqtab{\wt{\eta}_1}{\wt{\eta}}} \JW^{\wt{\eta}} \\
& = \pm \psi_{\subqtab{\wt{\eta}_{k+1}}{\wt{\eta}_1} \domtab{\wt{\eta}_1}} \psi_{\domtab{\wt{\eta}_1} \subqtab{\wt{\eta}_{k+1}}{\wt{\eta}_1}} \psi_{\domtab{\wt{\eta}_1} \subqtab{\wt{\eta}_1}{\wt{\eta}}} \JW^{\wt{\eta}} \\
& = \pm \psi_{\subqtab{\wt{\eta}_{k+1}}{\wt{\eta}_1} \subqtab{\wt{\eta}_{k+1}}{\wt{\eta}}}  \JW^{\wt{\eta}} \text{.}
\end{align*}
\end{proof}

Next, we show that other monomials wind up in an ideal of $B_n^{\kappa}$.

\begin{thm} \label{thm:othermonomials}
Let $U$ be a monomial in the generators of the Temperley--Lieb subalgebra. If $U \neq U^{\wt{\eta}}_{1,k}$ then 
\begin{equation*}
\psi_{\domtab{\wt{\eta}_1} \subqtab{\wt{\eta}_1}{\wt{\eta}}} U \psi_{\subqtab{\wt{\eta}_k}{\wt{\eta}} \subqtab{\wt{\eta}_k}{\wt{\eta}}} \JW^{\wt{\eta}} \in B_n^{\kappa, \domgreater \wt{\eta}_{k+1}} \text{.}
\end{equation*}
\end{thm}

\begin{proof}
Every monomial in the generators of the Temperley--Lieb subalgebra can be written in the form $U^{\wt{\eta}}_{r_1,s_1} U^{\wt{\eta}}_{r_2,s_2} \dotsm U^{\wt{\eta}}_{r_p,s_p}$ for some strictly decreasing sequences $r_1 > r_2 > \dotsb > r_p$ and $s_1 > s_2 > \dotsb > s_p$ of some length $p \geq 0$ with $r_j \leq s_j$ for all $j$. Suppose $U \neq U_{1,k}$ is a monomial of this form such that 
\begin{equation*}
\psi_{\domtab{\wt{\eta}_1} \subqtab{\wt{\eta}_1}{\wt{\eta}}} U \psi_{\subqtab{\wt{\eta}_k}{\wt{\eta}} \subqtab{\wt{\eta}_k}{\wt{\eta}}} \JW^{\wt{\eta}} \neq 0 \text{.}
\end{equation*}
First of all we must have $p \geq 1$ by Lemma \ref{lem:deg1vanishing}. 
%
Since $r_j>r_p \geq 1$ and $s_j>s_p\geq r_p \geq 1$ for all $1 \leq j<p$, we can apply \eqref{eq:Udeg1commute} to the expression above:
\begin{multline*}
\psi_{\domtab{\wt{\eta}_1} \subqtab{\wt{\eta}_1}{\wt{\eta}}} U^{\wt{\eta}}_{r_1,s_1} U^{\wt{\eta}}_{r_2,s_2}\dotsm U^{\wt{\eta}}_{r_{p-1},s_{p-1}} U^{\wt{\eta}}_{r_p,s_p} \psi_{\subqtab{\wt{\eta}_k}{\wt{\eta}} \subqtab{\wt{\eta}_k}{\wt{\eta}}} \JW^{\wt{\eta}}= \\
=U^{\wt{\eta}_1}_{r_1-1,s_1-1} U^{\wt{\eta}_1}_{r_2-1,s_2-1} \dotsm U^{\wt{\eta}_1}_{r_{p-1}-1,s_{p-1}-1}\psi_{\domtab{\wt{\eta}_1} \subqtab{\wt{\eta}_1}{\wt{\eta}}} U^{\wt{\eta}}_{r_p,s_p} \psi_{\subqtab{\wt{\eta}_k}{\wt{\eta}} \subqtab{\wt{\eta}_k}{\wt{\eta}}} \JW^{\wt{\eta}} \text{.}
\end{multline*}
Theorem~\ref{thm:U_1k-nonvanishing} then implies that $r_p=1$ and $s_p=k$. Assuming $U \neq U_{1,k}$, we must have $p>1$. 

Now suppose $s_{p-1}>k+1$. Applying Theorem~\ref{thm:U_1k-nonvanishing} again as well as \eqref{eq:Udeg1commute} and \eqref{eq:Upsicommute}, we observe that
\begin{align*}
U^{\wt{\eta}_1}_{s_{p-1}-1}\psi_{\domtab{\wt{\eta}_1} \subqtab{\wt{\eta}_1}{\wt{\eta}}} U^{\wt{\eta}}_{1,k} \psi_{\subqtab{\wt{\eta}_k}{\wt{\eta}} \subqtab{\wt{\eta}_k}{\wt{\eta}}} \JW^{\wt{\eta}} & =\pm U^{\wt{\eta}_1}_{s_{p-1}-1} \psi_{\subqtab{\wt{\eta}_{k+1}}{\wt{\eta}_1} \subqtab{\wt{\eta}_{k+1}}{\wt{\eta}}}  \JW^{\wt{\eta}} \\
& =\pm \psi_{\subqtab{\wt{\eta}_{k+1}}{\wt{\eta}_1} \subqtab{\wt{\eta}_{k+1}}{\wt{\eta}}} U^{\wt{\eta}}_{s_{p-1}}  \JW^{\wt{\eta}} \\
& =0 \text{.}
\end{align*} 
This is a factor of the previous expression, so it follows that $s_{p-1}=k+1$. 
Thus it is enough to show that
\begin{equation*}
\psi_{\domtab{\wt{\eta}_1} \subqtab{\wt{\eta}_1}{\wt{\eta}}} U^{\wt{\eta}}_{k+1} U^{\wt{\eta}}_{1,k} \psi_{\subqtab{\wt{\eta}_k}{\wt{\eta}} \subqtab{\wt{\eta}_k}{\wt{\eta}}} \JW^{\wt{\eta}}= \pm U^{\wt{\eta}_1}_k \psi_{\subqtab{\wt{\eta}_{k+1}}{\wt{\eta}_1} \subqtab{\wt{\eta}_{k+1}}{\wt{\eta}}}  \JW^{\wt{\eta}} \in B_n^{\kappa, \domgreater \wt{\eta}_{k+1}} \text{.}
\end{equation*}
Using Corollaries \ref{cor:transitivityofpsi} and \ref{cor:psidoublecrosses} this is equal to
\begin{equation*}
\pm U^{\wt{\eta}_1}_k \psi_{f_{\wt{\eta}}+ke}^2 \psi_{f_{\wt{\eta}}+(k-1)e}^2 \dotsm \psi_{f_{\wt{\eta}}+e}  \psi_{\domtab{\wt{\eta}} \subqtab{\wt{\eta}_1}{\wt{\eta}}}  \JW^{\wt{\eta}}
\end{equation*}
In the proof of Proposition \ref{prop:psicrosses} we showed that $U^{\wt{\eta}_1}_k=z_{k+1} x_{k+1}\psi_{f_{\wt{\eta}}+(k+1)e}$ for some $z_{k+1} \in B_n^{\kappa}$. Thus we obtain
\begin{equation*}
\pm U^{\wt{\eta}_1}_k \psi_{f_{\wt{\eta}}+ke}^2 \psi_{f_{\wt{\eta}}+(k-1)e}^2 \dotsm \psi_{f_{\wt{\eta}}+e}  \psi_{\domtab{\wt{\eta}} \subqtab{\wt{\eta}_1}{\wt{\eta}}}  \JW^{\wt{\eta}}\in B_n^{\kappa} \psi_{\domtab{\wt{\eta}_{k+2}} \subqtab{\wt{\eta}_{k+2}}{\wt{\eta}}} \JW^{\wt{\eta}} \leq B_n^{\kappa, \domgreater \wt{\eta}_{k+1}} \text{.}
\end{equation*}
\end{proof}

Finally we are in a position to calculate the socle.

\begin{thm} \label{thm:endwallprojsimplesocle}
We have $\soc P(\wt{\eta})=L(\wt{\eta})\langle 2m\rangle$. 
\end{thm}

\begin{proof}
Since $P(\wt{\eta})$ has a Weyl filtration and the socle of every Weyl module is $L(\wt{\eta})$, it is clear that $\soc P(\wt{\eta})$ is the direct sum of copies of $L(\wt{\eta})$. The graded decomposition numbers for singular weights (from Theorem~\ref{thm:grdecompblob-singular}) indicate that the socle can contain at most one copy of $L(\wt{\eta})\langle 2k\rangle$ for each integer $0 \leq k \leq m$ and no copies of $L(\wt{\eta})$ in odd degree. The submodule $L(\wt{\eta}) \leq \weyl(\wt{\eta}_m) \leq P(\wt{\eta})$ gives one copy of $L(\wt{\eta})$ of degree $2m$ in the socle. By Lemma \ref{lem:possible-socle}, if $\soc P(\wt{\eta})$ does contain a copy of $L(\wt{\eta})\langle 2k\rangle$ for some $k<m$, then it must be spanned by
\begin{equation*}
\JW^{\wt{\eta}} \psi_{\subqtab{\wt{\eta}_k}{\wt{\eta}} \subqtab{\wt{\eta}_k}{\wt{\eta}}} \JW^{\wt{\eta}} \text{.}
\end{equation*}
We will prove that this vector does not generate a copy of $L(\wt{\eta})$ in the socle by showing that
\begin{equation*}
\psi_{\domtab{\wt{\eta}_1} \subqtab{\wt{\eta}_1}{\wt{\eta}}} \JW^{\wt{\eta}} \psi_{\subqtab{\wt{\eta}_k}{\wt{\eta}} \subqtab{\wt{\eta}_k}{\wt{\eta}}} \JW^{\wt{\eta}} \neq 0 \text{.} \label{eq:smallsocle}
\end{equation*}

Write $\JW^{\wt{\eta}}$ as a sum of monomials. It is known that the coefficient of $U^{\wt{\eta}}_{1,k}$ in $\JW^{\wt{\eta}}$ is non-zero (see e.g.~\cite[Proposition 3.10]{frenkelkhovanov}), so we may write
\begin{equation*}
\JW^{\wt{\eta}}=cU^{\wt{\eta}}_{1,k}+\sum_{\text{monomials $U \neq U^{\wt{\eta}}_{1,k}$}} c_U U
\end{equation*}
where $c,c_U \in \field$ and $c \neq 0$. Then using Theorems \ref{thm:U_1k-nonvanishing} and \ref{thm:othermonomials} we obtain
\begin{align*}
\psi_{\domtab{\wt{\eta}_1} \subqtab{\wt{\eta}_1}{\wt{\eta}}} \JW^{\wt{\eta}} \psi_{\subqtab{\wt{\eta}_k}{\wt{\eta}} \subqtab{\wt{\eta}_k}{\wt{\eta}}} \JW^{\wt{\eta}}& =\psi_{\domtab{\wt{\eta}_1} \subqtab{\wt{\eta}_1}{\wt{\eta}}} \left(cU^{\wt{\eta}}_{1,k}+\sum_{\text{monomials $U \neq U^{\wt{\eta}}_{1,k}$}} c_U U\right) \psi_{\subqtab{\wt{\eta}_k}{\wt{\eta}} \subqtab{\wt{\eta}_k}{\wt{\eta}}} \JW^{\wt{\eta}} \\
& \in \psi_{\subqtab{\wt{\eta}_{k+1}}{\wt{\eta}_1} \subqtab{\wt{\eta}_{k+1}}{\wt{\eta}}}  \JW^{\wt{\eta}} + B_n^{\kappa,\domgreater \wt{\eta}_{k+1}} \text{.}
\end{align*}
As $\psi_{\subqtab{\wt{\eta}_{k+1}}{\wt{\eta}_1} \subqtab{\wt{\eta}_{k+1}}{\wt{\eta}}}  \JW^{\wt{\eta}} \notin B_n^{\kappa,\domgreater \wt{\eta}_{k+1}}$ we are done.
\end{proof}

Applying the globalisation functor, we see that $G\weyl(\wt{\eta})=\weyl(1^{n+1},1)$ and $GP(\wt{\eta})=P(1^{n+1},1)$. Using adjunction we see that
\begin{equation*}
\Hom_{B_{n+2}^{\kappa}}(G\weyl(\wt{\eta}),GP(\wt{\eta})) \iso \Hom_{B_n^{\kappa}}(\weyl(\wt{\eta}),FGP(\wt{\eta}))=\Hom_{B_n^{\kappa}}(L(\wt{\eta}),P(\wt{\eta}))
\end{equation*}
so $P(1^{n+1},1)$ also has a simple socle. Repeated globalisation in this manner allows us to drop our assumption on $n$ and extend our result to all singular weights. For a singular weight $\wt{\lambda} \in \Lambda(n)$, write $\wt{\lambda}_{\rm min},\wt{\lambda}_{\rm max} \in \Lambda(n)$ for the unique minimal and maximal weights respectively in the same linkage class.

\begin{cor} \label{cor:singprojsimplesocle}
Let $n$ be arbitrary, and let $\wt{\lambda} \in \Lambda(n)$ be a singular weight. 
Then we have
\begin{equation*}
\soc P(\wt{\lambda})=L(\wt{\lambda}_{\rm min})\langle 2\deg \subqtab{\wt{\lambda}_{\rm max}}{\wt{\lambda}} + \deg \subqtab{\wt{\lambda}}{\wt{\lambda}_{\rm min}} \rangle \text{.}
\end{equation*}
\end{cor}

\section{Main results}
\label{sec:mainresults}

\subsection{Regular projective modules}
\label{ss:regprojmod}

We introduce some useful weight terminology. Let $\wt{\lambda} \in \Lambda(n)$. If the linkage class of $\wt{\lambda}$ has a unique $\wt{\lambda}' \in \Lambda(n)$ which is incomparable to $\wt{\lambda}$ then we say that $\wt{\lambda}$ is \defnemph{paired}. Otherwise we call $\wt{\lambda}$ \defnemph{unpaired}. For example, every singular weight is unpaired because singular linkage classes are totally ordered. On the other hand the poset structure of a regular linkage class means that the only regular unpaired weights are either maximal (i.e.~are contained in the fundamental alcove) or possibly minimal.

\begin{lem}
Let $\wt{\lambda}=(1^{\lambda_1},1^{\lambda_2}) \in \Lambda(n)$ be a regular weight. Then $\wt{\lambda}$ is unpaired if and only if $\len(w_{\wt{\lambda}})=0$ or $|\lambda_1-\lambda_2| < 2\len(w_{\wt{\lambda}})e-n$.
\end{lem}

\begin{proof}
Suppose that $\wt{\lambda}$ is not contained in the fundamental alcove and that $\lambda_1>\lambda_2$. Let $w_{\wt{\lambda}}'$ be the unique element of $\W$ such that $\len(w_{\wt{\lambda}}')=\len(w_{\wt{\lambda}})$ but $w_{\wt{\lambda}}' \neq w_{\wt{\lambda}}$. Since $\len(w_{\wt{\lambda}}'w_{\wt{\lambda}}^{-1})=2\len(w_{\wt{\lambda}})$, the unique incomparable \emph{classical} weight in the global linkage class of $(\lambda_1-\lambda_2)$ is $(\lambda_1-\lambda_2)-2\len(w_{\wt{\lambda}})e$, which does not correspond to a weight in $\Lambda(n)$ if it is less than $-n$. The case where $\lambda_1<\lambda_2$ is similar.
\end{proof}

Generalising our singular terminology, for an arbitrary weight $\wt{\lambda} \in \Lambda(n)$ write $\wt{\lambda}_{\rm min} \in \Lambda(n)$ for \emph{some} minimal weight in the linkage class of $\wt{\lambda}$ and $\wt{\lambda}_{\rm max} \in \Lambda(n)$ for the unique maximal weight in the same linkage class. For $\wt{\lambda}$ regular it is evident that $\wt{\lambda}_{\rm max}=\wt{\lambda}_{\rm fund}$. We now can extend Corollary~\ref{cor:singprojsimplesocle} to all weights.

\begin{thm} \label{thm:allprojsocle}
Let $\wt{\lambda} \in \Lambda(n)$. We have
\begin{equation*}
\soc P(\wt{\lambda})=\begin{cases}
(L(\wt{\lambda}_{\rm min}) \oplus L(\wt{\lambda}_{\rm min}'))\langle 2\deg \subqtab{\wt{\lambda}_{\rm max}}{\wt{\lambda}} + \deg \subqtab{\wt{\lambda}}{\wt{\lambda}_{\rm min}}\rangle & \text{if $\wt{\lambda}_{\rm min}$ is paired,} \\
L(\wt{\lambda}_{\rm min})\langle 2\deg \subqtab{\wt{\lambda}_{\rm max}}{\wt{\lambda}} + \deg \subqtab{\wt{\lambda}}{\wt{\lambda}_{\rm min}}\rangle & \text{if $\wt{\lambda}_{\rm min}$ is unpaired.}
\end{cases}
\end{equation*}
\end{thm}

\begin{proof}
We prove the ungraded result first. Note that for any $\wt{\mu} \domgreater \wt{\lambda}$ in the same linkage class, 
the ungraded socle of $\weyl(\wt{\mu})$ is
\begin{equation*}
\begin{cases}
L(\wt{\lambda}_{\rm min}) \oplus L(\wt{\lambda}_{\rm min}') & \text{if $\wt{\lambda}_{\rm min}$ is paired,} \\
L(\wt{\lambda}_{\rm min}) & \text{if $\wt{\lambda}_{\rm min}$ is unpaired.}
\end{cases}
\end{equation*}
As $P(\wt{\lambda})$ is filtered by Weyl modules its socle may only contain copies of these simple modules.

Write $\wt{\lambda}=(1^{\lambda_1},1^{\lambda_2})$ and without loss of generality suppose $\lambda_1>\lambda_2$. If $\wt{\lambda}$ lies in the fundamental alcove, then $P(\wt{\lambda})=\weyl(\wt{\lambda})$ and the result follows by \cite[Theorem~9.4]{martin-woodcock}, so we will assume that $\wt{\lambda}$ does not lie in the fundamental alcove. Take $k \in \N$ minimal such that $\wt{\mu}=(1^{\lambda_1},1^{\lambda_2+k}) \in \Lambda(n+k)$ lies on a wall and let $\wt{\lambda}^{(1)}=(1^{\lambda_1},1^{\lambda_2+k-1}) \in \Lambda(n+k-1)$. There is a minimal weight $\wt{\lambda}^{(1)}_{\rm min} \in \Lambda(n+k-1)$ in the linkage class of $\wt{\lambda}^{(1)}$ whose classical weight is only $1$ away from $\wt{\mu}_{\rm min}$. We observe that 
\begin{align*}
\pr_{\wt{\mu}}(\ind \weyl(\wt{\lambda}^{(1)}_{\rm min}))& =\weyl(\wt{\mu}_{\rm min}) \text{,} \\ 
\pr_{\wt{\mu}}(\ind \weyl((\wt{\lambda}^{(1)}_{\rm min})')) & =\weyl((\wt{\mu}_{\rm min})_1) \text{,}
\end{align*}
and
\begin{equation*}
\restr P(\wt{\mu})\iso F(\ind P(\wt{\mu}))=FP(1^{\lambda_1+1},1^{\lambda_2+k})=P(\wt{\lambda}^{(1)})
\end{equation*}
using the tower of recollement structure on $B_n^{\kappa}$. Thus
\begin{align*}
\dim \Hom_{B_{n+k-1}^{\kappa}}(\weyl(\wt{\lambda}^{(1)}_{\rm min}),P(\wt{\lambda}^{(1)}))& =\dim \Hom{B_{n+k}^{\kappa}}(\weyl(\wt{\mu}_{\rm min}),P(\wt{\mu}))=1 \text{,} \\
\dim \Hom_{B_{n+k-1}^{\kappa}}(\weyl((\wt{\lambda}^{(1)}_{\rm min})'),P(\wt{\lambda}^{(1)}))& =\dim \Hom_{B_{n+k}^{\kappa}}(\weyl((\wt{\mu}_{\rm min})_1),P(\wt{\mu}))=1
\end{align*}
by Corollary~\ref{cor:singprojsimplesocle}. This establishes the result for $\wt{\lambda}^{(1)}$.

If $k=1$, then we are done as $\wt{\lambda}=\wt{\lambda}^{(1)}$. Otherwise let $\wt{\lambda}^{(2)}=(1^{\lambda_1},1^{\lambda_2+k-2}) \in \Lambda(n+k-2)$. Again, there is at least one minimal weight $\wt{\lambda}^{(2)}_{\rm min}$ in the linkage class of $\wt{\lambda}^{(2)}$ whose classical weight is $1$ away from $\wt{\lambda}^{(1)}_{\rm min}$ or $(\wt{\lambda}^{(1)}_{\rm min})'$, and if this weight is also paired then there is another minimal weight $(\wt{\lambda}^{(2)}_{\rm min})'$. It is clear that
$\pr_{\wt{\lambda}^{(1)}}(\ind \weyl(\wt{\lambda}^{(2)}_{\rm min}))$ (and $\pr_{\wt{\lambda}^{(1)}}(\ind \weyl((\wt{\lambda}^{(2)}_{\rm min})'))$ if it exists) is a minimal weight Weyl module. We also have
\begin{align*}
\pr_{\wt{\lambda}^{(2)}}(\restr P(\wt{\lambda}^{(1)}))& \iso \pr_{\wt{\lambda}^{(2)}}(F(\ind P(\wt{\lambda}^{(1)}))) \\
& \iso F(\pr_{\wt{\lambda}^{(2)}}(\ind P(\wt{\lambda}^{(1)}))) \\
& =F(P(1^{\lambda_1+1},1^{\lambda_1+k-1})) \\
& =P(\wt{\lambda}^{(2)}) \text{.}
\end{align*}
Thus $\dim \Hom_{B_{n+k-2}^{\kappa}}(\weyl(\wt{\lambda}^{(2)}_{\rm min}),P(\wt{\lambda}^{(2)}))=1$ (and similarly for $(\wt{\lambda}^{(2)}_{\rm min})'$ if it exists) and the result holds for $\wt{\lambda}^{(2)}$. Continuing in this fashion, we obtain the ungraded result for $\wt{\lambda}^{(k)}=\wt{\lambda}$. The correct grade shift is apparent from the graded decomposition numbers of $B_n^{\kappa}$ (Theorem~\ref{thm:grdecompblob}).
\end{proof}

\subsection{Tilting modules}

We are finally in a position to present the main results of this paper.

\begin{thm} \label{thm:fundtilt}
Let $\wt{\lambda} \in \Lambda(n)$ be a maximal weight.
\begin{enumerate}[label={\upshape (\roman*)}]
\item If $\wt{\lambda}_{\rm min}$ is unpaired, then
  $T(\wt{\lambda})=P(\wt{\lambda}_{\rm min})\langle -\deg \subqtab{\wt{\lambda}}{\wt{\lambda}_{\rm min}} \rangle$.
\item If $\wt{\lambda}_{\rm min}$ is paired, then $T(\wt{\lambda})$ is
  the unique non-split extension
\begin{equation*}
0 \to P(\wt{\lambda}_{\rm min})\langle -\deg \subqtab{\wt{\lambda}}{\wt{\lambda}_{\rm min}} \rangle \to  T(\wt{\lambda}) \to
\weyl(\wt{\lambda}_{\rm min}')\langle -\deg \subqtab{\wt{\lambda}}{\wt{\lambda}_{\rm min}} \rangle \to  0 \text{.}
\end{equation*}
\end{enumerate}
\end{thm}

\begin{proof}
As in the previous theorem we prove the ungraded form of the result first. 
For the first claim, if $\wt{\lambda}_{\rm min}$ is unpaired then $\soc
P(\wt{\lambda}_{\rm min})=L(\wt{\lambda}_{\rm min})$ by
Theorem~\ref{thm:allprojsocle}. Thus $P(\wt{\lambda}_{\rm min})$
embeds inside $I(\wt{\lambda}_{\rm min})$. But both modules have the
same character, so we must in fact have $P(\wt{\lambda}_{\rm
  min})=I(\wt{\lambda}_{\rm min})$ is self-dual and therefore is an
indecomposable tilting module. By weight considerations it must be
a grade shift of $T(\wt{\lambda})$, which we reverse using the 
singular graded decomposition numbers.

For the second claim, we induct on $n$. 
Assume that the indecomposable tilting module 
in $B_m^{\kappa}$ with the same classical weight
has the structure above for all
$m<n$. By stability of tilting multiplicities this implies that in
$B_n^{\kappa}$ we have 
\begin{equation*}
(T(\wt{\lambda}):\weyl(\wt{\mu}))=1=(P(\wt{\lambda}_{\rm min}):\weyl(\wt{\mu}))
\end{equation*}
whenever $\wt{\mu} \neq \wt{\lambda}_{\rm min},\wt{\lambda}_{\rm
  min}'$. By \cite[Lemma~A4.1]{donkbk} and its proof
$P(\wt{\lambda}_{\rm min})$ embeds inside $T(\wt{\lambda})$ and 
\begin{align*}
(T(\wt{\lambda}):\weyl(\wt{\lambda}_{\rm min}'))& =\dim
\Ext^1_{B_n^{\kappa}}(\weyl(\wt{\lambda}_{\rm min}'),P(\wt{\lambda}_{\rm min}))
\text{,} \\
(T(\wt{\lambda}):\weyl(\wt{\lambda}_{\rm min}))& =\dim
\Ext^1_{B_n^{\kappa}}(\weyl(\wt{\lambda}_{\rm min}),P(\wt{\lambda}_{\rm min}))+1
\text{.}
\end{align*}
We will calculate the dimension of the first $\Ext$-group; the second
calculation is similar.

Let $\Omega \weyl(\wt{\lambda}_{\rm min}')$ be the kernel of the natural map 
$P(\wt{\lambda}_{\rm min}') \rightarrow \weyl(\wt{\lambda}_{\rm min}')$.
We have a short exact sequence
\begin{equation*}
0 \to  \Omega \weyl(\wt{\lambda}_{\rm min}') \to
P(\wt{\lambda}_{\rm min}') \to  \weyl(\wt{\lambda}_{\rm min}') \to 0
\end{equation*}
which induces a long exact sequence
\newcommand{\lmin}{\wt{\lambda}_{\mathrm{min}}}
\newcommand{\ldashmin}{\wt{\lambda}_{\mathrm{min}}'}
$$
\begin{tikzpicture}[descr/.style={fill=white,inner sep=1.5pt}]
        \matrix (m) [
            matrix of math nodes,
            row sep=1.5em,
            column sep=1em,
            text height=1.5ex, text depth=0.25ex
        ]
        { 0 & \Hom_{B_n^{\kappa}}(\weyl(\ldashmin), P(\lmin)) &
              \Hom_{B_n^{\kappa}}(P(\ldashmin), P(\lmin)) & \Hom_{B_n^{\kappa}}(\Omega\weyl(\ldashmin), P(\lmin))  \\
          &   & \Ext^1_{B_n^{\kappa}}(\weyl(\ldashmin), P(\lmin)) &
            \Ext^1_{B_n^{\kappa}}(P(\ldashmin), P(\lmin))\  \smash{=0.} \\
        };

        \path[overlay,->]
        (m-1-1) edge (m-1-2)
        (m-1-2) edge (m-1-3)
        (m-1-3) edge (m-1-4)
        (m-1-4) edge[out=355,in=175] (m-2-3)
        (m-2-3) edge (m-2-4);
\end{tikzpicture}
$$
The first term has dimension $1$ by Theorem \ref{thm:allprojsocle},
while the second term has dimension $[P(\wt{\lambda}_{\rm
  min}):L(\wt{\lambda}_{\rm min}')]$. For the third term, we apply
\cite[Proposition~A3.13]{donkbk} several times to obtain
\begin{equation*}
\Hom_{B_n^{\kappa}}(\Omega \weyl(\wt{\lambda}_{\rm min}'),P(\wt{\lambda}_{\rm min}))
\iso \Hom_{B_{n-2r}^{\kappa}}(F^r(\Omega \weyl(\wt{\lambda}_{\rm min}')),F^r
P(\wt{\lambda}_{\rm min}))
\end{equation*}
where $r \in \N$ is minimal such that $F^r L(\wt{\lambda}_{\rm
  min})=F^r L(\wt{\lambda}_{\rm min}')=0$. 

Localising does not change the $\weyl$-multiplicities in $\Omega
\weyl(\wt{\lambda}_{\rm min}')$ because it has a $\weyl$-filtration
with subquotients labelled by weights larger than $\wt{\lambda}_{\rm
  min}'$. This means that $F^r(\Omega \weyl(\wt{\lambda}_{\rm min}'))$
has the same $\weyl$-multiplicities as
$T(1^{\lambda_1-r},1^{\lambda_2-r})$ by induction.  Let
$\wt{\mu}=(1^{\mu_1},1^{\mu_2}) \in \Lambda(n)$ be a weight dominating
$\wt{\lambda}_{\rm min}$ and $\wt{\lambda}_{\rm min}'$ but no other
weights, and define $\wt{\mu}' \neq \wt{\mu}$ similarly if such a
weight exists. Applying \cite[Proposition~A3.13]{donkbk} again we get
\begin{align*}
\dim \Hom_{B_{n-2r}^{\kappa}}(\weyl(1^{\wt{\mu}_1-r},1^{\mu_2-r}),F^r(\Omega \weyl(\wt{\lambda}_{\rm min}')))& =\dim \Hom_{B_{n-2r}^{\kappa}}(F^r \weyl(\wt{\mu}),F^r(\Omega \weyl(\wt{\lambda}_{\rm min}')))  \\
& =\dim \Hom_{B_n^{\kappa}}(\weyl(\wt{\mu}),\Omega \weyl(\wt{\lambda}_{\rm min}')) \\
& =1
\end{align*}
and similarly for $\wt{\mu}'$, so $\soc F^r(\Omega \weyl(\wt{\lambda}_{\rm min}'))=\soc T(1^{\lambda_1-r},1^{\lambda_2-r})$. Another application of \cite[Lemma~A4.1]{donkbk} establishes that $F^r(\Omega \weyl(\wt{\lambda}_{\rm min}'))=T(1^{\lambda_1-r},1^{\lambda_2-r})$.

On the other hand, from the short exact sequence
\begin{equation*}
0 \to \Omega \weyl(\wt{\lambda}_{\rm min}) \to
P(\wt{\lambda}_{\rm min}) \to \weyl(\wt{\lambda}_{\rm min})
\to 0
\end{equation*}
it is clear that $F^r P(\wt{\lambda}_{\rm min})=F^r \Omega \weyl(\wt{\lambda}_{\rm min})$. As before $F^r \Omega \weyl(\wt{\lambda}_{\rm min})=T(1^{\lambda_1-r},1^{\lambda_2-r})$. Thus
\begin{equation*}
\dim \Hom_{B_n^{\kappa}}(\Omega \weyl(\wt{\lambda}_{\rm min}'),P(\wt{\lambda}_{\rm min}))=\dim \End T(1^{\lambda_1-r},1^{\lambda_2-r}) \text{,}
\end{equation*}
which by \cite[Proposition~A2.2(ii)]{donkbk} equals the number of Weyl subquotients in $T(1^{\lambda_1-r},1^{\lambda_2-r})$. But by induction this is just $1$ less than the number of Weyl subquotients in $P(\wt{\lambda}_{\rm min})$, which is exactly $[P(\wt{\lambda}_{\rm min}):L(\wt{\lambda}_{\rm min}')]$. Thus the relevant $\Ext$-group is $1$-dimensional and the ungraded result follows. The correct grade shift follows from the regular graded decomposition numbers.
\end{proof}


To write the other tilting modules, it is useful to introduce some notation due to Donkin. For $\wt{\lambda} \in \Lambda(n)$ and $M$ a $B_n^{\kappa}$-module, write $O_{\domleq \wt{\lambda}}(M)$ for the maximal submodule of $M$ whose composition factors are all of the form $L(\wt{\mu})$ for some $\wt{\mu} \domleq \wt{\lambda}$. Dually we write $O^{\domleq \wt{\lambda}}(M)$ for the minimal submodule of $M$ such that $M/O^{\domleq \wt{\lambda}}(M)$ has composition factors of the form $L(\wt{\mu})$ for some $\wt{\mu} \domleq \wt{\lambda}$,

\begin{thm} \label{thm:resttilt}
Let $\wt{\lambda} \in \Lambda(n)$ be any weight. 
Then $T(\wt{\lambda})=O_{\domleq \wt{\lambda}}(T(\wt{\lambda}_{\rm max}))\langle -\deg \subqtab{\wt{\lambda}_{\rm max}}{\wt{\lambda}} \rangle$.
\end{thm}

\begin{proof}
First of all, it is clear that $O_{\domleq \wt{\lambda}}(T(\wt{\lambda}_{\rm max}))$ has a $\dweyl$-filtration. By \cite[Lemma~A4.5]{donkbk} $O_{\domleq \wt{\lambda}}(T(\wt{\lambda}_{\rm max}))$ is the indecomposable tilting module of highest weight $\wt{\lambda}$ in the algebra $B_n^{\kappa}(\domleq \wt{\lambda})=B_n^{\kappa}/O^{\domleq \wt{\lambda}}(B_n^{\kappa})$. Using \cite[Proposition~A3.3]{donkbk} we have
\begin{equation*}
\Ext^1_{B_n^{\kappa}}(O_{\domleq \wt{\lambda}}(T(\wt{\lambda}_{\rm max})),\dweyl(\wt{\mu}))=\Ext^1_{B_n^{\kappa}(\domleq \wt{\lambda})}(O_{\domleq \wt{\lambda}}(T(\wt{\lambda}_{\rm max})),\dweyl(\wt{\mu}))=0
\end{equation*}
for any $\wt{\mu} \domleq \wt{\lambda}$. This means that $O_{\domleq \wt{\lambda}}(T(\wt{\lambda}_{\rm max}))$ has a $\weyl$-filtration too, and thus must be a tilting module for $B_n^{\kappa}$. But the socle of $O_{\domleq \wt{\lambda}}(T(\wt{\lambda}_{\rm max}))$ is as small as possible by Theorem \ref{thm:allprojsocle}, so it must also be indecomposable, and thus $O_{\domleq \wt{\lambda}}(T(\wt{\lambda}_{\rm max}))$ is a grade shift of $T(\wt{\lambda})$, and we surmise the correct grade shift from knowledge of the graded decomposition numbers.
\end{proof}

\subsection{Tilting characters}

For $x,y \in \W$, define the (Laurent) polynomial $h^{x,y}$ by
\begin{equation*}
h^{x,y}(v)=\begin{cases}
v^{\len(x)-\len(y)} & \text{if $y \leq x$,} \\
0 & \text{otherwise.}
\end{cases}
\end{equation*}
Our use of a superscript is intentional. We mean to emphasise the fact that these are the \emph{inverse} Kazhdan--Lusztig polynomials associated to $\W$ (in the notation of \cite{soergel-KL}), which happen to coincide with ordinary Kazhdan--Lusztig polynomials in type $\tilde{A}_1$. The graded Weyl multiplicities of the regular indecomposable tilting modules are as follows.

\begin{cor} \label{cor:grtiltchar}
Let $\wt{\lambda},\wt{\mu}$ be regular weights lying in the same linkage class. Then we have 
\begin{equation*}
(T(\wt{\mu}): \weyl(\wt{\lambda}))_v=\overline{h^{w_{\wt{\lambda}},w_{\wt{\mu}}}} \text{.}
\end{equation*}
\end{cor}

There is also a singular version.

\begin{cor}
Let $\wt{\lambda}$ be a singular weight. Then we have
\begin{equation*}
(T(\wt{\lambda}_k) : \weyl(\wt{\lambda}))_v=v^{-k} \text{.}
\end{equation*}
\end{cor}

We conclude with a few remarks on possible extensions of this result.

\begin{rem} \label{rem:conclusion} \hfill
\begin{enumerate}
\item The blob algebra is the quotient of a level $2$ cyclotomic Hecke algebra. The \defnemph{generalised blob algebras} are analogous quasi-hereditary quotients of level $l$ cyclotomic Hecke algebras for integers $l>2$. These algebras have a very similar KLR presentation \cite{libedinskyplaza}. Moreover, the representation theory of the level $l$ generalised blob algebra is governed by the combinatorics of one-column $l$-multipartitions, with a linkage principle coming from the affine Weyl group of type $\tilde{A}_{l-1}$. As a result nearly all of the notation generalises to the level $l$ case easily. We conjecture that for two regular one-column $l$-multipartitions $\wt{\lambda},\wt{\mu}$, we have
\begin{equation*}
(T(\wt{\mu}) : \weyl(\wt{\lambda}))_v=\overline{h^{w_{\wt{\lambda}},w_{\wt{\mu}}}}
\end{equation*}
in the level $l$ generalised blob algebra over a field $\field$ of characteristic $0$, where $h^{x,y}$ is the inverse Kazhdan--Lusztig polynomial of type $\tilde{A}_{l-1}$.

\item Over a field $\field$ of characteristic $p>0$, the graded decomposition numbers of the blob algebra coincide with the \defnemph{$p$-Kazhdan--Lusztig polynomials} $\prescript{p}{}{h}_{y,x}$ of type $\tilde{A}_1$ \cite[Theorem~5.26]{libedinskyplaza} (see also \cite{coxgrahammartin}). We hypothesise that the graded Weyl multiplicities of the indecomposable tilting modules of the level $l$ generalised blob algebra should be given by a $p$-analogue $\prescript{p}{}{h}^{x,y}$ of inverse Kazhdan--Lusztig polynomials of type $\tilde{A}_{l-1}$. As far as we are aware, no such analogue has been constructed before. In the spherical $l=2$ case, it is reasonable to guess that the graded Weyl multiplicities of indecomposable tilting modules for $\mathrm{TL}_n(1)$ (the $n$-strand Temperley--Lieb algebra with parameter $1$ over $\field$) give a $p$-analogue $\overline{\prescript{p}{}{m}^{x,y}}$ of the inverse spherical Kazhdan--Lusztig polynomials, truncated after weight $n$. Equivalently, using the Ringel duality of $\mathrm{TL}_n(1)$ and (a quotient of) the hyperalgebra of $\mathrm{SL}_2$, we should have
\begin{equation*}
[\weyl_{\mathrm{SL}_2}(x \cdot_p 0) : L_{\mathrm{SL}_2}(y \cdot_p 0)]=\prescript{p}{}{m}^{x,y}(1) \text{,}
\end{equation*}
where $\cdot_p$ denotes the $p$-dilated dot action.
This can be extended to higher levels in the spherical case by replacing $\mathrm{SL}_2$ with $\mathrm{SL}_l$.

\item In general, $p$-Kazhdan--Lusztig polynomials are defined via Soergel bimodules over a field of characteristic $p$. The relationship between $p$-Kazhdan--Lusztig polynomials in type $\tilde{A}_1$ and graded decomposition numbers of the blob algebra is the combinatorial shadow of the `Categorical Blob vs Soergel conjecture' \cite[\S 1.8]{libedinskyplaza}. This conjecture posits an equivalence between a `blob category' (whose $\Hom$-spaces are certain idempotent truncations of the level $l$ generalised blob algebra) and the category of Soergel bimodules in type $\tilde{A}_{l-1}$. Such an equivalence, combined with our tilting character conjecture above, would imply that the inverse ($p$-)Kazhdan--Lusztig polynomials of type $\tilde{A}_{l-1}$ appear in the corresponding category of Soergel bimodules. Yet Soergel bimodules make sense for all types, so this would lead to a categorification (resp.~construction) of inverse ($p$-)Kazhdan--Lusztig polynomials in all types. The classical relationship between Kazhdan--Lusztig polynomials and inverse Kazhdan--Lusztig polynomials could then be reinterpreted as saying something about a form of `Ringel duality' for Soergel bimodules.
\end{enumerate}
\end{rem}

\end{document}

%% file: intro.tex
%
%
%
%
%
%

The blob algebra is an extension of the ordinary Temperley--Lieb algebra 
introduced by the second author and Saleur in \cite{martin-saleur}.
It can be 
thought of as the Temperley--Lieb algebra of type $B$,  as it is a 
quotient of the type $B$ Hecke algebra in much the same way
as the ordinary Temperley--Lieb algebra is a quotient of the Hecke algebra of type $A$.
Originally motivated by the need to control
lattice boundary conditions in 
lattice models in 
statistical mechanics,
the blob algebra and its generalizations remain an active 
topic of research in both physics (e.g.~\cite{gainutdinov3,gainutdinov,gainutdinov2}) 
and representation theory 
(e.g.~\cite{plaza-grdecompblob,plaza-ryom-hansen,bowmancoxspeyer}).

\medskip

Like the ordinary Temperley--Lieb algebra, the representation theory of 
the blob 
algebra is controlled by the values of its parameters. Generically the 
blob algebra is semisimple, with certain integral
representations $\weyl(\wt{\lambda})$ called \defnemph{Weyl modules} 
giving a complete set of simple modules.
Yet for some critical parameter values, the blob algebra 
is only quasi-hereditary, and the Weyl modules are no longer simple.
In this paper we focus on the \defnemph{doubly critical} case, when the
representation theory is the most interesting (e.g.~with blocks of 
arbitrary size, 
no known quiver-and-relations presentation, etc.).
In this case, the block structure is controlled by a linkage principle 
in terms of an affine Weyl group $\W$ of type $\tilde{A}_1$.

Recall that a \defnemph{tilting module} for a quasi-hereditary algebra
is a representation with a filtration by Weyl modules as well as a
filtration by dual Weyl modules. For each weight $\wt{\lambda}$, there
is an indecomposable tilting module $T(\wt{\lambda})$ of highest weight
$\wt{\lambda}$, and all indecomposable tilting modules are of this form.
Our main result in this paper is a construction of $T(\wt{\lambda})$ 
for the doubly critical blob algebra $B_n^{\kappa}$ over 
a field of characteristic $0$. 
The construction closely depends on the 
quasi-hereditary partial order $\domleq$ on weights, defined in 
\S \ref{sec:wtsbiparts}. 
The $\W$-orbit of $\wt{\lambda}$ 
has one maximal weight $\wt{\lambda}_{\rm max}$ 
and at most two minimal weights with respect to $\domleq$.
We write $L(\wt{\lambda})$
for the simple head of $\weyl(\wt{\lambda})$,
$P(\wt{\lambda})$ for
the projective indecomposable cover of $L(\wt{\lambda})$, and 
$O_{\domleq \wt{\lambda}}(M)$ for the maximal submodule of a module $M$ 
whose composition factors lie in 
$\{L(\wt{\mu}) : \wt{\mu \domleq \wt{\lambda}}\}$.
Using this notation, our construction is as follows (see also 
Theorems~\ref{thm:fundtilt} and \ref{thm:resttilt}). 

\begin{thm*}
Suppose $\wt{\lambda}$ is a weight for $B_n^{\kappa}$.
Let $\wt{\lambda}_{\rm min}$
be a minimal weight in the $\W$-orbit of $\wt{\lambda}$.
Then 
$T(\wt{\lambda})=O_{\domleq \wt{\lambda}}(T(\wt{\lambda}_{\rm max}))$.
The maximal highest weight tilting module
$T(\wt{\lambda}_{\rm max})$ is constructed from $P(\wt{\lambda}_{\rm min})$
as follows.

\begin{enumerate}[label={\rm (\roman*)}]
\item If $\wt{\lambda}_{\rm min}$ is the only minimal weight in the
$\W$-orbit of $\wt{\lambda}$, then 
$T(\wt{\lambda}_{\rm max})=P(\wt{\lambda}_{\rm min})$.

\item If there is another minimal weight $\wt{\lambda}_{\rm min}'$ in the 
$\W$-orbit of $\wt{\lambda}$, then 
$T(\wt{\lambda}_{\rm max})$ is the unique extension of the form
\begin{equation*}
0 \to P(\wt{\lambda}_{\rm min}) \to  T(\wt{\lambda}_{\rm max}) \to
\weyl(\wt{\lambda}_{\rm min}') \to  0 \text{.}
\end{equation*}
\end{enumerate}
\end{thm*}

For $x,y \in \W$, write
\begin{equation*}
h^{x,y}(v)=\begin{cases}
v^{\len(x)-\len(y)} & \text{if $y \leq x$,} \\
0 & \text{otherwise,}
\end{cases}
\end{equation*}
which is the inverse Kazhdan--Lusztig polynomial of type $\tilde{A}_1$.
Using the decomposition numbers for $B_n^{\kappa}$ (first
calculated in \cite{martin-woodcock}), our construction implies 
the following Weyl multiplicities 
for the regular indecomposable tilting modules 
(see also Corollary~\ref{cor:grtiltchar}). 
Here for each regular weight $\wt{\lambda}$,
let $w_{\wt{\lambda}} \in \W$ 
such that 
$w_{\wt{\lambda}}(\wt{\lambda}_{\rm max})=\wt{\lambda}$.

\begin{thm*}
Let $\wt{\lambda},\wt{\mu}$ be regular weights for $B_n^{\kappa}$. Then
\begin{equation*}
(T(\wt{\mu}) : \weyl(\wt{\lambda}))=h^{w_{\wt{\lambda}},w_{\wt{\mu}}}(1) \text{.}
\end{equation*}
\end{thm*}

See Figure~\ref{fig:intro} for an example depicting the weight and alcove labels 
used in these theorems.

\begin{figure}
    \centering
    \includegraphics{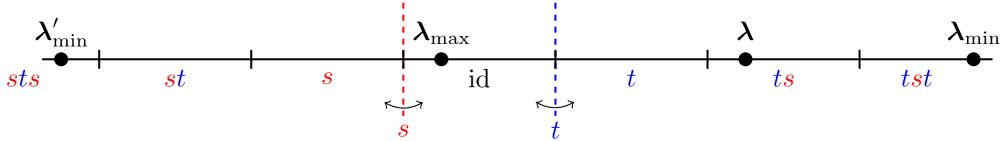}
    \caption{The (classical) weights $\wt{\lambda}_{\rm max}$, $\wt{\lambda}_{\rm min}$, and $\wt{\lambda}_{\rm min}'$, with alcoves labelled by $w_{\wt{\lambda}}$.}
    \label{fig:intro}
\end{figure}

\medskip

Our proofs depends in a crucial way not only on the decomposition
numbers and structure of the Weyl modules from \cite{martin-woodcock},
but also on the \emph{graded} representation
theory of the blob algebra. The existence of a non-trivial 
`hidden' grading on the blob algebra
is a consequence of the Brundan--Kleshchev
isomorphism \cite{brundan-kleshchev} between cyclotomic Hecke algebras
and KLR algebras, which are graded. (This explains why previous work such as 
\cite{martin-ryom-hansen,reeves} on full tilting modules did not get
very close to determining the indecomposable tilting modules.) As a bonus
we obtain the graded Weyl multiplicities of the graded 
indecomposable tilting modules with no extra work. Our
result is perhaps the first example of how the hidden grading on 
the blob algebra can be used to solve problems which a priori are
not graded at all.

We also make extensive use of KLR diagrammatics for the 
KLR presentation of the blob algebra, 
as described in \cite{libedinskyplaza}. The classical diagrammatic
calculus for the blob algebra in terms of `Temperley--Lieb diagrams 
with blobs' gives a cellular basis which is integral and multiplicative.
However, it is difficult in general to describe the simple modules
in terms of this basis. By contrast, KLR algebras have a
complicated diagram calculus reflecting the KLR presentation, in which
certain fixed parameter values are `built-in' and cannot be changed.
On the other hand, KLR diagrams give more information about the
structure of projective modules, in particular whether certain 
composition factors 
(or extensions between composition factors) are present. 
Fortunately for us, we will only need a
simplified (but still complicated) version of the KLR diagram calculus.


Much of this machinery applies, at least in principle,
to the generalised blob algebras
(cf.~e.g.~\cite{bowmancoxspeyer}, \cite{martin-woodcock2}, 
\cite{libedinskyplaza}).
For example, the level $l$ generalised blob algebras are controlled by 
an affine Weyl group $\W_l$ of type $\tilde{A}_{l-1}$, and there is a 
corresponding KLR presentation. For $\wt{\lambda}$ a regular weight
for the level $l$ generalised blob algebra
and $\wt{\lambda}_{\rm max}$ maximal in the $\W_l$-orbit of $\wt{\lambda}$,
let $w_{\wt{\lambda}} \in \W_l$ to be the unique element in the
affine Weyl group such that 
$w_{\wt{\lambda}}(\wt{\lambda}_{\rm max})=\wt{\lambda}$. For 
$x,y \in \W_l$, write
$h^{x,y}$ for the inverse Kazhdan--Lusztig polynomial of 
type $\tilde{A}_{l-1}$.
The following conjecture is the natural extension of our Weyl
multiplicities result.


\begin{conj*}
Let $\wt{\lambda},\wt{\mu}$ be weights for the level $l$ generalised 
blob algebra over a field of characteristic $0$. Then
\begin{equation*}
(T(\wt{\mu}) : \weyl(\wt{\lambda}))=h^{w_{\wt{\lambda}},w_{\wt{\mu}}}(1) \text{.}
\end{equation*}
\end{conj*}

The biggest obstacle to proving this conjecture is the lack of knowledge
about the (graded) structure of the Weyl modules and the projective modules in
higher levels. In the modular setting, it is not immediately 
obvious what should replace the inverse Kazhdan--Lusztig polynomials
$h^{x,y}$ above, although we have some ideas (see 
Remark~\ref{rem:conclusion}) based on the `Blob vs Soergel'
conjecture of Libedinsky--Plaza \cite{libedinskyplaza}.



\medskip

The layout of the paper is as follows.
In \S\ref{sec:prelim} we define the doubly critical blob algebra $B_n^{\kappa}$
using the KLR presentation and describe the corresponding weight combinatorics.
In \S\ref{sec:cellularity} we summarise the quasi-hereditary representation 
theory of $B_n^{\kappa}$.
In \S\ref{sec:bases} we exploit the KLR presentation to obtain bases for the
indecomposable projective modules and their composition factors.
In \S\ref{sec:singprojmod} we get to work with KLR diagrammatic calculations
which give the main result in the case of singular weights.
Finally in \S\ref{sec:mainresults} we use the singular version to prove
the main result for all weights.